\documentclass[12pt,a4paper,twoside]{article}

\usepackage{fancyhdr,graphicx,subfigure,psfrag}
\usepackage{amsmath,amsfonts,amsthm}
\usepackage[abs]{overpic}
\usepackage{epsfig}

 \DeclareMathOperator{\diag}{diag}

\newcommand{\D}{\mathrm{d}}
\newcommand{\p}{\partial}

\newcommand{\tr}{\mathrm{tr}}

\newtheorem{theorem}{Theorem}
\newtheorem{lemma}{Lemma}
\newtheorem{corollary}{Corollary}
\newtheorem{remark}{Remark}
\newtheorem{proposition}{Proposition}
\newtheorem{definition}{Definition}

\numberwithin{equation}{section} \numberwithin{theorem}{section}
\numberwithin{lemma}{section} \numberwithin{corollary}{section}
\numberwithin{remark}{section} \numberwithin{proposition}{section}
\numberwithin{definition}{section}

\begin{document}
\title{The rank 1 real Wishart spiked model\footnote{The author  acknowledges financial
   support by the EPSRC grant EP/G019843/1.}}
\author{M. Y. Mo}
\date{}
\maketitle
\begin{abstract}
In this paper, we consider $N$-dimensional real Wishart matrices $Y$
in the class $W_{\mathbb{R}}\left(\Sigma,M\right)$ in which all but
one eigenvalues of $\Sigma$ is $1$. Let the non-trivial eigenvalue
of $\Sigma$ be $1+\tau$, then as $N$, $M\rightarrow\infty$, with
$M/N=\gamma^2$ finite and non-zero, the eigenvalue distribution of
$Y$ will converge into the Marchenko-Pastur distribution inside a
bulk region. When $\tau$ increases from zero, one starts to see a
stray eigenvalue of $Y$ outside of the support of the
Marchenko-Pastur density. As the this stray eigenvalue leaves the
bulk region, a phase transition will occur in the largest eigenvalue
distribution of the Wishart matrix. In this paper we will compute
the asymptotics of the largest eigenvalue distribution when the
phase transition occur. We will first establish the results that are
valid for all $N$ and $M$ and will use them to carry out the
asymptotic analysis. In particular, we have derived a contour
integral formula for the Harish-Chandra Itzykson-Zuber integral
$\int_{O(N)}e^{\tr(XgYg^T)}g^T\D g$ when $X$, $Y$ are real symmetric
and $Y$ is a rank 1 matrix. This allows us to write down a Fredholm
determinant formula for the largest eigenvalue distribution and
analyze it using orthogonal polynomial techniques. As a result, we
obtain an integral formula for the largest eigenvalue distribution
in the large $N$ limit characterized by Painlev\'e transcendents.
The approach used in this paper is very different from a recent
paper \cite{BB}, in which the largest eigenvalue distribution was
obtained using stochastic operator method. In particular, the
Painlev\'e formula for the largest eigenvalue distribution obtained
in this paper is new.
\end{abstract}
\section{Introduction}
Let $X$ be an $N\times M$ (throughout the paper, we will assume $M>
N$ and $N$ is even) matrix such that each column of $X$ is an
independent, identical $N$-variate random variable with normal
distribution and zero mean. Let $\Sigma$ be its covariance matrix,
i.e. $\Sigma_{ij}=E(X_{i1}X_{j1})$. Then $\Sigma$ is an $N\times N$
positive definite symmetric matrix and we will denote its
eigenvalues by $1+\tau_j$. The matrix $Y$ defined by
$Y=\frac{1}{M}XX^T$ is a real Wishart matrix in the class
$W_{\mathbb{R}}\left(\Sigma,M\right)$ and $M$ is called the degree
of freedom of the Wishart matrix. We can think of each column of $X$
as a draw from a $N$-variate random variable with the normal
distribution and zero mean, and $Y$ the the sample covariance matrix
for the samples represented by $X$. Real Wishart matrices are good
models of sample covariance matrices in many situations and have
applications in many areas such as finance \cite{JP}, \cite{Hd},
genetic studies and climate data. (See \cite{johnstone} for
example.)

In many of these applications, one has to deal with data in which
both $N$ and $M$ are large, while the ratio $M/N$ is finite and
non-zero. In this case, the sample covariance matrix $Y$ becomes a
poor approximation for the covariance matrix $\Sigma$. However,
since it is often reasonable to approximate the data by $N$-variate
gaussian random variables, a comparison between the eigenvalue
distribution of the sample covariance matrix and Wishart matrices
with a known covariance matrix will give a good estimate for the
spectrum of the true covariance matrix $\Sigma$. In particular, in
applications to principle component analysis, one would like to
study the asymptotic behavior of the largest eigenvalue of $Y$ as
$N$, $M\rightarrow\infty$ with $M/N\rightarrow\gamma^2\geq 1$ fixed.

For many statistical data with large $N$ and $M$ and $M/N$ finite,
it was noted in \cite{johnstone} that the eigenvalue distribution of
the sample covariance matrix is well-approximated by the
Marchenko-Pastur law \cite{MP} inside a bulk region (See
\cite{Bai},\cite{BS}, \cite{BS98}, \cite{BS99}, \cite{S95}.)
\begin{equation}\label{eq:MP}
\rho(\lambda)=\frac{\gamma^2}{2\pi\lambda}\sqrt{(\lambda-b_-)(b_+-\lambda)}\chi_{[b_-,b_+]},
\end{equation}
where $\chi_{[b_-,b_+]}$ is the characteristic function for the
interval $[b_-,b_+]$ and $b_{\pm}=(1\pm \gamma^{-1})^2$. However,
outside of the bulk region, there are often a finite number of large
eigenvalues at isolated locations. This behavior prompted the
introduction of spiked models in \cite{johnstone}, which are Wishart
matrices with a covariance matrix $\Sigma$ such that only finitely
many eigenvalues of $\Sigma$ are different from one. These
non-trivial eigenvalues in $\Sigma$ will then be responsible for the
spikes that appear in the eigenvalue distribution of the sample
covariance matrix \cite{BaiS}. The number of these non-trivial
eigenvalues in $\Sigma$ is called the rank of the spiked model.

Of particular interest is a phase transition that arises in the
largest eigenvalue distributions when the first of these spikes
starts leaving the bulk region. This phenomenon was first studied in
\cite{baik04} for the complex Wishart spiked model and then in
\cite{Wang} for the rank 1 quarternionic Wishart spiked model. Let
$1+\tau$ be the non-trivial eigenvalue in $\Sigma$, in both cases,
it was shown that the phase transition in the largest eigenvalue
distribution occurs when $\tau=\gamma^{-1}$. Their results are
expressed in terms of the Hastings-McLeod solution of the Painlev\'e
II equation, which is the unique solution to the Painlev\'e II
equation
\begin{equation}\label{eq:PII}
\phi_0^{\prime\prime}(\zeta)=\zeta\phi_0+2\phi_0^3,
\end{equation}
with the following asymptotic behavior
\begin{equation}\label{eq:HM}
\begin{split}
&\phi_0\sim Ai(\zeta),\quad \zeta\rightarrow+\infty,\\
&\phi_0=\sqrt{\frac{-\zeta}{2}}\left(1+\frac{1}{8\zeta^3}+O\left(\zeta^{-6}\right)\right),\quad
\zeta\rightarrow-\infty.
\end{split}
\end{equation}
In particular, they have proved the following.
\begin{theorem}(\cite{baik04} for complex and \cite{Wang} for
quarternions) Let $W(\Sigma,M,2)=W_{\mathbb{C}}(\Sigma,M)$ and
$W(\Sigma,M,4)=W_{\mathbb{Q}}(\Sigma,M)$ be the $N$-dimensional
complex and quaternionic Wishart matrices with $M$ degrees of
freedom respectively. Suppose all but one eigenvalues of $\Sigma$
are 1 and the other eigenvalue is $1+\tau$. Then as $N$,
$M\rightarrow\infty$ with $M/N\rightarrow\gamma^2\geq 1$ finite, we
have
\begin{enumerate}
\item For $-1<\tau\leq \gamma^{-1}$, the largest eigenvalue
distribution is given by
\begin{equation*}
\lim_{M\rightarrow\infty}P\left(\left(\lambda_{max}-(1+\gamma^{-1})^2\right)\frac{\gamma
M^{\frac{2}{3}}}{(1+\gamma)^{\frac{4}{3}}}\leq
\zeta\right)=TW_{\beta}(\zeta)
\end{equation*}
where $TW_{\beta}(\zeta)$ is the Tracy-Widom distribution.
\begin{equation}\label{eq:TWdis}
\begin{split}
TW_2(\zeta)&=\exp\left(-\int_{\zeta}^{\infty}(y-\zeta)\phi_0^2(y)\D
y\right),\\
TW_4(\zeta)&=\frac{1}{2}\sqrt{TW_2(\zeta)}\left(e^{-\frac{1}{2}\int_{\zeta}^{\infty}\phi_0\D
y} +e^{\frac{1}{2}\int_{\zeta}^{\infty}\phi_0\D y}\right).
\end{split}
\end{equation}
\item For $\tau=\gamma^{-1}$,
\begin{equation*}
\lim_{M\rightarrow\infty}P\left(\left(\lambda_{max}-(1+\gamma^{-1})^2\right)\frac{\gamma
M^{\frac{2}{3}}}{(1+\gamma)^{\frac{4}{3}}}\leq
\zeta\right)=F_{\beta}(\zeta)
\end{equation*}
where $F_{\beta}(\zeta)$ is given by
\begin{equation*}
\begin{split}
F_2(\zeta)=TW_1^2(\zeta),\quad F_4(\zeta)=TW_1(\zeta),\quad
TW_1(\zeta)=\sqrt{TW_2(\zeta)}e^{-\frac{1}{2}\int_{\zeta}^{\infty}\phi_0\D
y}.
\end{split}
\end{equation*}
\item For $\tau>\gamma^{-1}$,
\begin{equation*}
\lim_{M\rightarrow\infty}P\left(\left(\lambda_{max}-(\tau+1)\left(1+\gamma^{-2}\tau^{-1}\right)\right)\frac{
\sqrt{\beta M/2}}{(1+\tau)\sqrt{1-\gamma^{-2}\tau^{-2}}}\leq
\zeta\right)=erf\left(\zeta\right)
\end{equation*}
where $erf(\zeta)$ is the error function.
\end{enumerate}
\end{theorem}
The functions $TW_{\beta}$ are called the Tracy-Widom distributions
in the literature and they give the largest eigenvalue distributions
for the real ($\beta=1$), complex ($\beta=2$) and quarternionic
($\beta=4$) Wishart ensembles with $\Sigma=I$, as well as the
largest eigenvalue distributions for a large class of random matrix
models.

Note that in \cite{baik04}, the phase transition was in fact
computed for spiked models of any finite rank. These previous
results naturally divides the range of the non-trivial eigenvalue
$1+\tau$ into 3 regimes, which are illustrated in Figure
\ref{fig:regime}.
\begin{figure}
\includegraphics[scale=0.5]{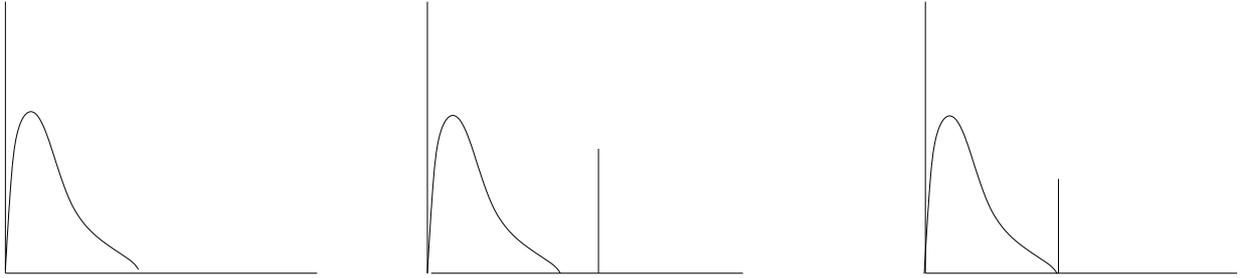}
\caption{The 3 different regimes. Left: The subcritical regime where
the perturbation in $\Sigma$ is not strong enough to form a spike in
the spectrum of $Y$. Middle: The critical regime where the
perturbation is just strong enough to form a spike at the edge of
the spectrum. Right: The super-critical regime where a spike has
left the spectrum}\label{fig:regime}
\end{figure}
\begin{enumerate}
\item The subcritical regime: When $-1<\tau<\gamma^{-1}$, the
perturbation in $\Sigma$ is not strong enough to form a spike in the
eigenvalue distribution of $Y$ and the largest eigenvalue
distribution in the Wishart matrix is not affected by this
non-trivial eigenvalue. For real Wishart ensembles, this case was
studied in \cite{FS} and it was shown that the largest eigenvalue
distribution remains the same as the case when $\Sigma=I$. The
result of \cite{FS} also applies to much more general sample
covariance matrices that are not necessarily Gaussian.
\item The critical regime: When $1-\gamma^{-1}\tau=O\left(M^{-\frac{1}{3}}\right)$, the perturbation is just strong enough to form a spike and a phase transition
occurs in the largest eigenvalue distribution. The result in
\cite{baik04} for the complex phase transition in fact extends to
the whole regime where the authors showed that the largest
eigenvalue distribution is described by a function $F_2(\zeta,w)$
which equals $TW_1^2(\zeta)$ when $w=0$. The critical regime for the
real case is the subject of this paper. It was also studied recently
in \cite{BB} using a completely different approach.
\item The super-critical regime: When $\tau>1$, the perturbation in
$\Sigma$ is strong enough to form a spike and the largest eigenvalue
is located at the spike instead of the edge of the bulk region. For
rank 1 real Wishart spiked model, this was studied in \cite{Paul}.
The results in \cite{Paul} has also been generalized to a large
class of rank 1 spiked perturbed random matrix models in
\cite{Wang2}.
\end{enumerate}
Despite having the most applications, the phase transition for real
Wishart spiked model has not been solved until very recently
\cite{BB}. The main goal of this paper is to obtain the asymptotic
largest eigenvalue distribution for the rank 1 real Wishart spiked
model in the critical regime. In a recent paper \cite{BB}, the
asymptotic largest eigenvalue distribution for the rank 1 real
Wishart ensemble was obtained by using a completely different
approach to ours. In \cite{BB}, the authors first use the
Housefolder algorithm to reduce a Wishart matrix into tridiagonal
form. Such tridiagonal matrix is then treated as a discrete random
Schr\"odinger operator and by taking an appropriate scaling limit,
the authors obtained a continuous random Schr\"odinger operator on
the half-line. By doing so, the authors in \cite{BB} bypass the
problem of determining the eigenvalue j.p.d.f. for the real Wishart
ensemble and obtain the largest eigenvalue distribution in the
asymptotic limit. In \cite{BB}, the largest eigenvalue distribution
is characterized in several different ways, one of which is the
solution of a PDE.
\begin{theorem}\label{thm:BB}(Theorem 1.7 of \cite{BB})
Let
$w=\left(\frac{N}{(1+\gamma^{-1})^2}\right)^{\frac{1}{3}}(1-\gamma\tau)
\in(-\infty,\infty)$, then the following boundary value problem
\begin{equation*}
\begin{split}
&\frac{\p F}{\p\zeta}+\frac{2}{\beta}\frac{\p^2F}{\p w^2}+(\zeta-w^2)\frac{\p F}{\p w}=0,\\
&F(\zeta,w)\rightarrow 1,\quad \textrm{$\zeta$,$w\rightarrow+\infty$ together},\\
&F(\zeta,w)\rightarrow 0,\quad w\rightarrow-\infty,\quad x<x_0<\infty
\end{split}
\end{equation*}
has a unique solution $F_{\beta}(\zeta,w)$ and
\begin{equation*}
\lim_{M\rightarrow\infty}P\left(\left(\lambda_{max}-(1+\gamma^{-1})^2\right)\frac{\gamma
M^{\frac{2}{3}}}{(1+\gamma)^{\frac{4}{3}}}\leq
\zeta\right)=F_{\beta}(\zeta,w),
\end{equation*}
where $\beta=1$, $2$ and $4$ for the real, complex and quarternionic Wishart ensembles respectively.
\end{theorem}
The distribution also has another characterization which requires
more set up to explain. We will refer the readers to \cite{BB} for
more details. The above result in \cite{BB} in fact also extends to
the spiked perturbations of Gaussian ensembles in random matrix
theory and to their general $\beta$ analogues, which are defined in
\cite{BB}.

On the other hand, the approach in this paper uses orthogonal
polynomial techniques that are closer to those in \cite{baik04} and
\cite{Wang} and we obtain a characterization of the largest
eigenvalue distribution in the critical regime in terms of
Painlev\'e transcendents. It will be very interesting see how these
two representations of the largest eigenvalue distribution can be
converted into one another. We will now state our results and
outline the method used in the paper.

\section{Statement of result}
We will now explain the approach used in this paper.

Let $\lambda_j$ be the eigenvalues of the Wishart matrix, then the
j.p.d.f. for the real Wishart ensemble is given by
\begin{equation}\label{eq:jpdf}
P(\lambda)=\frac{1}{Z_{M,N}}|\Delta(\lambda)|\prod_{j=1}^{N}\lambda_j^{\frac{M-N-1}{2}}\int_{O(N)}
e^{-\frac{M}{2}\tr(\Sigma^{-1}gY g^{-1})}g^Tdg,
\end{equation}
where $g^Tdg$ is the Haar measure on $O(N)$ and $Z_{M,N}$ is a
normalization constant. The j.p.d.f. for complex and quarternionic
Wishart ensembles are similar
\begin{equation*}
P_2(\lambda)=\frac{1}{Z_{M,N}}|\Delta(\lambda)|^2\prod_{j=1}^{N}\lambda_j^{M-N}\int_{U(N)}
e^{-M\tr(\Sigma^{-1}gY g^{-1})}g^{\dagger}dg,
\end{equation*}
for complex Wishart and
\begin{equation*}
P_4(\lambda)=\frac{1}{Z_{M,N}}|\Delta(\lambda)|^4\prod_{j=1}^{N}\lambda_j^{2(M-N)+1}\int_{Sp(N)}
e^{-2M\mathrm{Re}\left(\tr(\Sigma^{-1}gY g^{-1})\right)}g^{-1}dg,
\end{equation*}
for the quarternionic Wishart.

A major difficulty in the asymptotic analysis of the real Wishart
ensembles is to find a simple expression for the j.p.d.f. in terms
of its eigenvalues. For complex Wishart ensembles, the integral over
$U(N)$ in the expression of the j.p.d.f. can be evaluated using the
Harish-Chandra \cite{HC} Itzykson Zuber \cite{IZ} formula to obtain
a compact determinantal formula for the j.p.d.f. in terms of the
eigenvalues $\lambda_j$. This has led to the results in
\cite{baik04} and results on the largest eigenvalue distributions
for more general complex Wishart ensembles \cite{El}, \cite{On}. In
the quarternonic case, although the Harish-Chandra and Itzykson
Zuber formula does not apply, the integral over $Sp(N)$ in the
j.p.d.f. can still be evaluated using Zonal polynomial expansion to
obtain a determinantal formula for the j.p.d.f. in the rank 1 case
\cite{Wang}. In the real case, however, neither of these methods
apply and so far there has not been any simple formula for the
j.p.d.f. that allows the asymptotic analysis in the real case. Our
first result is a contour integral formula that would allow the
asymptotic analysis of the rank 1 real Wishart spiked model.
\begin{theorem}\label{thm:main1}
Assuming $N$ is even. Let the non-trivial eigenvalue in the
covariance matrix $\Sigma$ be $1+\tau$ and suppose $\tau\neq 0$.
Then the j.p.d.f. of the eigenvalues in the rank 1 real Wishart
spiked model with covariance matrix $\Sigma$ is given by
\begin{equation}\label{eq:main1}
P(\lambda)=\tilde{Z}_{M,N}^{-1}
\int_{\Gamma}|\Delta(\lambda)|e^{\frac{M\tau}{2(1+\tau)}t}\prod_{j=1}^Ne^{-\frac{M}{2}\lambda_j}\lambda_j^{\frac{M-N-1}{2}}\left(t-\lambda_j\right)^{-\frac{1}{2}}dt,
\end{equation}
where $\Gamma$ is a contour that encloses all the points
$\lambda_1,\ldots,\lambda_N$ that is oriented in the
counter-clockwise direction and $\tilde{Z}_{M,N}$ is the
normalization constant. The branch cuts of the square root
$\left(t-x\right)^{-\frac{1}{2}}$ is chosen to be the line
$\arg(t-x)=\pi$.
\end{theorem}
We will present two different proofs of this in the paper. The first
one is a geometric proof which involves choosing a suitable set of
coordinates on $O(N)$ and decompose the Haar measure into two parts
so that the integral in (\ref{eq:jpdf}) can be evaluated. This will
be achieved in Sections \ref{se:haar} and \ref{se:int}. The second
proof is an algebraic proof that uses the Zonal polynomial expansion
to verify the formula in Theorem \ref{thm:main1}. This proof will be
given in Appendix A where integral formulae of the form
(\ref{eq:main1}) for the complex and quarternionic Wishart ensembles
will also be derived.
\begin{remark}
The integral formula derived here is very similar to a more general
formula in \cite{BE}, in which the matrix integral over $O(N)$ is
given by
\begin{equation*}
\int_{O(N)}e^{-\tr\left(XgYg^{-1}\right)}g^T\D g\propto
\int\frac{e^{\tr(S)}}{\prod_{j=1}^{N}\det(S-y_jX)^{\frac{1}{2}}}\D S
\end{equation*}
where the integral of $S$ is over $\sqrt{-1}$ times the space of
$N\times N$ real symmetric matrices and $y_j$ are the eigenvalues of
$Y$. The measure $\D S$ is the flat Lebesgue measure on this space.
The proof in \cite{BE} is also similar to our geometric proof of
(\ref{eq:main1}) presented in the main text. The main difference
between the derivation in \cite{BE} and our derivation is that in
obtaining (\ref{eq:main1}), we use the assumption that $Y$ is a rank
1 matrix to decompose the Haar measure to obtain the simpler formula
(\ref{eq:main1}) in the rank 1 case.
\end{remark}
\begin{remark}
Shortly after the first part of this paper, which contains
(\ref{eq:main1}) appeared in the preprint server ArXiv, we learnt
that D. Wang has also derived (\ref{eq:main1}) independently and use
it for the asymptotic analysis in the super-critical regime
\cite{Wang2}.
\end{remark}
From the expression of the j.p.d.f., we see that the largest
eigenvalue distribution is given by
\begin{equation}\label{eq:Pmax}
\begin{split}
&\mathbb{P}(\lambda_{max}\leq z)=\int_{\lambda_1\leq\ldots\leq\lambda_N\leq z}\ldots\int P(\lambda)\D\lambda_1\ldots\D\lambda_N,\\
&=\tilde{Z}_{M,N}^{-1} \int_{\Gamma}e^{\frac{M\tau
t}{2(1+\tau)}}\int_{\lambda_1\leq\ldots\leq\lambda_N\leq
z}\ldots\int|\Delta(\lambda)|\prod_{j=1}^Nw(\lambda_j)\D\lambda_1\ldots\D\lambda_Ndt
\end{split}
\end{equation}
where $w(x)$ is
\begin{equation}\label{eq:w1}
w(x)=e^{-\frac{M}{2}x}x^{\frac{M-N-1}{2}}\left(t-x\right)^{-\frac{1}{2}}
\end{equation}
and $\Gamma$ is a close contour that encloses the interval $[0,z]$.

We can analyze the integrand as in \cite{Met}, \cite{TW}, \cite{TW1}
and \cite{TW2}. By an identity of Brujin \cite{Brujin}, we can
express the multiple integral as a Pfaffian.
\begin{equation}\label{eq:pff}
\begin{split}
&\int_{\lambda_1\leq\ldots\leq\lambda_N\leq
z}\ldots\int|\Delta(\lambda)|\prod_{j=1}^Nw(\lambda_j)\D\lambda_1\ldots\D\lambda_N\\
&=Pf\left(\left<(1-\chi_{[z,\infty)})r_j(x),(1-\chi_{[z,\infty)})r_k(y)\right>_1\right).
\end{split}
\end{equation}
where $r_j(x)$ is an arbitrary sequence of degree $j$ monic
polynomials and $\left<f,g\right>_1$ is the skew product
\begin{equation}\label{eq:skewinner}
\left<f,g\right>_1=\int_0^{\infty}\int_0^{\infty}\epsilon(x-y)f(x)g(y)w(x)w(y)\D
x\D y.
\end{equation}
where $\epsilon(x)=\frac{1}{2}\mathrm{sgn}(x)$. In defining the skew
product, the contour of integration will be defined such that if $t$
is on $(0,\infty)$, then the interval $(0,\infty)$ will be deformed
into the upper or lower half plane near $t$ such that the integral
is well defined. As $\Gamma$ will only intersect $(0,\infty)$ at a
point $x_0>z$, the left hand side of (\ref{eq:pff}), and hence the
Pfaffian, will not be affected by such deformations. Then by
following the method in \cite{TW}, \cite{TW1} and \cite{TW2}, we can
write the Pffafian as the square root of a Fredholm determinant. Let
$\mathbb{D}$ be the moment matrix with entries
$\left<r_j,r_k\right>_1$ and $\det_2$ be the regularized
2-determinant $\det_2(I+A)=\det((I+A)e^{-A})e^{\tr(A_{11}+A_{22})}$
for the $2\times 2$ matrix kernel $A$ with entries $A_{ij}$, then we
have
\begin{equation*}
Pf\left(\left<(1-\chi_{[z,\infty)})(x)r_j(x),(1-\chi_{[z,\infty)})(y)r_k(y)\right>_1\right)
=\sqrt{\det\mathbb{D}(t)}\sqrt{\mathrm{det}_2\left(I-\chi
K\chi\right)},
\end{equation*}
where $\chi=\chi_{[z,\infty)}$ and $K$ is the operator whose kernel
is given by
\begin{equation}\label{eq:matker}
K(x,y)=\begin{pmatrix}S_1(x,y)&-\frac{\p}{\p y}S_1(x,y)\\
IS_1(x,y)&S_1(y,x)\end{pmatrix}
\end{equation}
and $S_1(x,y)$ and $IS_1(x,y)$ are the kernels
\begin{equation}\label{eq:ker}
\begin{split}
S_1(x,y)&=-\sum_{j,k=0}^{N-1}r_j(x)w(x)\mu_{jk}\epsilon(r_kw)(y),\\
IS_1(x,y)&=-\sum_{j,k=0}^{N-1}\epsilon(r_jw)(x)\mu_{jk}\epsilon(r_kw)(y)
\end{split}
\end{equation}
and $\mu_{jk}$ is the inverse of the matrix $\mathbb{D}$. As shown
in \cite{W}, the kernel can now be expressed in terms of the
Christoffel Darboux kernel of some suitable orthogonal polynomials,
together with a correction term which gives rise to a finite rank
perturbation to the Christoffel Darboux kernel. In this paper, we
introduce a new proof of this using skew orthogonal polynomials and
their representations as multi-orthogonal polynomials. By using
ideas from \cite{AF} to write skew orthogonal polynomials in terms
of orthogonal polynomials, we can express the skew orthogonal
polynomials with respect to the weight $w(x)$ in terms of a sum of
Laguerre polynomials. Let $\pi_{k,1}$ be the monic skew orthogonal
polynomial of degree $k$ with respect to the weight $w(x)$.
\begin{equation}\label{eq:sop}
\left<\pi_{2k+1,1},y^j\right>_1=\left<\pi_{2k,1},y^j\right>_1=0,\quad
j=0,\ldots,2k-1.
\end{equation}
Then we can write these down in terms of Laguerre polynomials.
\begin{proposition}
Let $L_k$ be the degree $k$ monic Laguerre polynomial with respect
to the weight $w_0(x)$
\begin{equation*}
\int_0^{\infty}L_k(x)L_j(x)w_0(x)\D x=\delta_{jk}h_{j,0},\quad
w_0(x)=x^{M-N}e^{-Mx}.
\end{equation*}
If $\left<L_{2k-1},L_{2k-2}\right>_1\neq 0$, then the skew
orthogonal polynomials $\pi_{2k,1}$ and $\pi_{2k+1,1}$ both exist
and $\pi_{2k,1}$ is unique while $\pi_{2k+1,1}$ is unique up to an
addition of a multiple of $\pi_{2k,1}$. Moreover, we have
$\left<L_{2k},L_{2k-1}\right>_1=0$ and the skew orthogonal
polynomials are given by
\begin{equation*}
\begin{split}
\pi_{2k,1}&=L_{2k}-\frac{\left<L_{2k},L_{2k-2}\right>_1}{\left<L_{2k-1},L_{2k-2}\right>_1}L_{2k-1},\\
\pi_{2k+1,1}&=L_{2k+1}-\frac{\left<L_{2k+1},L_{2k-2}\right>_1}{\left<L_{2k-1},L_{2k-2}\right>_1}L_{2k-1}
+\frac{\left<L_{2k+1},L_{2k-1}\right>_1}{\left<L_{2k-1},L_{2k-2}\right>_1}L_{2k-2}+c\pi_{2k,1},
\end{split}
\end{equation*}
where $c$ is an arbitrary constant.
\end{proposition}
Next, by representing skew orthogonal polynomials as
multi-orthogonal polynomials and write them in terms of the solution
of a Riemann-Hilbert problem as in \cite{Pierce}, we can apply the
results of \cite{KDaem} and \cite{baikext} to express the kernel
$S_1(x,y)$ as a finite rank perturbation of the Christoffel Darboux
kernel of the Laguerre polynomials.
\begin{theorem}\label{thm:baik}
Suppose
$\left<L_{N-1},L_{N-2}\right>_1\left<L_{N-3},L_{N-4}\right>_1\neq
0$. Let $S_1(x,y)$ be defined by (\ref{eq:ker}) and choose the
sequence of monic polynomials $r_j(x)$ such that $r_j(x)$ are
arbitrary degree $j$ monic polynomials that are independent on $t$
and $r_j(x)=\pi_{j,1}(x)$ for $j=N-2,N-1$. Then we have
\begin{equation}\label{eq:kerform1}
\begin{split}
&S_1(x,y)-K_2(x,y)=\\
&\epsilon\left(\pi_{N+1,1}w\quad\pi_{N,1}w\right)(y)\begin{pmatrix}0&-\frac{M }{2h_{N-1,0}}\\
-\frac{M }{2h_{N-2,0}}&\frac{Mt- (M+N)}{2h_{N-1,0}}\end{pmatrix}\begin{pmatrix}L_{N-2}(x)\\
L_{N-1}(x)\end{pmatrix}w(x)
\end{split}
\end{equation}
where $K_2(x,y)$ is the kernel of the Laguerre polynomials
\begin{equation}\label{eq:k2}
K_2(x,y)=\left(\frac{y(t- y)}{x(t- x)}\right)^{\frac{1}{2}}
w_0^{\frac{1}{2}}(x)w_0^{\frac{1}{2}}(y)\frac{L_{N}(x)L_{N-1}(y)-L_N(y)L_{N-1}(x)}{h_{N-1,0}(x-y)}
\end{equation}
\end{theorem}
Note that the correction term on the right hand side of
(\ref{eq:kerform1}) is the kernel of a finite rank operator. Its
asymptotics can be computed using the known asymptotics of the
Laguerre polynomials and the method in \cite{St}, \cite{DG} and
\cite{DGKV}. The actual asymptotic analysis of this correction term,
however, is particularly tedious as one would need to compute the
asymptotics of the skew orthogonal polynomials up to the third
leading order term due to cancelations. To compute the contribution
from the determinant $\det\mathbb{D}$, we derive the following
expression for the logarithmic derivative of $\det\mathbb{D}$.
\begin{proposition}\label{pro:logder0}
Let $\mathbb{D}$ be the moment matrix with entries
$\left<r_j,r_k\right>_1$, where the sequence of monic polynomials
$r_j(x)$ is chosen such that $r_j(x)$ are arbitrary degree $j$ monic
polynomials that are independent on $t$ and $r_j(x)=\pi_{j,1}(x)$
for $j=N-2,N-1$. Then the logarithmic derivative of $\det\mathbb{D}$
with respect to $t$ is given by
\begin{equation}
\frac{\p}{\p
t}\log\det\mathbb{D}=-\int_{\mathbb{R}_+}\frac{S_1(x,x)}{t- x}\D x.
\end{equation}
\end{proposition}
This then allows us to express the largest eigenvalue distribution
$\mathbb{P}(\lambda_{max}\leq z)$ as an integral of determinant.
\begin{theorem}\label{thm:main2}
The largest eigenvalue distribution of the rank 1 real Wishart
ensemble can be written in the following integral form.
\begin{equation}\label{eq:Pmaxdet}
\begin{split}
\mathbb{P}(\lambda_{max}\leq z)= C\int_{\Gamma}\exp\left(\frac{M\tau
t}{2(1+\tau)}-\frac{1}{2}\int_{c_0}^t\int_{\mathbb{R}_+}\frac{S_1(x,x)}{s-
x}\D x\D s\right)\sqrt{\mathrm{det}_2\left(I-\chi K\chi\right)}\D t.
\end{split}
\end{equation}
for some constant $c_0$ and $K$ is the operator with kernel given by
(\ref{eq:ker}). The integration contour $\Gamma$ is a close contour
that encloses the interval $[0,z]$ in the anti-clockwise direction.
\end{theorem}
These results so far are valid for all $N$ and $M$ and are exact. As
$N$, $M\rightarrow\infty$, the integral expression
(\ref{eq:Pmaxdet}) can be used for the asymptotic analysis to obtain
the largest eigenvalue distribution when the phase transition occur.
In this paper, we will demonstrate the asymptotic analysis for the
case where $M/N\rightarrow\gamma^2=1$, but the analysis can also be
applied to the case of any finite $\gamma$.

In the asymptotic limit, the $t$ integral in (\ref{eq:Pmaxdet}) can
be computed using steepest descent analysis. In fact, we shall see
that when $t\neq b_+$, the integrand in (\ref{eq:Pmaxdet}) is of
order
\begin{equation*}
\exp\left(\frac{M\tau
t}{2(1+\tau)}-\frac{1}{2}\int_{c_0}^t\int_{\mathbb{R}_+}\frac{S_1(x,x)}{s-
x}\D x\D s\right)\sqrt{\mathrm{det}_2\left(I-\chi K\chi\right)}\sim
Ce^{\frac{M\tau t}{2(1+\tau)}-\frac{1}{2
}\int_{\mathbb{R}_+}\rho(s)\log\left(t-s\right)\D s},
\end{equation*}
where $\rho(s)$ is given in (\ref{eq:MP}). The saddle point
$t_{saddle}$ of this equation does not belong to the bulk region
$[b_-,b_+]$ unless $\tau$ is at the critical value
$\tau_c=\gamma^{-1}$. When $\tau<\tau_c$, the steepest descent
contour $\Gamma$ can be deformed such that it does not intersect
$[b_-,b_+]$. In this case, $t$ will always be of finite distance
from $[b_-,b_+]$ and the factor $(t- x)^{-\frac{1}{2}}$ in $w(x)$
will have no effect on the asymptotics of the kernel $S_1$ at the
edge point $b_+$. In this case, the kernel at the edge point will be
given by the Airy kernel.
\begin{equation*}
\lim_{N\rightarrow\infty}\frac{(1+\gamma)^{\frac{4}{3}}}{\gamma
M^{\frac{2}{3}}}S_{1}\left(x,y\right)=\frac{Ai(\xi_1)Ai^{\prime}(\xi_2)-Ai(\xi_2)Ai^{\prime}(\xi_1)}{\xi_1-\xi_2}
+\frac{1}{2}Ai(\xi_1)\int_{-\infty}^{\xi_2}Ai(s)\D s
\end{equation*}
where $\xi_1=\left(x-b_+\right)\frac{\gamma
M^{\frac{2}{3}}}{(1+\gamma)^{\frac{4}{3}}}$ and
$\xi_2=\left(y-b_+\right)\frac{\gamma
M^{\frac{2}{3}}}{(1+\gamma)^{\frac{4}{3}}}$ are finite.

In the critical case, however, the main contribution to the contour
integral (\ref{eq:Pmaxdet}) comes from a small neighborhood of
$t=b_+$ and the factor $(t- x)^{-\frac{1}{2}}$ in $w(x)$ will now
significantly alter the behavior of the kernel $S_1$ and change it
from the Airy kernel to into a more complicated kernel. This gives
rise to the phase transition and a new distribution function. Our
next result is the representation of this new distribution function
in terms of the Hastings-McLeod solution of the Painlev\'e II
equation. First let us define some functions that will appear in our
formula.

Let $\psi(u,T)=(T- u)^{\frac{1}{2}}$ and
$H_j(u,T)=Ai^{(j)}(u)\psi^{-1}(u,T)$ and let $\mathcal{S}_{10}$,
$\mathcal{S}_{11}$ be
\begin{equation*}
\begin{split}
\mathcal{S}_{10}(T)&=\int_{-\infty}^{\infty}H_1(u,s)\D
u+\int_{T}^{\infty}\frac{\left(\int_{-\infty}^{\infty}H_1(u,s)\D
u\right)^2}{\int_{-\infty}^{\infty}H_0(u,s)\D u}\D s-\frac{2}{3}T^{\frac{3}{2}},\\
\mathcal{S}_{11}(T)&= 2\int_{-\infty}^0\left(H_1\int_u^{\infty}H_2\D
v-\frac{(-u)^{\frac{1}{2}}}{2\pi(T- u)}\right)\D u\\
&+2\int_0^{\infty}H_1\int_u^{\infty}H_2\D v\D u
-\int_{-\infty}^{\infty}H_1\D u\int_{-\infty}^{\infty}H_2\D u.
\end{split}
\end{equation*}
Define the function $\mathcal{S}(T)$ to be
\begin{equation}\label{eq:St}
\mathcal{S}(T)=\mathcal{S}_{10}(T)+\int_{0}^T\mathcal{S}_{11}(s)\D
T-1/2\int_{0}^{\infty}\left(\mathcal{S}_{11,+}(s)-\mathcal{S}_{11,-}(s)\right)\D
s,
\end{equation}
where the contour of integration in the first term remains in the
upper half plane and $\mathcal{S}_{11,\pm}$ are the boundary values
of $\mathcal{S}_{11}$ as it approaches the real axis in the
upper/lower half plane.

Now let $U$ be the matrix
\begin{equation*}
U=\begin{pmatrix}0&0&0&-\psi\phi_0\\
0&0&\psi^{-1}\phi_0&-\psi^{-1}\phi_0\\
0&-\frac{ \sigma}{\phi_0\psi}&\frac{\p}{\p\zeta}\log\left(\phi_0/\psi\right)&0\\
-\phi_0\psi^{-1}&0&\frac{1}{2\psi^2}&0
\end{pmatrix}
\end{equation*}
and define $\vec{h}_j$ to be the vector
$\vec{h}_j=\left(0,0,0,\frac{\phi_j}{\psi}\right)^T$ for $j=0,1,2$
and $\vec{h}_j=0$ for $j=3$ and $j=4$, where $\phi_0$ is the
Hastings-McLeod solution of Painlev\'e II (\ref{eq:HM}) and
\begin{equation*}
\begin{split}
\sigma(\zeta)&=\int_{\zeta}^{\infty}\phi_0^2\D\xi,\quad
\phi_1=\phi_0^{\prime}+\sigma\phi_0^{\prime},\\
\phi_2&=\left(\zeta+\int_{\zeta}^{\infty}\phi_0(\xi)\phi_1(\xi)\D\xi\right)\phi_0-\sigma\phi_1.
\end{split}
\end{equation*}
Let $\vec{v}_j$ be the vector that satisfies the linear system of
ODEs with the following boundary condition
\begin{equation}\label{eq:vecv}
\begin{split}
&\frac{\p\vec{v}_j}{\p\zeta}=U(\zeta)\vec{v}_j+\vec{h}_j,\quad
\vec{v}_j\sim\left(0,0,0,\int_{-\infty}^{\infty}H_j\D u\right),\quad
\zeta\rightarrow+\infty,\quad j=0,1,2,\\
&\vec{v}_3\sim\left(0,0,-1,0\right),\quad \zeta\rightarrow-\infty,
\quad \vec{v}_4\sim\left(0,0,-1,0\right),\quad
\zeta\rightarrow+\infty
\end{split}
\end{equation}
Then the largest eigenvalue distribution at the phase transition is
given by
\begin{theorem}\label{thm:phdis}
Suppose $N$ is even. Let
$w=\left(\frac{N}{4}\right)^{\frac{1}{3}}(1-\tau)\in(-\infty,\infty)$
and let $\zeta=(z-4)\left(N/4\right)^{\frac{2}{3}}$, then as
$N,M\rightarrow\infty$ such that $M/N\rightarrow\gamma^2=1$, the
largest eigenvalue distribution at the phase transition is given by
\begin{equation*}
\begin{split}
&\lim_{N\rightarrow\infty}\mathbb{P}\left(\left(\lambda_{max}-4\right)\left(\frac{N}{4}\right)^{\frac{2}{3}}\leq\zeta\right)
=C\sqrt{TW_2(\zeta)}\mathrm{Im}\Bigg(\int_{\Xi_+}e^{-\frac{wT}{2}-\frac{1}{2}\mathcal{S}}
\left(\int_{-\infty}^{\infty}H_0\D
u\right)^{\frac{1}{2}}\\
&\times\left(\det\left(\delta_{jk}-(\alpha_j,\beta_k)\right)_{1\leq
j,k\leq 3} \right)^{\frac{1}{2}}\D T\Bigg),
\end{split}
\end{equation*}
for some constant $C$, where $\Xi_+$ is a contour in the upper half
plane that does not contain any zero of
$\int_{-\infty}^{\infty}H_0\D u$ and approaches $\infty$ in the
sector $\pi/3<\arg T<\pi$. It intersects $\mathbb{R}$ at the point
$\zeta$. The entries in the $3\times 3$ matrix are given by
\begin{equation*}
\begin{split}
\left(\alpha_1,\beta_1\right)&=\frac{v_{34}-v_{33}+1}{2},\quad
\left(\alpha_1,\beta_2\right)=\frac{v_{32}}{2},\quad
\left(\alpha_1,\beta_1\right)=\frac{\psi^2W\left(v_{33},\phi_1\psi^{-1}\right)}{2 \sigma},\\
\left(\alpha_2,\beta_1\right)&=\frac{T}{2}q_0 -\mathcal{B}_1q_1- q_2
-\mathcal{B}_2\tilde{\mathcal{R}}_{-},\quad
\left(\alpha_3,\beta_1\right)=\frac{1}{2}q_1
+\mathcal{B}_1q_0\\
\left(\alpha_2,\beta_j\right)&=\frac{T}{2}v_{j-2,2}
-\mathcal{B}_1u_{-,1j-1}- u_{-,2j-2}
-\mathcal{B}_2\tilde{\mathcal{P}}_{-,j-2},\\
\left(\alpha_3,\beta_j\right)&=\frac{1}{2}u_{-,1j-2}
+\mathcal{B}_1v_{j-2,2},\quad j=2,3,
\end{split}
\end{equation*}
where $W(f,g)$ is the Wronskian $W(f,g)=fg^{\prime}-gf^{\prime}$ and
$v_{jk}$ are the components of the vectors $v_j$ in (\ref{eq:vecv}).
The functions $\mathcal{B}_1$, $\mathcal{B}_2$, $q_j$,
$\tilde{R}_-$, $\tilde{P}_{-,j}$ and $u_{-,jk}$ are given by
\begin{equation*}
\begin{split}
&q_{j}=v_{j4}-v_{j3}-\int_{-\infty}^{\infty}H_j\D
u,\quad u_{-,jk}=\frac{\psi^2W\left(v_{j3},\phi_k\psi^{-1}\right)}{ \sigma},\quad j=1,2,\\
&\tilde{\mathcal{R}}_-=v_{44}-v_{43}+1,\quad
\tilde{\mathcal{P}}_{-,0}=v_{42},\quad
\tilde{\mathcal{P}}_{-,1}=\frac{\psi^2W\left(v_{43},\phi_1\psi^{-1}\right)}{ \sigma},\\
\mathcal{B}_1&=-\frac{ \int_{-\infty}^{\infty}H_1(u)\D
u+1}{2\int_{-\infty}^{\infty}H_0(u)\D u},\\
\mathcal{B}_2&=-\frac{\mathcal{B}_1}{2}-\frac{T}{4}\int_{-\infty}^{\infty}H_0\D
u+\frac{\mathcal{B}_1}{2}\int_{-\infty}^{\infty}H_1\D
u+\frac{1}{2}\int_{-\infty}^{\infty}H_2\D u.
\end{split}
\end{equation*}
\end{theorem}
The distribution in Theorem \ref{thm:phdis} is expressed in terms of
integrals of the Airy function, together with the functions
$v_{jk}$, which are solutions of linear ODEs with known boundary
conditions. The coefficients of the ODEs satisfied by the $v_{jk}$
are given in terms of the Hastings-McLeod solutions and its
derivatives and are therefore known functions.
\begin{remark}
The distribution in Theorem \ref{thm:phdis} can also be expressed in
terms of an integral involving the solution of a Riemann-Hilbert
problem. (see Appendix B) With the recent advancement in the
numerical computation of Riemann-Hilbert problems \cite{Ol}, this
representation may be useful for the numerical computation of the
distribution in Theorem \ref{thm:phdis}.
\end{remark}
\subsection{Other results}
In obtaining the main result in Theorem \ref{thm:phdis}, we have
obtained some other results which may also be of interest to
mathematicians and physicists working in random matrix theory.
\subsubsection{Random matrix with external source}
One applications of the orthogonal polynomial approach developed in
this paper is in the studies of random matrix with a rank 1 external
source. Random matrix with external source are random matrix models
on the space of real symmetric, Hermitian or Hermitian self-dual
$N\times N$ matrices with the following probability measure
\begin{equation*}
P(Y)\D Y=\frac{1}{Z}e^{-M\tr\left(V(Y)-AY\right)},
\end{equation*}
for some real-valued function $V(x)$ such that $e^{-V(x)}$ decays
fast enough as $x\rightarrow\pm \infty$ and a real symmetric,
Hermitian or Hermitian self-dual matrix $A$. The function $V(x)$ is
usually called the potential while the matrix $A$ is known as the
external source. The real Wishart ensemble can be thought of as a
special case when $V(x)=\frac{x}{2}-\frac{M-N-1}{2M}\log x$. Random
matrices with external source are first studied by Br\'ezin and
Hikami \cite{Bre1}, \cite{Bre2} and P. Zinn-Justin \cite{ZJ1},
\cite{ZJ2} as a model of systems with both random and deterministic
parts. For Hermitian random matrices, the external source model can
be studied using multi-orthogonal polynomial and Riemann-Hilbert
techniques \cite{BK1} and there are many recent advancements in the
asymptotic analysis of Hermitian external source models with
non-Gaussian potential $V(x)$ \cite{BW}, \cite{Ber1}, \cite{Ber2},
\cite{BK2}, \cite{BK3}, \cite{BK4}, \cite{BK5}. In \cite{Wang2}, the
contour integral formula (\ref{eq:main1}) was derived independently
and was used to study real symmetric random matrix with a rank 1
external source and a large class of potential $V(x)$. By using the
linear statistics results of Johansson \cite{Johan} and a
representation for the largest eigenvalue distribution that is
different to (\ref{eq:Pmaxdet}), the largest eigenvalue distribution
was obtained in the super-critical regime. The orthogonal polynomial
approach developed in this paper can be used to extend the results
in \cite{Wang2} to the critical regime. In fact, for real random
matrix with a rank 1 external source, the expression
(\ref{eq:Pmaxdet}) for the largest eigenvalue distribution remains
valid, although the kernels $S_1$ and $K$ will have to be modified
according to the potential. For a polynomial potential $V(x)$, the
analogue of (\ref{eq:kerform1}), which expresses the kernel $S_1$ in
terms of orthogonal polynomials is well-known \cite{W}. By using the
asymptotics of orthogonal polynomials found in \cite{DKV} and
\cite{DKV2}, such representation can be used to compute the
asymptotics of the kernel $S_1$, which can then be used in
(\ref{eq:Pmaxdet}) to obtain the largest eigenvalue distribution.
\subsubsection{Orthogonal ensembles}
A large part of this paper involves the analysis of an orthogonal
ensemble with weight (\ref{eq:w1}) and in doing so, we have obtained
some new results for orthogonal ensembles.

The first of these results is the logarithmic derivative of the
partition function $Z_1=\det\mathbb{D}$ in Proposition
\ref{pro:logder0}. While the proposition is stated in the terms of
the derivative of the parameter $t$ in $w(x)$, the proof can easily
be generalized to obtain the logarithmic derivative of
$\det\mathbb{D}$ for a general orthogonal ensemble with respect to
any parameter.
\begin{equation}\label{eq:DI}
\frac{\p}{\p
t}\log\det\mathbb{D}=2\int_{\mathbb{R}_+}S_1(x,x)\p_t\log w\D
x,\quad w(x)=e^{-\frac{NV(t,x)}{2}}.
\end{equation}
For polynomial potential, the asymptotics of $S_1(x,x)$ can be
obtained through the asymptotics of orthogonal polynomials found in
\cite{DKV}, \cite{DKV2}. This could then be used to compute the
asymptotics of the partition function $Z_1=\det\mathbb{D}$. (For
unitary ensembles, this was done in \cite{BI1} for the quartic
potential $V(x)=\frac{x^4}{4}+\frac{tx^2}{2}$ using a different
differential identity) Asymptotic analysis of the partition function
is of importance in extending the universality results of orthogonal
and symplectic ensembles to general weights $w$. At the moment,
universality in the orthogonal and symplectic ensembles are proven
for a large class of potentials $V(x)$ \cite{DG}, \cite{DGKV},
\cite{St}, \cite{Sc1}, \cite{Sc2}. However, except for the quartic
case $V(x)=\frac{x^4}{4}+\frac{tx^2}{2}$, all the available results
are restricted to the case where the limiting eigenvalue
distribution is supported on a single interval. The main obstacle in
extending these results to more general potential is the computation
of
\begin{equation*}
\det Q
=\left(\frac{Z_{N,4}(e^{-NV})Z_{2N,1}(e^{-\frac{NV}{2}})}{2^{2N}N!Z_{2N,2}(e^{-NV})}\right)^2
\end{equation*}
in the limit $N\rightarrow\infty$, where $Z_{n,\beta}(w)$ is the
partition functions of the ensembles (See remark 2.4 of \cite{St}
and remark 1.5 of \cite{DG})
\begin{equation*}
Z_{n,\beta}(w) =\int\cdots\int_{\mathbb{R}^n}\prod_{1\leq j<k\leq
n}|\lambda_j-\lambda_k|^{\beta}\prod_{j=1}^nw(\lambda_j)\D\lambda_j.
\end{equation*}
In order to proof the universality in orthogonal and symplectic
ensembles using the method in \cite{DG}, \cite{DG2} and \cite{St},
one needs to show that $\lim_{N\rightarrow\infty}\det Q\neq 0$.
While the leading order asymptotics of these partition functions can
be found using the estimates in \cite{Johan}, their combined
contributions to $\det Q$ cancel in the leading order and hence
higher order terms in the asymptotics of $Z_{n,\beta}(w)$ are needed
to show that $\lim_{N\rightarrow\infty}\det Q\neq 0$. At the moment,
asymptotics for the sub-leading order terms of partition functions
are only available for $\beta=2$. By using a differential identity
for $\log Z_{N,2}$, the authors in \cite{BI1} computed the
sub-leading order terms in $Z_{N,2}$ for the potential
$V(x)=\frac{x^4}{4}+\frac{tx^2}{2}$ as $N\rightarrow\infty$. A
combination of the method in \cite{BI1} and the differential
identity (\ref{eq:DI}) may provide a way to compute the sub-leading
order terms in the partition function $Z_{N,1}$ and help extend the
universality results in orthogonal and symplectic ensembles to more
general potentials.

Another interesting observation is Corollary \ref{cor:linear}, in
which we showed that $\left<L_k,L_{k-1}\right>_1=0$ whenever $k$ is
even. This turns out to be a very useful identity in the analysis of
the phase transition when the double scaling limit
$t-4=T\left(4/N\right)^{\frac{2}{3}}$ for the ensemble with weight
(\ref{eq:w1}) has to be considered. When analyzing this double
scaling limit, the identity $\left<L_k,L_{k-1}\right>_1=0$ leads to
cancelation in the leading order terms of the kernel $S_1$. This
enables us to show that $S_1$ is of order $N^{\frac{2}{3}}$ instead
of $N^{\frac{4}{3}}$, which is essential for the scaled limit of the
determinant $\det_2\left(I-\chi K\chi\right)$ to exist. We believe
this type of identity will also be useful in the analysis of other
double scaling limits in orthogonal ensembles.

The paper is organized as follows. In Section \ref{se:Cont} we will
prove the contour integral formula (\ref{eq:main1}) and in Section
\ref{se:Skew} the Christoffel-Darboux formula (\ref{eq:kerform1})
for the kernel $S_1$ will be derived. In Section \ref{se:Der}, we
will prove the differential identity for the moment matrix
$\mathbb{D}$ in Proposition \ref{pro:logder0}. The results in these
sections are all exact and apply to all $N$ and $M$.

We will start the asymptotic analysis in Section \ref{se:asymskew}
in which the asymptotics for the kernel $S_1$, $\det{\mathbb{D}}$
and $\det_2\left(I-\chi K\chi\right)$ will be obtained. Finally, we
will express the asymptotics of the determinant $\det_2\left(I-\chi
K\chi\right)$ in terms of the Painlev\'e transcendents in Section
\ref{se:Fred}.

Throughout the paper, we shall assume that $N$ is even and that
$M-N>0$.
\section{Contour integral formula for the j.p.d.f.}\label{se:Cont}
In this section we will prove the integral formula for the j.p.d.f.
in Theorem \ref{thm:main1}.
\subsection{Haar measure on $SO(N)$}\label{se:haar} In this section,
we will find a convenient set of coordinate on $O(N)$ to evaluate
the integral
\begin{equation*}
\int_{O(N)} e^{-\frac{M}{2}\tr(\Sigma^{-1}gY g^{-1})}g^Tdg
\end{equation*}
that appears in the expression of the j.p.d.f. (\ref{eq:jpdf}). As
both $\Sigma^{-1}$ and $Y$ are symmetric matrices, they can be
diagonalized by matrices in $O(N)$. We can therefore replace both
$\Sigma^{-1}$ and $Y$ by the diagonal matrices $\Sigma_d^{-1}$ and
$\Lambda_d$.
\begin{equation*}
\begin{split}
\Sigma_d^{-1}&=\diag\left(\frac{1}{1+\tau_1},\ldots,\frac{1}{1+\tau
_N}\right),\quad
\Lambda_d=\diag\left(\lambda_1,\ldots,\lambda_N\right)
\end{split}
\end{equation*}
The group $O(N)$ has two connected components, $SO(N)$ and $O_-(N)$
that consists of orthogonal matrices that have determinant $1$ and
$-1$ respectively. Let $T$ be the matrix
\begin{equation*}
T=\begin{pmatrix}0&1&0\\
1&0&0\\
0&0&I_{N-2}\end{pmatrix},
\end{equation*}
then the left multiplication by $T$ defines an diffeomorphism from
$O_-(N)$ to $SO(N)$. In particular, we can write the integral over
$O(N)$ in (\ref{eq:jpdf}) as
\begin{equation*}
\begin{split}
I(\Sigma,\Lambda)&=\int_{O(N)}e^{-\frac{M}{2}\tr(\Sigma_d^{-1}g\Lambda_d
g^{-1})}g^Tdg,\\&=
\int_{SO(N)}e^{-\frac{M}{2}\tr(\Sigma_d^{-1}g\Lambda_d
g^{-1})}g^Tdg+\int_{O_-(N)}e^{-\frac{M}{2}\tr(\Sigma_d^{-1}g\Lambda_d
g^{-1})}g^Tdg\\
&=\int_{SO(N)}e^{-\frac{M}{2}\tr(\Sigma_d^{-1}g\Lambda_d
g^{-1})}g^Tdg+\int_{SO(N)}e^{-\frac{M}{2}\tr(\Sigma_d^{-1}Tg\Lambda_d
g^{-1}T^{-1})}g^Tdg\\
&=\int_{SO(N)}e^{-\frac{M}{2}\tr(\Sigma_d^{-1}g\Lambda_d
g^{-1})}g^Tdg+\int_{SO(N)}e^{-\frac{M}{2}\tr(\tilde{\Sigma}_d^{-1}g\Lambda_d
g^{-1})}g^Tdg,
\end{split}
\end{equation*}
where $\tilde{\Sigma}_d$ is the diagonal matrix with the first two
entries of $\Sigma_d$ swapped.
\begin{equation*}
\tilde{\Sigma}_d^{-1}=\diag\left(\frac{1}{1+\tau _2},\frac{1}{1+\tau
_1}\ldots,\frac{1}{1+\tau _N}\right).
\end{equation*}
Note that $g^Tdg$ is also the Haar measure on $SO(N)$.

As we are considering the rank 1 spiked model, we let $\tau
_1=\ldots=\tau _{N-1}=0$ and $\tau _N=\tau $. Therefore
$\tilde{\Sigma}_d=\Sigma_d$ and we have
\begin{equation}\label{eq:I}
\begin{split}
I(\Sigma,\Lambda)
&=2\int_{SO(N)}e^{-\frac{M}{2}\tr(\Sigma_d^{-1}g\Lambda_d
g^{-1})}g^Tdg
\end{split}
\end{equation}
Let $g_{ij}$ be the entries of $g\in SO(N)$. Then the integral $I$
can be written as
\begin{equation*}
\begin{split}
I(\Sigma,\Lambda)&=2\int_{SO(N)}e^{-\frac{M}{2}\tr(\Sigma_d^{-1}g\Lambda_d
g^{-1})}g^Tdg,\\
&=2\int_{SO(N)}e^{-\frac{M}{2}\tr(\left(\Sigma_d^{-1}-I_N\right)g\Lambda_d
g^{-1})}e^{-\frac{M}{2}\tr(g\Lambda_d g^{-1})}g^Tdg,\\
&=2\prod_{j=1}^Ne^{-\frac{M}{2}\lambda_j}\int_{SO(N)}e^{\frac{\tau M}{2(1+\tau )}\sum_{j=1}^N\lambda_jg_{jN}^2}g^Tdg,\\
\end{split}
\end{equation*}
We will now find an expression of the Haar measure and use it to
compute the integral $I(\Sigma,\Lambda)$.

First let us define a set of coordinates on $SO(N)$ that is
convenient for our purpose. We will then express the Haar measure on
$SO(N)$ in terms of these coordinates.

An element $g\in SO(n)$ can be written in the following form
\begin{equation*}
g=\left(\vec{g}_1,\ldots,\vec{g}_n\right),\quad |\vec{g}_i|=1,\quad
\vec{g}_i\cdot\vec{g}_j=\delta_{ij},\quad i,j=1,\ldots, n.
\end{equation*}
This represents $SO(N)$ as the set of positively oriented
orthonormal frames in $\mathbb{R}^N$ whose coordinate axis are given
by the vectors $\vec{g}_i$. As the vector $\vec{g}_N$ is a unit
vector, we can write its components as
\begin{equation}\label{eq:angles}
\begin{split}
g_{1N}&=\cos\phi_1,\quad
g_{jN}=\prod_{k=1}^{j-1}\sin\phi_k\cos\phi_j,\quad
j=2,\ldots,n-1,\\
g_{NN}&=\prod_{k=1}^{N-1}\sin\phi_k
\end{split}
\end{equation}
The remaining vectors $\vec{g}_1,\ldots,\vec{g}_{N-1}$ form an
orthonormal frame with positive orientation in a copy of
$\mathbb{R}^{N-1}$ that is orthogonal to $\vec{g}_N$. Therefore the
set of vectors $\vec{g}_1,\ldots,\vec{g}_{N-1}$ can be identified
with $SO(N-1)$. To be precise, let $\vec{u}$ be a unit vector in
$\mathbb{R}^N$ and let $G(\vec{u})\in SO(N)$ be a matrix that maps
$\vec{u}$ to the vector $\left(0,\ldots,0,1\right)^T$. As $G$ is
orthogonal, we have
\begin{equation}\label{eq:Gvecn}
G(\vec{g}_N)\vec{g}_j=\left(v_{j1},\ldots,v_{j,N-1},0\right)^T,\quad
j<N
\end{equation}
In particular, the matrix $V$ with entries $v_{ij}$ for $1\leq
i,j\leq N-1$ is in $SO(N-1)$. The following then gives a set of
coordinates on $SO(N)$.
\begin{equation}\label{eq:coord}
g=\left(\vec{g}_N,V\right).
\end{equation}
In the above equation, $\vec{g}_N$ is identified with the
coordinates $\phi_{j}$ in (\ref{eq:angles}), while the matrix $V$ is
identified with the coordinates in $SO(N-1)$ that correspond to $V$.
In terms of these coordinates, the left action of an element $S\in
SO(N)$ on $g$ is given by the following.
\begin{equation*}
\begin{split}
Sg&=\left(S\vec{g}_1,\ldots,S\vec{g}_{N-1},S\vec{g}_N\right)^T \\
&=\left(SG(\vec{g}_N)^{-1}\vec{v}_1,\ldots,SG(\vec{g}_N)^{-1}\vec{v}_{N-1},S\vec{g}_N\right)^T
\end{split}
\end{equation*}
Then as in (\ref{eq:Gvecn}), we have
\begin{equation*}
G(S\vec{g}_N)SG(\vec{g}_N)^{-1}\vec{v}_j=\left(\tilde{v}_{j1},\ldots,\tilde{v}_{j,N-1},0\right)^T.
\end{equation*}
The matrix $\tilde{V}$ with entries $\tilde{v}_{ij}$ are again in
$SO(N-1)$, therefore the matrix $G(S\vec{g}_N)SG(\vec{g}_N)^{-1}$ is
of the form
\begin{equation}\label{eq:action}
G(S\vec{g}_N)SG(\vec{g}_N)^{-1}=\begin{pmatrix}\tilde{S}_{N-1}&\vec{s}\\
                                               0&s_N\end{pmatrix}
\end{equation}
From the fact that $G(S\vec{g}_N)SG(\vec{g}_N)^{-1}$ is an
orthogonal matrix, it is easy to check that $\vec{s}=0$ and $s_N=\pm
1$. To determine $s_N$, let us consider the action of
$G(S\vec{g}_N)SG(\vec{g}_N)^{-1}$ on $(0,0,\ldots,1)^T$. We have
\begin{equation*}
G(S\vec{g}_N)SG(\vec{g}_N)^{-1}(0,0,\ldots,1)^T=G(S\vec{g}_N)S\vec{g}_N=(0,0,\ldots,1)^T
\end{equation*}
Therefore $s_N=1$ and $\tilde{S}_{N-1}$ is in $SO(N-1)$. The action
of $S$ on $g$ is therefore given by
\begin{equation}\label{eq:Saction}
Sg=\left(S\vec{g}_N,\tilde{S}_{N-1}V\right).
\end{equation}
We will now write the Haar measure on $SO(N)$ in terms of the
coordinates (\ref{eq:coord}). These coordinates give a local
diffeomorphism between $SO(N)$ and $S^{N-1}\times SO(N-1)$ as
$\vec{g}_N\in S^{N-1}$ and $V\in SO(N-1)$. Let $dX$ be a measure on
$S^{N-1}$ that is invariant under the action of $SO(N)$ and $V^TdV$
be the Haar measure on $SO(N-1)$, then the following measure
\begin{equation*}
dH=dX\wedge V^TdV,
\end{equation*}
is invariant under the left action of $SO(N)$. Let $S\in SO(N)$,
then its action on the point $(\vec{g}_N,V)$ is given by
(\ref{eq:Saction}), where $\tilde{S}_{N-1}$ depends only on the
coordinates $\phi_1,\ldots,\phi_{N-1}$. Therefore under the action
of $S$, the measure $dH$ becomes
\begin{equation}\label{eq:dehaar}
dH\rightarrow dX\wedge
V^T\tilde{S}_{N-1}^T\tilde{S}_{N-1}dV=dX\wedge V^TdV,
\end{equation}
as $dX$ is invariant under the action of $S$. Therefore if we can
find a measure on $S^{N-1}$ that is invariant under the action of
$SO(N)$, then $dX\wedge V^TdV$ will give us a left invariant measure
on $SO(N)$. Since the left invariant measure on a compact group is
also right invariant, this will give us the Haar measure on $SO(N)$.
As the metric on $S^{N-1}$ is invariant under the action of $SO(N)$,
it is clear that the volume form on $S^{N-1}$ is invariant under the
action of $SO(N)$. Let $dX$ be the volume form on $S^{N-1}$, then
from (\ref{eq:dehaar}), we see that the measure $dX\wedge V^TdV$ is
invariant under the action of $SO(N)$.
\begin{proposition}\label{pro:haar}
Let $dX$ be the volume form on $S^{N-1}$ given by
\begin{equation*}
\begin{split}
dX=\sin^{N-2}(\phi_1)\sin^{N-1}(\phi_2)\ldots\sin(\phi_{N-2})\wedge_{j=1}^{N-1}d\phi_j
\end{split}
\end{equation*}
in terms of the coordinates $\phi_1,\ldots,\phi_{N-1}$ in
(\ref{eq:angles}) and (\ref{eq:coord}), then the Haar measure on
$SO(N)$ is equal to a constant multiple of
\begin{equation*}
dH=dX\wedge V^TdV,
\end{equation*}
where $V^TdV$ is the Haar measure on $SO(N-1)$ in terms of the
coordinates (\ref{eq:coord}).
\end{proposition}
We can now compute the integral $I(\Sigma,\Lambda)$.

\subsection{Integral formula}\label{se:int} By
using the expression of the Haar measure derived in the last
section, we can now write the integral $I(\Sigma,\Lambda)$ as
\begin{equation*}
\begin{split}
I(\Sigma,\Lambda)&=2\prod_{j=1}^Ne^{-\frac{M}{2}\lambda_j}\int_{SO(N)}e^{\frac{\tau M}{2(1+\tau )}\sum_{j=1}^N\lambda_jg_{jN}^2}g^Tdg,\\
&=2\prod_{j=1}^Ne^{-\frac{M}{2}\lambda_j}\int_{SO(N-1)}V^TdV\int_{S^{N-1}}e^{\frac{\tau M}{2(1+\tau )}\sum_{j=1}^N\lambda_jg_{jN}^2}dX,\\
&=2C\prod_{j=1}^Ne^{-\frac{M}{2}\lambda_j}\int_{S^{N-1}}e^{\frac{\tau
M}{2(1+\tau )}\sum_{j=1}^N\lambda_jg_{jN}^2}dX,
\end{split}
\end{equation*}
for some constant $C$, where the $N-1$ sphere $S^{N-1}$ in the above
formula is defined by $\sum_{j=1}^Ng_{jN}^2=1$ and $dX$ is the
volume form on it. If we let $g_{jN}=x_j$, then the above can be
written as
\begin{equation}\label{eq:Iint}
\begin{split}
I(\Sigma,\Lambda)&=2C\prod_{j=1}^Ne^{-\frac{M}{2}\lambda_j}\int_{\mathbb{R}^N}e^{\frac{\tau
M}{2(1+\tau )}\sum_{j=1}^N\lambda_jx_{j}^2}
\delta\left(\sum_{j=1}^Nx_j^2-1\right)dx_1\ldots dx_N.
\end{split}
\end{equation}
This can be seen most easily by the use of polar coordinates in
$\mathbb{R}^N$, which are given by
\begin{equation*}
\begin{split}
x_{1}&=r\cos\phi_1,\quad
x_{j}=r\prod_{k=1}^{j-1}\sin\phi_k\cos\phi_j,\quad
j=2,\ldots,N-1,\\
x_{N}&=r\prod_{k=1}^{N-1}\sin\phi_k,
\end{split}
\end{equation*}
Then the volume form in $\mathbb{R}^N$ is given by
\begin{equation*}
dx_1\ldots
dx_N=r^{N-1}\sin^{N-2}\phi_1\ldots\sin\phi_{N-2}drd\phi_1\ldots
d\phi_{N-1}
\end{equation*}
Therefore in terms of polar coordinates, we have
\begin{equation*}
\begin{split}
&\int_{\mathbb{R}^N}e^{\frac{\tau M}{2(1+\tau
)}\sum_{j=1}^N\lambda_jx_{j}^2}
\delta\left(\sum_{j=1}^Nx_j^2-1\right)dx_1\ldots dx_N\\
&=\int_{0}^{\pi}d\phi_1\int_0^{2\pi}d\phi_2\ldots\int_0^{2\pi}d\phi_{N-1}
\int_{0}^{\infty}dr\delta\left(\sum_{j=1}^Nr^2-1\right)r^{N-1}\\
&\times e^{\frac{\tau M}{2(1+\tau )}\sum_{j=1}^N\lambda_jx_{j}^2}
\sin^{N-2}\phi_1\ldots\sin\phi_{N-2}\\
&=\int_{S^{N-1}}e^{\frac{\tau M}{2(1+\tau
)}\sum_{j=1}^N\lambda_jx_{j}^2} dX.
\end{split}
\end{equation*}
To compute the integral $I(\Sigma,\Lambda)$, we use a method in the
studies of random pure quantum systems (see, e.g. \cite{Maj}). The
idea is to consider the Laplace transform of the function
$I(\Sigma,\Lambda,t)$ defined by
\begin{equation*}
I(\Sigma,\Lambda,t)=2C\prod_{j=1}^Ne^{-\frac{M}{2}\lambda_j}\int_{\mathbb{R}^N}e^{\frac{\tau
M}{2(1+\tau )}\sum_{j=1}^N\lambda_jx_{j}^2}
\delta\left(\sum_{j=1}^Nx_j^2-t\right)dx_1\ldots dx_N,
\end{equation*}
then $I(\Sigma,\Lambda,1)=I(\Sigma,\Lambda)$. The Laplace transform
of $I(\Sigma,\Lambda,t)$ in the variable $t$ is given by
\begin{equation*}
\int_0^{\infty}e^{-st}I(\Sigma,\Lambda,t)dt=
2C\prod_{j=1}^Ne^{-\frac{M}{2}\lambda_j}\int_{\mathbb{R}^N}e^{\sum_{j=1}^N\left(-s+\frac{\tau
M}{2(1+\tau )}\lambda_j\right)x_{j}^2} dx_1\ldots dx_N
\end{equation*}
Then, provided $\mathrm{Re}(s)>\max_j\left(\lambda_j\right)$, the
integral can be computed explicitly to obtain
\begin{equation*}
\int_0^{\infty}e^{-st}I(\Sigma,\Lambda,t)dt=
2C\prod_{j=1}^Ne^{-\frac{M}{2}\lambda_j}\left(s-\frac{\tau
M}{2(1+\tau )}\lambda_j\right)^{-\frac{1}{2}}
\end{equation*}
Taking the inverse Laplace transform, we obtain an integral
expression for $I(\Sigma,\Lambda)$.
\begin{equation*}
I(\Sigma,\Lambda)= \frac{C}{\pi i}
\int_{\Gamma}e^s\prod_{j=1}^Ne^{-\frac{M}{2}\lambda_j}\left(s-\frac{\tau
M}{2(1+\tau )}\lambda_j\right)^{-\frac{1}{2}}ds,
\end{equation*}
where $\Gamma$ is a contour that encloses all the points $\frac{\tau
M}{2(1+\tau )}\lambda_1,\ldots,\frac{\tau M}{2(1+\tau )}\lambda_N$
that is oriented in the counter-clockwise direction. Rescaling the
variable $s$ to $s=Mt$, we obtain
\begin{equation*}
I(\Sigma,\Lambda)\propto \int_{\Gamma}e^{\frac{M\tau
t}{2(1+\tau)}}\prod_{j=1}^Ne^{-\frac{M}{2}\lambda_j}\left(t-\lambda_j\right)^{-\frac{1}{2}}dt,
\end{equation*}
This then give us an integral expression for the j.p.d.f. in Theorem
\ref{thm:main1}.

For the purpose of computing the largest eigenvalue distribution
$\mathbb{P}\left(\lambda_{max}\leq z\right)$, we can assume that the
eigenvalues are all smaller than or equal to a constant $z$.
\section{Skew orthogonal polynomials and the kernel
$S_1$}\label{se:Skew} In this section we will prove the
representation of the kernel $S_1$ in Theorem \ref{thm:baik}. We
will do so by using the multi-orthogonal polynomial representation
of skew orthogonal polynomials in \cite{V} and then apply the
Christoffel-Darboux formula for multi-orthogonal polynomials in
\cite{KDaem} to write the kernel $S_1$ as a finite sum of
multi-orthogonal polynomials. We then simplify this sum further by
using a result in \cite{baikext}. This gives a new proof to a
well-known result of Widom \cite{W}.
\subsection{Skew orthogonal polynomials}
As explain in the introduction, we need to find the skew orthogonal
polynomials with the weight (\ref{eq:w1}). Let us consider the skew
orthogonal polynomials with respect to the weight $w$ in
(\ref{eq:w1}). We shall use the ideas in \cite{AF} to express the
skew orthogonal polynomials in terms of a linear combinations of
Laguerre polynomials.

Let $Q_j(x)$ to be the degree $j+2$ polynomial
\begin{equation*}
Q_j(x)=\frac{d}{dx}\left(x^{j+1}(t- x)w(x)\right)w^{-1}(x),\quad
j\geq 0.
\end{equation*}
Then as we assume $M-N>0$, it is easy to see that
\begin{equation}\label{eq:parts}
\left<f(x),H_j(y)\right>_1=\left<f(x),x^j\right>_2,
\end{equation}
for any $f(x)$ such that $\int_0^{\infty}f(x)w(x)\D x$ is finite,
where the product $\left<\right>_2$ is defined by
\begin{equation}\label{eq:w0}
\left<f(x)g(x)\right>_2=\int_{0}^{\infty}f(x)g(x)w_0(x)dx,\quad
w_0(x)=x^{M-N}e^{-Mx}.
\end{equation}
Note that $w_0(x)$ is not the square of $w(x)$. The fact that
$w_0(x)$ is the weight for the Laguerre polynomials allows us to
express the skew orthogonal polynomials for the weight (\ref{eq:w1})
in terms of Laguerre polynomials.

In particular, this implies that the conditions (\ref{eq:sop}) is
equivalent to the following conditions
\begin{equation}\label{eq:sop2}
\begin{split}
\left<\pi_{2k,1},y^j\right>_1&=0,\quad j=0,1,\\
\left<\pi_{2k,1},y^j\right>_2&=0,\quad j=0,\ldots,2k-3.
\end{split}
\end{equation}
and the exactly same conditions for $\pi_{2k+1,1}(x)$. In
particular, the second condition implies the skew orthogonal
polynomials can be written as
\begin{equation*}
\begin{split}
\pi_{2k,1}(x)&=L_{2k}(x)+\gamma_{2k,1}L_{2k-1}(x)+\gamma_{2k,2}L_{2k-2}(x),\\
\pi_{2k+1,1}(x)&=L_{2k+1}(x)+\gamma_{2k+1,0}L_{2k}(x)+\gamma_{2k+1,1}L_{2k-1}(x)+\gamma_{2k+1,2}L_{2k-2}(x),
\end{split}
\end{equation*}
where $L_j(x)$ are the degree $j$ monic Laguerre polynomials that
are orthogonal with respect to the weight $w_0(x)$.
\begin{equation}\label{eq:laguerre}
\begin{split}
L_{n}(x)&=\frac{(-1)^ne^{Mx}x^{-M+N}}{M^n}\frac{d^n}{dx^n}\left(e^{-Mx}x^{n+M-N}\right),\\
&=x^n-\frac{(M-N+n)n}{M}x^{n-1}+O(x^{n-2}).
\end{split}
\end{equation}
The constants $\gamma_{k,j}$ are to be determined from the first
condition in (\ref{eq:sop2}). We will now show that if
$\left<L_{2k-1},L_{2k-2}\right>_1\neq 0$, then the skew orthogonal
polynomials $\pi_{2k,1}$ and $\pi_{2k+1,1}$ exist and that
$\pi_{2k,1}$ is unique. First let us show that the first condition
in (\ref{eq:sop2}) is equivalent to
\begin{equation*}
\begin{split}
\left<\pi_{2k,1},L_{2k-j}\right>_1&=0,\quad j=1,2.
\end{split}
\end{equation*}
To do this, we will first define a map $\varrho_N$ from the span of
$L_{2k-1}$ and $L_{2k-2}$ to the span of $y$ and $1$.

Let $P(x)$ be a polynomial of degree $m$. Then we can write the
polynomial $P(x)$ as
\begin{equation}\label{eq:PRq}
P(x)=\frac{d}{dx}\left(q(x)x(t- x)w(x)\right)w^{-1}(x)+R(x)
\end{equation}
where $q(x)$ is a polynomial of degree $m-2$ and $R(x)$ is a
polynomial of degree less than or equal to 1. By writing down the
system of linear equations satisfied by the coefficients of $q(x)$
and $R(x)$, we see that the polynomials $q(x)$ and $R(x)$ are
uniquely defined for any given $P(x)$. In particular, the map
$f:P(x)\mapsto R(x)$ is a well-defined linear map from the space of
polynomial to the space of polynomials of degrees less than or equal
to $1$. Let $\varrho_{k}$ be the following restriction of this map.
\begin{definition}\label{de:fn}
For any polynomial $P(x)$, let $f$ be the map that maps $P(x)$ to
$R(x)$ in (\ref{eq:PRq}). Then the map $\varrho_{k}$ is the
restriction of $f$ to the linear subspace spanned by the orthogonal
polynomials $L_{k},L_{k-1}$.
\end{definition}
We then have the following.
\begin{lemma}\label{le:linear}
If $\left<L_{k},L_{k-1}\right>_1\neq 0$, then the map $\varrho_{k}$
is invertible.
\end{lemma}
\begin{proof} Suppose there is exists non-zero constants $a_1$ and
$a_2$ such that
\begin{equation*}
a_1L_{k}+a_2L_{k-1}=\frac{d}{dx}(q(x)x(t-  x)w)w^{-1}
\end{equation*}
for some polynomial $q(x)$ of degree $k-2$, then by taking the skew
product $\left<\right>_1$ of this polynomial with $L_{k}$, we obtain
\begin{equation*}
a_2\left<L_{k-1},L_{k}\right>_1=
\left<a_1L_{k}+a_2L_{k-1},L_{k}\right>_1=\left<q(x),L_{k}\right>_2=0.
\end{equation*}
As $q(x)$ is of degree $k-2$. Since
$\left<L_{k-1},L_{k}\right>_1\neq 0$, this shows that $a_2=0$. By
taking the skew product with $L_{k-1}$, we conclude that $a_1=0$ and
hence the map $\varrho_{k}$ has a trivial kernel.
\end{proof}
In particular, we have the following.
\begin{corollary}\label{cor:linear}
If $k$ is even, then $\left<L_{k},L_{k-1}\right>_1=0$.
\end{corollary}
\begin{proof}
Let $q(x)$ be a polynomial of degree $k-2$ that satisfies the
following conditions
\begin{equation}\label{eq:qcon}
\int_{\mathbb{R}_+}\frac{d}{dx}(q(x)w_4(x))x^jw_4(x)\D x=0, \quad
j=0,\ldots, k-2,
\end{equation}
where $w_4(x)=x^{\frac{M-N+1}{2}}(t-
x)^{\frac{1}{2}}e^{-\frac{Mx}{2}}$. A non trivial polynomial $q(x)$
of degree $k-2$ that satisfies these conditions exists if and only
if the moment matrix with entries
$\int_{\mathbb{R}_+}\frac{d}{dx}(x^iw_4(x))x^jw_4(x)\D x$ has a
vanishing determinant. For even $k$, the moment matrix is of odd
dimension and anti-symmetric and hence its determinant is always
zero.

Assuming $k$ is even and let $q(x)$ be a polynomial that satisfies
(\ref{eq:qcon}). By taking the inner product $\left<\right>_2$ with
$x^j$, we see that there exists non-zero constants $a_1$ and $a_2$
such that
\begin{equation*}
a_1L_{k}+a_2L_{k-1}=\frac{d}{dx}(q(x)x(t-  x)w)w^{-1},
\end{equation*}
Therefore by Lemma \ref{le:linear}, we see that if $k$ is even, we
will have $\left<L_{k},L_{k-1}\right>_1=0$.
\end{proof}
Lemma \ref{le:linear} shows that if
$\left<L_{i},L_{i-1}\right>_1\neq 0$, then there exist two
independent polynomials $R_0(y)$ and $R_1(y)$ in the span of $y$ and
$1$ such that $R_j(y)=\varrho_{i}(L_{i-j})$. Then we have
\begin{equation*}
R_j(y)=-\frac{d}{dy}(q_j(y)y(t- y)w)w^{-1}+L_{i-j}(y),\quad j=0,1.
\end{equation*}
In particular, the skew product of $R_j(y)$ with $L_{i-l}$, $l<2$ is
given by
\begin{equation*}
\left<L_{i-l}(x),R_j(y)\right>_1=-\left<L_{i-l},q_j\right>_2+\left<L_{i-l},L_{i-j}\right>_1.
\end{equation*}
As $q_j$ is a polynomial of degree less than or equal to $i-2$ and
$l<2$, the first term on the right hand side is zero. Therefore we
have
\begin{equation}\label{eq:prod}
\left<L_{i-l}(x),R_j(y)\right>_1=\left<L_{i-l},L_{i-j}\right>_1,\quad
l< 2,\quad j=0,1.
\end{equation}
We can now show that the skew orthogonal polynomials $\pi_{2k,1}$
and $\pi_{2k+1,1}$ exist if $\left<L_{2k-1},L_{2k-2}\right>_1\neq
0$.
\begin{proposition}\label{pro:p2N}
If $\left<L_{2k-1},L_{2k-2}\right>_1\neq 0$, then the skew
orthogonal polynomials $\pi_{2k,1}$ and $\pi_{2k+1,1}$ both exist
and $\pi_{2k,1}$ is unique while $\pi_{2k+1,1}$ is unique up to an
addition of a multiple of $\pi_{2k,1}$. Moreover, we have
$\left<L_{2k},L_{2k-1}\right>_1=0$ and the skew orthogonal
polynomials are given by
\begin{equation}\label{eq:p2N}
\begin{split}
\pi_{2k,1}&=L_{2k}-\frac{\left<L_{2k},L_{2k-2}\right>_1}{\left<L_{2k-1},L_{2k-2}\right>_1}L_{2k-1},\\
\pi_{2k+1,1}&=L_{2k+1}-\frac{\left<L_{2k+1},L_{2k-2}\right>_1}{\left<L_{2k-1},L_{2k-2}\right>_1}L_{2k-1}
+\frac{\left<L_{2k+1},L_{2k-1}\right>_1}{\left<L_{2k-1},L_{2k-2}\right>_1}L_{2k-2}+c\pi_{2k,1},
\end{split}
\end{equation}
for $k\geq 2$, where $c$ is an arbitrary constant.
\end{proposition}
\begin{proof} Let $\pi_{2k,1}$ and $\pi_{2k+1,1}$ be polynomials
defined by
\begin{equation*}
\begin{split}
\pi_{2k,1}(x)&=L_{2k}(x)+\gamma_{2k,1}L_{2k-1}(x)+\gamma_{2k,2}L_{2k-2}(x),\\
\pi_{2k+1,1}(x)&=L_{2k+1}(x)+\gamma_{2k+1,1}L_{2k-1}(x)+\gamma_{2k+1,2}L_{2k-2}(x),
\end{split}
\end{equation*}
for some constants $\gamma_{j,k}$. If we can show that
$\left<\pi_{2k-l,1},y^j\right>_1=0$ for $j=0,1$ and $l=-1,0$, then
$\pi_{2k-l,1}$ will be the skew orthogonal polynomial. Let $R_0$ and
$R_1$ be the images of $L_{2k-1}$ and $L_{2k-2}$ under the map
$\varrho_{2k-1}$. Then by the assumption in the Proposition, they
are independent in the span of $y$ and $1$. Therefore the conditions
$\left<\pi_{2k-l,1},y^j\right>_1=0$ are equivalent to
$\left<\pi_{2k-l,1},R_j(y)\right>_1=0$. By taking $i=2k-1$ in
(\ref{eq:prod}), we see that this is equivalent to
$\left<\pi_{2k-l,1},L_{2k-1-j}\right>_1=0$. This implies
\begin{equation*}
\begin{split}
\left<\pi_{2k,1},L_{2k-1}\right>_1&=\left<L_{2k},L_{2k-1}\right>_1+\gamma_{2k,2}\left<L_{2k-2},L_{2k-1}\right>_1=0,\\
\left<\pi_{2k,1},L_{2k-2}\right>_1&=\left<L_{2k},L_{2k-2}\right>_1+\gamma_{2k,1}\left<L_{2k-1},L_{2k-2}\right>_1=0.
\end{split}
\end{equation*}
Hence we have
\begin{equation*}
\gamma_{2k,1}=-\frac{\left<L_{2k},L_{2k-2}\right>_1}{\left<L_{2k-1},L_{2k-2}\right>_1},\quad
\gamma_{2k,2}=\frac{\left<L_{2k},L_{2k-1}\right>_1}{\left<L_{2k-1},L_{2k-2}\right>_1},
\end{equation*}
which exist and are unique as $\left<L_{2k-1},L_{2k-2}\right>_1\neq
0$. This determines $\pi_{2k,1}$ uniquely. By Corollary
\ref{cor:linear}, we have $\left<L_{2k},L_{2k-1}\right>_1=0$ and
hence $\gamma_{2k,2}=0$. Similarly, the coefficients for
$\pi_{2k+1,1}$ are
\begin{equation*}
\gamma_{2k+1,1}=-\frac{\left<L_{2k+1},L_{2k-2}\right>_1}{\left<L_{2k-1},L_{2k-2}\right>_1},\quad
\gamma_{2k+1,2}=\frac{\left<L_{2k+1},L_{2k-1}\right>_1}{\left<L_{2k-1},L_{2k-2}\right>_1}.
\end{equation*}
Again, these coefficients exist and are unique. However, as
$\left<\pi_{2k,1},\pi_{2k,1}\right>_1=0$ and
$\left<\pi_{2k,1},y^j\right>_1=0$ for $j=0,\ldots,2k-1$, adding any
multiple of $\pi_{2k,1}$ to $\pi_{2k+1,1}$ will not change the
orthogonality conditions $\left<\pi_{2k+1,1},y^j\right>_1=0$ that is
satisfied by $\pi_{2k+1,1}$ and hence $\pi_{2k+1,1}$ is only
determined up to the addition of a multiple of $\pi_{2k,1}$.
\end{proof}
\subsection{The Christoffel Darboux formula for the kernel} In
\cite{Pierce}, skew orthogonal polynomials were interpreted as
multi-orthogonal polynomials and represented as the solution of a
Riemann-Hilbert problem. This representation allows us to use the
results in \cite{KDaem} to derive a Christoffel-Darboux formula for
the kernel (\ref{eq:ker}) in terms of the Riemann-Hilbert problem.

Let us recall the definitions of multi-orthogonal polynomials. First
let the weights $w_0$, $w_1$ and $w_2$ be
\begin{equation*}
w_0(x)=x^{M-N}e^{-Mx},\quad
w_l(x)=w(x)\int_{\mathbb{R}_+}\epsilon(x-y)L_{N-l-2}(y)w(y)\D
y,\quad l=1,2.
\end{equation*}
Note that the weights $w_l(x)$ are defined with the polynomials
$L_{N-3}$ and $L_{N-4}$ instead of $L_{N-1}$ and $L_{N-2}$. This is
because the construction below involves the polynomial $\pi_{N-2,1}$
as well as $\pi_{N,1}$. By taking $i=N-3$ in (\ref{eq:prod}), we see
that the orthogonality conditions for $\pi_{N,1}$ is also equivalent
to
\begin{equation*}
\begin{split}
\left<\pi_{N,1},x^j\right>_2=0,\quad j=0,\ldots,N-3,\quad
\left<\pi_{N,1},L_{N-j}\right>_2=0,\quad j=3,4,
\end{split}
\end{equation*}
provided $\left<L_{N-3},L_{N-4}\right>_1$ is also non-zero.

Then the multi-orthognoal polynomials of type II $P^{II}_{N,l}(x)$,
\cite{Ap1}, \cite{Ap2}, \cite{BK1}, \cite{vanGerKuij} are
polynomials of degree $N-1$ such that
\begin{equation}\label{eq:mop1}
\begin{split}
\int_{\mathbb{R}_+}P^{II}_{N,l}(x)x^jw_0(x)\D x&=0,\quad 0\leq j\leq N-3,\\
\int_{\mathbb{R}_+}P^{II}_{N,l}(x)w_m(x)\D x&=-2\pi
i\delta_{lm},\quad l,m=1,2.
\end{split}
\end{equation}
\begin{remark} More accurately, these are in fact the
multi-orthogonal polynomials with indices $(N-2,\vec{e}_l)$, where
$\vec{e}_1=(0,1)$ and $\vec{e}_2=(1,0)$.
\end{remark}
We will now define the multi-orthogonal polynomials of type I. Let
$P^{I}_{N,l}(x)$ be a function of the following form
\begin{equation}\label{eq:typeI}
P^{I}_{N,l}(x)=B_{N,l}(x)w_0+\sum_{k=1}^2\left(\delta_{kl}x+A_{N,k,l}\right)w_k,
\end{equation}
where $B_{N,l}(x)$ is a polynomial of degree $N-3$ and $A_{N,k,l}$
is independent on $x$. Moreover, let $P^{I}_{N,l}(x)$ satisfy
\begin{equation*}
\int_{\mathbb{R}_+}P^{I}_{N,l}(x)x^j\D x=0,\quad j=0,\ldots,N-1.
\end{equation*}
Then the polynomials $B_{N,l}(x)$ and $\delta_{kl}x+A_{N,k,l}$ are
multi-orthogonal polynomials of type I with indices $(N,2,1)$ for
$l=1$ and $(N,1,2)$ for $l=2$. We will now show that $P^{II}_{N,1}$
and $P^{II}_{N,2}$ exist and are unique if both
$\left<L_{N-1},L_{N-2}\right>_1$ and
$\left<L_{N-3},L_{N-4}\right>_1$ are non-zero.
\begin{lemma}\label{le:exist}
Suppose both $\left<L_{N-1},L_{N-2}\right>_1$ and
$\left<L_{N-3},L_{N-4}\right>_1$ are non-zero, then the polynomials
$P^{II}_{N,1}$ and $P^{II}_{N,2}$ exist and are unique.
\end{lemma}
\begin{proof}
Let us write $L_{N-l}$ as
\begin{equation*}
L_{N-l}=\frac{d}{dx}\left(q_l(x)x(t- x)w\right)w^{-1}+R_l(x),\quad
l=1,\ldots, 4.
\end{equation*}
for some polynomials $q_l$ of degree $N-l-2$ and $R_l$ of degree
$1$, then by Lemma \ref{le:linear}, we see that both $\varrho_{N-1}$
and $\varrho_{N-3}$ are invertible and hence the composition
$\varrho_{N-1}\varrho_{N-3}^{-1}$ is also invertible. In particular,
there exists a $2\times 2$ invertible matrix with entries $c_{l,k}$
and polynomials $\tilde{q}_l$ of degree $N-3$ such that
\begin{equation*}
\begin{split}
L_{N-l}=\frac{d}{dx}\left(\tilde{q}_l(x)x(t-
x)w\right)w^{-1}+c_{l-2,1}L_{N-1}+c_{l-2,2}L_{N-2},\quad l=3,4.
\end{split}
\end{equation*}
Then by the first condition in (\ref{eq:mop1}), we see that
\begin{equation}\label{eq:ex}
\int_{\mathbb{R}_+}P^{II}_{N,l}w_j(x)dx=\left<P^{II}_{N,l},c_{l,1}L_{N-1}+c_{l,2}L_{N-2}\right>_1=-2\pi
i\delta_{lj},\quad l=1,2,\quad j=1,2.
\end{equation}
As the matrix with entries $c_{ij}$ is invertible, we see that the
linear equations (\ref{eq:ex}) has a unique solution in the linear
span of $L_{N-1}$ and $L_{N-2}$ if and only if
$\left<L_{N-1},L_{N-2}\right>_1\neq 0$.
\end{proof}
As we shall see, existence and uniqueness of $P^{II}_{N,l}$ would
imply that the multi-orthogonal polynomials of type I also exist and
are unique. As in \cite{Pierce}, the multi-orthogonal polynomial
together with the skew orthogonal polynomials form the solution of a
Riemann-Hilbert problem. Let $Y(x)$ be the matrix
\begin{equation}\label{eq:Ymatr}
Y(x)=\begin{pmatrix}\pi_{N,1}(x)&C\left(\pi_{N,1}w_0\right)&C\left(\pi_{N,1}w_1\right)
&C\left(\pi_{N,1}w_{2}\right)\\
\kappa\pi_{N-2,1}(x)&\cdots&&\cdots&\\
P^{II}_{N,1}(x)&\ddots&&\cdots&\\
P^{II}_{N,2}(x)&\cdots&&\cdots
\end{pmatrix},
\end{equation}
where $\kappa$ is the constant
\begin{equation}\label{eq:kappa}
\kappa=-\frac{2\pi i}{
\left<\pi_{N-2,1},x^{N-3}\right>_2}=-\frac{4\pi i}{
Mh_{N-1,1}},\quad
h_{2j-1,1}=\left<\pi_{2j-2,1},\pi_{2j-1,1}\right>_1.
\end{equation}
and $C(f)$ is the Cauchy transform
\begin{equation}\label{eq:cauchy}
C(f)(x)=\frac{1}{2\pi i}\int_{\mathbb{R}_+}\frac{f(s)}{s-x}ds.
\end{equation}
Then by using the orthogonality conditions of skew orthogonal
polynomials and multi-orthogonal polynomials, together with the jump
discontinuity of the Cauchy transform, one can check that $Y(x)$
satisfies the following Riemann-Hilbert problem.
\begin{equation}\label{eq:RHPY}
\begin{split}
1.\quad &\text{$Y(z)$ is analytic in
$\mathbb{C}\setminus\mathbb{R}_+$},\\
2.\quad &Y_+(z)=Y_-(z)\begin{pmatrix}1&w_0&w_1(z)&w_{2}(z)\\
0&1&0&0\\
0&0&1&0\\
0&0&0&1
\end{pmatrix},\quad z\in\mathbb{R}_+,\\
3.\quad &Y(z)=\left(I+O(z^{-1})\right)\begin{pmatrix}z^{N}\\
&z^{-N+2}\\
&&z^{-1}\\
&&&z^{-1}
\end{pmatrix},\quad z\rightarrow\infty.
\end{split}
\end{equation}
Similarly, multi-orthogonal polynomials of type I can also be
arranged to satisfy a Riemann-Hilbert problem. Let $\epsilon$ be the
operator
\begin{equation}\label{eq:epsilon}
\epsilon(f)(x)=\frac{1}{2}\int_{0}^{\infty}\epsilon(x-y)f(y)\D y.
\end{equation}
First note that the functions $\psi_{j}(x)=\epsilon(\pi_{j,1}w)$ for
$j\geq 2$ can be express in the form of (\ref{eq:typeI}). By Lemma
\ref{le:linear}, we can write $\pi_{j,1}(x)$ as
\begin{equation*}
\pi_{j,1}(x)=\frac{d}{dx}\left(B_j(x)x(t-
x)w\right)w^{-1}+A_{j,1}L_{N-3}+A_{j,2}L_{N-4},
\end{equation*}
Then we have
\begin{equation}\label{eq:pitypeI}
\epsilon(\pi_{j,1}w)w=B_j(x)w_0+A_{j,1}w_1+A_{j,2}w_2.
\end{equation}
where $B_j(x)$ is a polynomial of degree $j-2$.

Let $X(z)$ be the matrix value function defined by
\begin{equation}\label{eq:Xmat}
X(z)=\begin{pmatrix}-\frac{\kappa M}{2}
C\left(\psi_{N-2}w\right)&\frac{\kappa M}{2}B_{N-2}&\frac{\kappa
M}{2}
A_{N-2,1}&\frac{\kappa M}{2}A_{N-2,2}\\
-\frac{ M}{2}C\left(\psi_{N}w\right)&\frac{ M}{2}B_{N}&\frac{
M}{2}A_{N,1}&\frac{ M}{2}
A_{N,2}\\
-C\left(P^{I}_{N,1}\right)&\cdots&\cdots&\\
-C\left(P^{I}_{N,2}\right)&\cdots&\cdots
\end{pmatrix}.
\end{equation}
Then by using the orthogonality and the the jump discontinuity of
the Cauchy transform, it is easy to check that $X^{-T}(z)$ and
$Y(z)$ satisfy the same Riemann-Hilbert problem and hence the
multi-orthogonal polynomials of type I also exist and are unique.

We will now show that the kernel $S_1(x,y)$ given by (\ref{eq:ker})
can be expressed in terms of the matrix $Y(z)$.
\begin{proposition}\label{pro:CD}
Suppose
$\left<L_{N-3},L_{N-4}\right>_1\left<L_{N-1},L_{N-2}\right>_1\neq 0
$ and let the kernel $S_1(x,y)$ be
\begin{equation}\label{eq:kerS}
S_1(x,y)=-\sum_{j,k=0}^{N-1}r_j(x)w(x)\mu_{jk}\epsilon(r_kw)(y),
\end{equation}
where $r_j(x)$ is an arbitrary degree $j$ monic polynomial for
$j<N-2$ and $r_j(x)=\pi_{j,1}(x)$ for $j\geq N-2$. The matrix $\mu$
with entries $\mu_{jk}$ is the inverse of the matrix $\mathbb{D}$
whose entries are given by
\begin{equation}\label{eq:M}
\left(\mathbb{D}\right)_{jk}=\left<r_j,r_k\right>_1,\quad
j,k=0,\ldots,N-1.
\end{equation}
Then the kernel $S_1(x,y)$ exists and is equal to
\begin{equation}\label{eq:kerRHP}
S_{1}(x,y)=\frac{w(x)w^{-1}(y)}{2\pi i(x-y)}\begin{pmatrix}0&w_0(y)&
w_1(y)&w_{2}(y)\end{pmatrix}Y_+^{-1}(y)Y_+(x)\begin{pmatrix}1\\0\\0\\0\end{pmatrix}.
\end{equation}
\end{proposition}
\begin{proof} First note that, since $\left<L_{N-1},L_{N-2}\right>_1\neq 0$ and
$\left<L_{N-3},L_{N-4}\right>_1\neq 0$, the skew orthogonal
polynomials $\pi_{N-l,1}$ exist for $l=-1,\ldots,2$. In particular,
the moment matrix $\tilde{\mathbb{D}}$ with entries
\begin{equation*}
\left(\tilde{\mathbb{D}}\right)_{jk}=\left<x^j,y^k\right>_1, \quad
j,k=0,\ldots,N-1
\end{equation*}
is invertible. Since the polynomials $\pi_{N-j,1}$ exist for
$j=-1,\ldots,2$, the sequence $r_k(x)$ and $x^k$ are related by an
invertible transformation. Therefore the matrix $\mathbb{D}$ in
(\ref{eq:M}) is also invertible. As the matrix $\mathbb{D}$ is of
the form
\begin{equation*}
\mathbb{D}=\begin{pmatrix} \mathbb{D}_0&0\\
0&h_{N-1,1}\mathcal{J}
\end{pmatrix},\quad \mathcal{J}=\begin{pmatrix}0&1\\-1&0\end{pmatrix}
\end{equation*}
where $\mathbb{D}_0$ has entries $\left<r_j,r_k\right>_1$ for $j,k$
from $0$ to $N-3$. From this, we see that the matrix $\mu_{jk}$ is
of the form
\begin{equation}\label{eq:mu}
\mu=\begin{pmatrix}
\mathbb{D}_0^{-1}&0&\\
0&-h_{N-1,1}^{-1}\mathcal{J}
\end{pmatrix}.
\end{equation}
As in \cite{KDaem}, let us now expand the functions $xr_j(x)$ and
$x\epsilon(r_jw)$.
\begin{equation}\label{eq:cd}
\begin{split}
xr_j(x)&=\sum_{k=0}^{N-1}c_{jk}r_k(x)+\delta_{N-1,j}\pi_{N,1}(x),\\
x\epsilon(r_jw)(x)&=\sum_{k=0}^{N-1}d_{jk}\epsilon(r_kw)+d_{j,N}\psi_{N}
+d_{j,N+1}\frac{P^I_{N,1}(x)}{w(x)}+d_{j,N+2}\frac{P^I_{N,2}(x)}{w(x)}.
\end{split}
\end{equation}
Then the coefficients $c_{jk}$ and $d_{jk}$ for $j,k=0,\ldots,N-1$
are given by
\begin{equation*}
\begin{split}
c_{jk}=\sum_{l=0}^{N-1}\left<xr_j,r_l\right>_1\mu_{lk},\quad
d_{jk}=\sum_{l=0}^{N-1}\mu_{kl}\left<xr_l,r_k\right>_1.
\end{split}
\end{equation*}
Therefore if we let $\mathcal{C}$ be the matrix with entries
$c_{jk}$, $j,k=0,\ldots N-1$ and $\mathcal{D}$ be the matrix with
entries $d_{jk}$ for $j,k=0,\ldots N-1$, then we have
\begin{equation}\label{eq:expand}
\mathcal{C}=\mathbb{D}_1\mu,\quad
\mathcal{D}=\left(\mu\mathbb{D}_1\right)^T,
\end{equation}
where $\mathbb{D}_1$ is the matrix with entries
$\left<xr_j,r_k\right>_1$ for $j,k=0,\ldots,N-1$. From
(\ref{eq:cd}), we obtain
\begin{equation}\label{eq:S1}
\begin{split}
&(y-x)S_1(x,y)=\sum_{j=0}^{N-1}\pi_{N,1}(x)w(x)\mu_{N-1,j}\epsilon\left(r_{j}w\right)(y)\\
&-\sum_{j,k=0}^{N-1}r_j(x)w(x)\mu_{jk}\left(d_{k,N}\psi_{N}(y)+d_{k,N+1}\frac{P^I_{N,1}(y)}{w(y)}
+d_{k,N+2}\frac{P^I_{N,2}(y)}{w(y)}\right)\\
&+r^T(x)w(x)\left(\mathcal{C}^T\mu-\mu\mathcal{D}\right)\epsilon(rw)(y),
\end{split}
\end{equation}
where $r(x)$ is the column vector with components $r_k(x)$. Now by
(\ref{eq:expand}), we see that
\begin{equation*}
\mathcal{C}^T\mu=\mu^T\mathbb{D}_1^T\mu,\quad
\mu\mathcal{D}=\mu\mathbb{D}_1^T\mu^T,
\end{equation*}
which are equal as $\mu^T=-\mu$. Now from the form of $\mu$ in
(\ref{eq:mu}), we see that
$\mu_{N-1,j}=\delta_{j,N-2}h_{N-1,1}^{-1}$. Hence we have
\begin{equation}\label{eq:first}
\sum_{j=0}^{N-1}\pi_{N,1}(x)w(x)\mu_{N-1,j}\epsilon\left(r_{j}w\right)(y)=h_{N-1,1}^{-1}\pi_{N,1}(x)w(x)\psi_{N-2}(y).
\end{equation}
Let us now consider the second term in (\ref{eq:S1}). As in
(\ref{eq:pitypeI}) we can write $\epsilon(r_kw)w$ as
\begin{equation*}
\epsilon(r_kw)w=q_k(x)w_0+D_{k,1}w_1+D_{k,2}w_2,
\end{equation*}
where $q_k(x)$ is of degree $k-2$. Then from the form of
$P_{N,j}^{I}$ in (\ref{eq:typeI}) and the orthogonality condition
(\ref{eq:mop1}), we see that the coefficients $d_{k,N+l}$, $l=1,2$
are given by
\begin{equation*}
d_{k,N+l}=D_{k,l}=-\frac{1}{2\pi
i}\int_{\mathbb{R}_+}P_{N,l}^{II}(x)\epsilon(r_kw)w\D x.
\end{equation*}
For $k\neq N-1$, the polynomial $q_k$ is of degree less than or
equal to $N-4$, while for $k=N-1$, $B_{N}$ in (\ref{eq:pitypeI}) is
a polynomial of degree $N-2$ and hence the coefficient $d_{k,N}$ is
zero unless $k=N-1$. For $k=N-1$, it is given by the leading
coefficient of $q_{N-1}$ divided by the leading coefficient of
$B_{N}$. Since
\begin{equation*}
\pi_{N-1,1}(x)=\frac{d}{dx}\left(q_{N-1}(x)x(t-
x)w\right)w^{-1}+D_{k,1}L_{N-3}+D_{k,2}L_{N-4},
\end{equation*}
we see that both the leading coefficient of $q_{N-1}(x)$ and
 $B_{N}$ is $\frac{2}{M }$. Hence $d_{k,N}$ is $\delta_{k,N-1}$. This gives us
\begin{equation*}
\sum_{k,j=0}^{N-1}r_j(x)w(x)\mu_{jk}d_{k,N}\psi_{N}(y)=-h_{N-1,1}^{-1}\pi_{N-2,1}(x)w(x)\psi_{N}(y).
\end{equation*}
To express the second term in (\ref{eq:S1}) in terms of the
multi-orthogonal polynomials, let us now express $P^{II}_{N,l}$ in
terms of the polynomials $r_k$. Let us write
$P^{II}_{N,l}=\sum_{j=0}^{N-1}a_jr_j(x)$. Then we have
\begin{equation*}
\begin{split}
\int_{\mathbb{R}_+}P^{II}_{N,l}(x)\epsilon(r_kw)w\D
x=\sum_{j=0}^{N-1}a_j\left(\mathbb{D}_1\right)_{jk},\quad
\sum_{k=0}^{N-1}\int_{\mathbb{R}_+}P^{II}_{N,l}(x)\epsilon(r_kw)w\D
x\mu_{kj}=a_j.
\end{split}
\end{equation*}
Hence $P^{II}_{N,l}(x)$ can be written as
\begin{equation*}
\begin{split}
P^{II}_{N,l}(x)=\sum_{k,j=0}^{N-1}\left(\int_{\mathbb{R}_+}P^{II}_{N,l}(x)\epsilon(r_kw)w\D
x\mu_{kj}\right)r_j(x)=-2\pi
i\sum_{k,j=0}^{N-1}d_{k,N+l}\mu_{kj}r_j(x)
\end{split}
\end{equation*}
Therefore the second term in (\ref{eq:S1}) is given by
\begin{equation*}
\begin{split}
&\sum_{j,k=0}^{N-1}r_j(x)w(x)\mu_{jk}\left(d_{k,N}\psi_{N}(y)+d_{k,N+1}\frac{P^I_{N,1}(y)}{w(y)}
+d_{k,N+2}\frac{P^I_{N,2}(y)}{w(y)}\right)\\
&=-h_{N-1,1}^{-1}\pi_{N-2,1}(x)\psi_{N}(y)+\frac{1}{2\pi
i}\sum_{l=1}^2P^{II}_{N,l}(x)w(x)w^{-1}(y)P^I_{N,l}(y).
\end{split}
\end{equation*}
From this and (\ref{eq:first}), we obtain
\begin{equation*}
\begin{split}
(y-x)S_1(x,y)&=h_{N-1,1}^{-1}\pi_{N,1}(x)w(x)\psi_{N-2}(y)+
h_{N-1,1}^{-1}\pi_{N-2,1}(x)w(x)\psi_{N}(y)\\
&-\frac{1}{2\pi
i}\sum_{l=1}^2P^{II}_{N,l}(x)w(x)w^{-1}(y)P^I_{N,l}(y).
\end{split}
\end{equation*}
By using the fact that $Y^{-1}(y)=X^T(y)$ and the expressions of the
matrix $Y$ (\ref{eq:Ymatr}) and $X$ (\ref{eq:Xmat}), together with
(\ref{eq:kappa}), we see that this is the same as (\ref{eq:kerRHP}).
\end{proof}
\subsection{The kernel in terms of Laguerre polynomials} We will
now use a result in \cite{baikext} to further simplify the
expression of the kernel $S_1(x,y)$ so that its asymptotics can be
computed using the asymptotics of Laguerre polynomials. Let us
recall the set up in \cite{baikext}. First let $Y(x)$ be a matrix
satisfying the Riemann-Hilbert problem
\begin{equation*}
\begin{split}
1.\quad &\text{$Y(z)$ is analytic in
$\mathbb{C}\setminus\mathbb{R}_+$},\\
2.\quad &Y_+(z)=Y_-(z)\begin{pmatrix}1&w_0(z)&w_1(z)&\cdots&w_{r}(z)\\
0&1&0&\cdots&0\\
\vdots&&\ddots&\vdots&\\
0&&\cdots&&1
\end{pmatrix},\quad z\in\mathbb{R}_+\\
3.\quad &Y(z)=\left(I+O(z^{-1})\right)\begin{pmatrix}z^{n}\\
&z^{-n+r}\\
&&z^{-1}\\
&&&\ddots\\
&&&&z^{-1}
\end{pmatrix},\quad z\rightarrow\infty.
\end{split}
\end{equation*}
and let $\mathcal{K}_1(x,y)$ be the kernel given by
\begin{equation}\label{eq:k1}
\mathcal{K}_1(x,y)=\frac{w_0(x)w_0^{-1}(y)}{2\pi
i(x-y)}\begin{pmatrix}0&w_0(y)&
\ldots&w_{r}(y)\end{pmatrix}Y_+^{-1}(y)Y_+(x)\begin{pmatrix}1\\0\\\vdots\\0\end{pmatrix}.
\end{equation}
Let $\pi_{j,2}(x)$ be the monic orthogonal polynomials with respect
to the weight $w_0(x)$. Let $\mathcal{K}_0(x,y)$ be the following
kernel.
\begin{equation}\label{eq:k0}
\mathcal{K}_0(x,y)=w_0(x)\frac{\pi_{2,n}(x)\pi_{2,n-1}(y)-\pi_{2,n}(y)\pi_{2,n-1}(x)}{h_{n-1,2}(x-y)},
\end{equation}
where $h_{j,2}=\int_{0}^{\infty}\pi_{j,2}^2w_0\D x$. Let $\pi(z)$,
$v(z)$ and $u(z)$ be the following vectors
\begin{equation}\label{eq:tv}
\pi(z)=\left(\pi_{n-r,2},\ldots,\pi_{n-1,2}\right)^T,\quad
v(z)=\left(w_1,\ldots,w_r\right)^Tw_0^{-1}, \quad
u(z)=\left(I-\mathcal{K}_0^T\right)v(z)
\end{equation}
and let $\mathcal{W}$ be the matrix
$\mathcal{W}=\int_{\mathbb{R}_+}\pi(z)v^T(z)w_0(z)dz$. Then the
kernel $\mathcal{K}_1(x,y)$ can be express as \cite{baikext}
\begin{proposition}\label{pro:baik}
Suppose $\int_{\mathbb{R}_+}p(x)w_i(x)\D x$ converges for any
polynomial $p(x)$. Then the kernel $\mathcal{K}_1(x,y)$ defined by
(\ref{eq:k1}) is given by
\begin{equation}\label{eq:kerform}
\mathcal{K}_1(x,y)-\mathcal{K}_0(x,y)=w_0(x)u^T(y)\mathcal{W}^{-1}\pi(x).
\end{equation}
\end{proposition}
\begin{remark} Although in \cite{baikext}, the theorem is stated
with the jump of $Y(x)$ on $\mathbb{R}$ instead of $\mathbb{R}_+$,
while the weights are of the special form $w_0=e^{-NV(x)}$,
$w_j(x)=e^{a_jx}$, where $V(x)$ is an even degree polynomial, the
proof in \cite{baikext} in fact remain valid as long as integrals of
the form $\int_{\mathbb{R}_+}p(x)w_i(x)\D x$ converges for any
polynomial $p(x)$. This is true in our case.
\end{remark}
We can now apply Proposition \ref{pro:baik} to our case. In our
case, the vectors $\pi(z)$ and $v(z)$ are given by
\begin{equation*}
\pi(z)=\left(L_{N-2}\quad L_{N-1}\right)^T,\quad v(z)=\left(w_1\quad
w_2\right)^Tw_0^{-1},
\end{equation*}
while the matrix $\mathcal{W}$ is given by
\begin{equation*}
\mathcal{W}=\begin{pmatrix}\left<L_{N-2},L_{N-3}\right>_1&\left<L_{N-2},L_{N-4}\right>_1\\
\left<L_{N-1},L_{N-3}\right>_1&\left<L_{N-1},L_{N-4}\right>_1\end{pmatrix}
\end{equation*}
By corollary \ref{cor:linear}, we see that
$\left<L_{N-2},L_{N-3}\right>_1=0$ and hence the determinant of
$\mathcal{W}$ is
\begin{equation*}
\det
\mathcal{W}=\left<L_{N-2},L_{N-4}\right>_1\left<L_{N-1},L_{N-3}\right>_1.
\end{equation*}
Suppose $\det\mathcal{W}\neq 0$ and that the multi-orthogonal
polynomials $P^{II}_{N,l}$ exist. Let us now consider the vector
$u(z)$. By the Christoffel-Darboux formula, we have
\begin{equation*}
u(z)=\left(I-\mathcal{K}_0^T\right)v(z)=v(z)-\sum_{j=0}^{N-1}\frac{L_j(z)}{h_{j,0}}\left(\left<L_j,L_{N-3}\right>_1\quad\left<L_j,L_{N-4}\right>_1\right),
\end{equation*}
where $h_{j,0}=\left<L_j,L_j\right>_2$. We shall show that $L_{N-3}$
and $L_{N-4}$ can be written in the following form
\begin{equation}\label{eq:mappi}
L_{N-l}=\frac{d}{dx}\left(q_l(x)x(t-
x)w\right)w^{-1}+Q_{l-2,1}\pi_{N+1,1}+Q_{l-2,2}\pi_{N,1},\quad
l=3,4,
\end{equation}
for some polynomial $q_l(x)$ of degree $N-1$. By Lemma
\ref{le:linear}, we see that if $\left<L_{N-3},L_{N-4}\right>_1\neq
0$, then the map $\varrho_{N-3}$ in Definition \ref{de:fn} is
invertible. Therefore if the restriction of the map $f$ in
Definition \ref{de:fn} is also invertible on the span of
$\pi_{N+1,1}$ and $\pi_{N,1}$, we will be able to write $L_{N-3}$
and $L_{N-4}$ in the form of (\ref{eq:mappi}).
\begin{lemma}\label{le:pi2nmap}
Let $f$ be the map in Definition \ref{de:fn} and let $\varrho_{\pi}$
be its restriction to the span of $\pi_{N+1,1}$ and $\pi_{N,1}$.
Then $\varrho_{\pi}$ is invertible.
\end{lemma}
\begin{proof} Suppose there exist $a_1$ and $a_2$ such that
\begin{equation*}
a_1\pi_{N+1,1}+a_2\pi_{N,1}=\frac{d}{dx}\left(q(x)x(t-
x)w\right)w^{-1}
\end{equation*}
for some polynomial $q(x)$ of degree at most $N-1$. Then we have
\begin{equation*}
\left<x^j,a_1\pi_{N+1,1}+a_2\pi_{N,1}\right>_1=\left<x^jq(x)\right>_2=0,\quad
j=0,\ldots,N-1.
\end{equation*}
As the degree of $q(x)$ is at most $N-1$, this is only possible if
$q(x)=0$.
\end{proof}
The composition of $\varrho_{N-3}$ and $\varrho_{\pi}^{-1}$ will
therefore give us a representation of $L_{N-3}$ and $L_{N-4}$ in the
form of (\ref{eq:mappi}). By using this representation and the fact
that $q_l(x)$ is of degree at most $N-1$, we see that
\begin{equation*}
\begin{split}
\mathcal{K}_0^T\left(w_lw_0^{-1}\right)&=\sum_{j=0}^{N-1}\frac{L_j(x)}{h_{j,0}}\left<L_j,L_{N-l-2}\right>_1\\
&=\sum_{j=0}^{N-1}\frac{L_j(x)}{h_{j,0}}\left(\left<L_j,q_l\right>_2+\left<L_j,Q_{l-2,1}\pi_{N+1,1}+Q_{l-2,2}\pi_{N,1}\right>_1\right),\\
&=q_l(x).
\end{split}
\end{equation*}
Therefore the vector $u(x)$ is given by
\begin{equation*}
u(x)=w(x)w_0^{-1}(x)Q\epsilon\left(\pi_{N+1,1}w\quad
\pi_{N,1}w\right)^T
\end{equation*}
where $Q$ is the matrix with entries $Q_{i,j}$. We will now
determine the constants $Q_{i,j}$.
\begin{lemma}\label{le:A}
Let $\pi_{N,1}$ and $\pi_{N+1,1}$ be the monic skew orthogonal
polynomial with respect to the weight $w(x)$ and choose
$\pi_{N+1,1}$ so that the constant $c$ in (\ref{eq:p2N}) is zero.
Then the vector $u(y)$ in (\ref{eq:kerform}) is given by
\begin{equation}\label{eq:ux}
u(x)=w(x)w_0^{-1}(x)Q\epsilon\left(\pi_{N+1,1}w\quad
\pi_{N,1}w\right)^T,
\end{equation}
where $Q$ is the matrix whose entries $Q_{i,j}$ are given by
\begin{equation}\label{eq:Aent}
\begin{split}
Q_{i,1}&=-\frac{M \left<L_{N-1},L_{N-i-2}\right>_1}{2h_{N-1,0}},\\
Q_{i,2}&=\left(Mt-
\left(N+M\right)\right)\frac{\left<L_{N-1},L_{N-i-2}\right>_1}{2h_{N-1,0}}-M
\frac{\left<L_{N-2},L_{N-i-2}\right>_1}{2h_{N-2,0}}.
\end{split}
\end{equation}
\end{lemma}
\begin{proof} First let us compute the leading order coefficients of
the polynomial $q_l(x)$ in (\ref{eq:mappi}). Let
$q_l(x)=q_{l,N-1}x^{N-1}+q_{l,N-2}x^{N-2}+O(x^{N-3})$, then we have
\begin{equation*}
\begin{split}
&\frac{d}{dx}\left(q_l(x)x(t- x)w\right)w^{-1}=\frac{M }{2}q_{l,N-1}x^{N+1},\\
&+\left(-\frac{ }{2}(N+M)q_{l,N-1}-\frac{Mt}{2}q_{l,N-1}+\frac{M
}{2}q_{l,N-2}\right)x^{N}+O(x^{N-1})
\end{split}
\end{equation*}
From (\ref{eq:mappi}), we see that
\begin{equation}\label{eq:qlead}
q_{l,N-1}=-\frac{2}{M }Q_{l-2,1}.
\end{equation}
On the other hand, by orthogonality, we have
\begin{equation*}
\left<L_{N-1},L_{N-l}\right>_1=\left<L_{N-1},q_l\right>_2=q_{l,N-1}h_{N-1,0}.
\end{equation*}
Therefore $Q_{l-2,1}$ is given by
\begin{equation*}
Q_{l-2,1}=-\frac{M \left<L_{N-1},L_{N-l}\right>_1}{2h_{N-1,0}}
\end{equation*}
Let us now compute $Q_{i,2}$. By taking the skew product, we have
\begin{equation*}
\left<L_{N-2},L_{N-l}\right>_1=q_{l,N-1}\left<L_{N-2},x^{N-1}\right>_2
+q_{l,N-2}h_{N-2,0}
\end{equation*}
Now from (\ref{eq:laguerre}), we obtain
\begin{equation}\label{eq:nnp1}
\begin{split}
\left<x^{j+1},L_{j}\right>_2&=\left<L_{j+1}+\frac{(M-N+j+1)(j+1)}{M}x^j,L_j\right>_2,\\
&=\frac{(M-N+j+1)(j+1)}{M}h_{j,0}.
\end{split}
\end{equation}
Hence $q_{l,N-2}$ is equal to
\begin{equation}\label{eq:q2}
q_{l,N-2}=\frac{\left<L_{N-2},L_{N-l}\right>_1}{h_{N-2,0}}-q_{l,N-1}\frac{(M-1)(N-1)}{M}
\end{equation}
By using the expansion (\ref{eq:p2N}) of the skew orthogonal
polynomials in terms of $L_{N-k}$, we have
\begin{equation*}
\begin{split}
\left<L_{N-l},L_{N}\right>_2&=\frac{M }{2}q_{l,N-1}\left<x^{N+1},L_{N}\right>_2+Q_{l-2,2}h_{N,0}\\
&+\left(\frac{ 1}{2}(-N-M)q_{l,N-1}-\frac{Mt}{2}q_{l,N-1}+\frac{M
}{2}q_{l,N-2}\right)h_{N,0}
\end{split}
\end{equation*}
By substituting (\ref{eq:q2}) into this, we obtain
\begin{equation*}
Q_{l-2,2}=\left(-\frac{1}{2}\left(M+N\right)+\frac{Mt}{2}\right)\frac{\left<L_{N-1},L_{N-l}\right>_1}{h_{N-1,0}}-\frac{M
}{2} \frac{\left<L_{N-2},L_{N-l}\right>_1}{h_{N-2,0}}
\end{equation*}
This proves the lemma.
\end{proof}
Note that the matrix $\det
(Q\mathcal{W}^{-1})=M^2/(4h_{N-1,0}h_{N-2,0})$ and hence
$Q\mathcal{W}^{-1}$ is invertible. From Lemma \ref{le:A} and
(\ref{eq:kerform}), we obtain Theorem \ref{thm:baik}.

We will now further simplify the expression for the correlation
kernel $S_1(x,y)$. First we make the following observation. From the
definition (\ref{eq:kerS}) of the kernel $S_1$, it is easy to check
that the following identity is true.
\begin{equation}\label{eq:commD}
-\frac{\p}{\p y}S_1(x,y)=\frac{\p}{\p x}S_1(y,x).
\end{equation}
We can use this simple observation in the expression
(\ref{eq:kerform1}) that we obtained in the last section. First let
us write the Christoffel Darboux kernel terms of the solution of a
Riemann-Hilbert problem, we have
\begin{equation}\label{eq:K2form}
K_2(x,y)=\left(\frac{y(t- y)}{x(t-
x)}\right)^{\frac{1}{2}}\frac{1}{2\pi i(x-y)}\left(0\quad
1\right)Z_+^{-1}(y)Z_+(x)\begin{pmatrix}1\\0\end{pmatrix},
\end{equation}
where $Z(x)$ is the following matrix
\begin{equation}\label{eq:Phimat}
Z(x)=\begin{pmatrix}L_N(x)&C\left(L_Nw_0\right)\\
\kappa_{0,N-1}L_{N-1}(x)&\kappa_{0,N-1}C\left(L_{N-1}w_0\right)
\end{pmatrix}w_0^{\frac{\sigma_3}{2}},
\end{equation}
where $\sigma_3$ is the Pauli matrix $\sigma_3=\diag(1,-1)$,
$\kappa_{0,n}=-\frac{2\pi i}{h_{0,n}}$ and $C(f)$ is the Cauchy
transform.

It is well-known that if the logarithmic derivative of $w_0$ is
rational, then the matrix $Z(x)$ in (\ref{eq:Phimat}) satisfies a
linear system of ODE
\begin{equation}\label{eq:LODE}
\frac{\p}{\p x}Z(x)=A(x)Z(x)
\end{equation}
for some rational function $A(x)$. By using (\ref{eq:laguerre}) and
the recurrence relation of the $L_n$
\begin{equation}\label{eq:recur}
L_n=-\frac{n}{M}\left(2+\frac{M-N-1-Mx}{n}\right)L_{n-1}-\frac{n(n-1)}{M^2}\left(1+\frac{M-N-1}{n}\right)L_{n-2},
\end{equation}
we obtain the matrix $A$ as
\begin{equation}\label{eq:matA}
A(x)=\begin{pmatrix}-\frac{M}{2}+\frac{M+N}{2x}&\frac{N}{\kappa_{0,N-1} x}\\
-\frac{M\kappa_{0,N-1}}{x}&\frac{M}{2}-\frac{M+N}{2x}\end{pmatrix}
\end{equation}
Let us now consider the derivative $\frac{\p K_2(x,y)}{\p y}$. We
have, from (\ref{eq:K2form})
\begin{equation}\label{eq:k2y}
\begin{split}
&\frac{\p K_2(x,y)}{\p y}=\frac{t-2 y}{\left(xy(t- x)(t-
y)\right)^{\frac{1}{2}}} \left(0\quad
1\right)\frac{Z_+^{-1}(y)Z_+(x)}{4\pi i(x-y)}\begin{pmatrix}1\\0\end{pmatrix}\\
&+\left(\frac{y(t- y)}{x(t-
x)}\right)^{\frac{1}{2}}\Bigg(\left(0\quad
1\right)\frac{Z_+^{-1}(y)Z_+(x)}{2\pi
i(x-y)^2}\begin{pmatrix}1\\0\end{pmatrix}-\left(0\quad
1\right)\frac{Z_+^{-1}(y)A(y)Z_+(x)}{2\pi
i(x-y)}\begin{pmatrix}1\\0\end{pmatrix}\Bigg).
\end{split}
\end{equation}
Similarly, $\frac{\p}{\p x}K_2(y,x)$ is given by
\begin{equation}\label{eq:k2x}
\begin{split}
&\frac{\p K_2(y,x)}{\p x}=\frac{t-2 x}{\left(xy(t- x)(t-
y)\right)^{\frac{1}{2}}} \left(0\quad
1\right)\frac{Z_+^{-1}(x)Z_+(y)}{4\pi i(y-x)}\begin{pmatrix}1\\0\end{pmatrix}\\
&+\left(\frac{x(t- x)}{y(t-
y)}\right)^{\frac{1}{2}}\Bigg(\left(0\quad
1\right)\frac{Z_+^{-1}(x)Z_+(y)}{2\pi
i(y-x)^2}\begin{pmatrix}1\\0\end{pmatrix}+\left(0\quad
1\right)\frac{Z_+^{-1}(x)A(x)Z_+(y)}{2\pi
i(x-y)}\begin{pmatrix}1\\0\end{pmatrix}\Bigg).
\end{split}
\end{equation}
The sum of these two derivatives is therefore given by
\begin{equation*}
\begin{split}
&\frac{\p K_2(y,x)}{\p x}+\frac{\p K_2(x,y)}{\p
y}=h_{0,N-1}^{-1}w(x)w(y)
\Bigg(-L_N(x)L_N(y)M \\
&+\left(L_{N-1}(y)L_N(x)+L_{N-1}(x)L_N(y)\right) \left(\frac{M
}{2}(x+y)-\frac{Mt+(M+N) }{2}\right)
\\
&-L_{N-1}(x)L_{N-1}(y)N \Bigg).
\end{split}
\end{equation*}
By using the recurrence relations of the Laguerre polynomials, we
obtain
\begin{equation}\label{eq:k2diff}
\begin{split}
&\frac{\p K_2(y,x)}{\p x}+\frac{\p K_2(x,y)}{\p y}=
h_{0,N-1}^{-1}w(x)w(y)\Bigg(\frac{(N-1)(M-1) }{2M}L_{N-2,N}(x,y)\\
&+\frac{M }{2}L_{N+1,N-1}(x,y) +\frac{(M+N) -Mt}{2}L_{N-1,N}(x,y)
\Bigg)
\end{split}
\end{equation}
where $L_{n,m}(x,y)$ is given by
\begin{equation}\label{eq:Lnm}
L_{n,m}(x,y)=L_{n}(x)L_m(y)+L_m(x)L_n(y).
\end{equation}
Let $K_1(x,y)$ be the correction kernel in (\ref{eq:kerform1}).
\begin{equation}\label{eq:ker1}
K_1(x,y)=\epsilon\left(\pi_{N+1,1}w\quad\pi_{N,1}w\right)(y)\begin{pmatrix}0&-\frac{M }{2h_{N-1}}\\
-\frac{M }{2h_{N-2}}&\frac{Mt-
(N+M)}{2h_{N-1}}\end{pmatrix}\pi(x)w(x)
\end{equation}
Then we have
\begin{equation}\label{eq:k1diff}
\begin{split}
&\frac{\p K_1(y,x)}{\p x}+\frac{\p K_1(x,y)}{\p
y}=w(x)w(y)\Bigg(-\frac{M }{2h_{0,N-2}}L_{N,N-2}(x,y)\\
&-\frac{M }{2h_{0,N-1}}L_{N+1,N-1}(x,y)+\frac{Mt- (N+M)}{2h_{0,N-1}}L_{N,N-1}(x,y)\\
&+\left( \frac{M }{2h_{0,N-2}}\frac{\left<L_N,L_{N-2}\right>_1}
{\left<L_{N-1},L_{N-2}\right>_1}-\frac{M
}{2h_{0,N-1}}\frac{\left<L_{N+1},L_{N-1}\right>_1}{\left<L_{N-1},L_{N-2}\right>_1}\right)L_{N-1,N-2}(x,y)
\\
&+\left(\frac{M }{h_{0,N-1}}\frac{\left<L_{N+1},L_{N-2}\right>_1}
{\left<L_{N-1},L_{N-2}\right>_1}-\frac{Mt-
(N+M)}{h_{0,N-1}}\frac{\left<L_N,L_{N-2}\right>_1}
{\left<L_{N-1},L_{N-2}\right>_1}\right)L_{N-1,N-1}(x,y)\Bigg)
\end{split}
\end{equation}
Then by using $\frac{\p S_1(y,x)}{\p x}+\frac{\p S_1(x,y)}{\p y}=0$
and the formula for $h_{0,n}$,
\begin{equation}\label{eq:hn}
h_{n,0}=\frac{n!\left(n+M-N\right)!}{M^{2n+M-N+1}},
\end{equation}
we obtain the following.
\begin{equation}\label{eq:relprod}
\begin{split}
\left<L_{N+1},L_{N-1}\right>_1
&=\frac{(N-1)(M-1)}{M^2}\left<L_{N},L_{N-2}\right>_1,\\
\left<L_{N+1},L_{N-2}\right>_1&=\frac{Mt- (N+M)}{M
}\left<L_{N},L_{N-2}\right>_1
\end{split}
\end{equation}
In particular, the correction term $K_1(x,y)$ in the kernel can be
written as
\begin{equation}\label{eq:k1form}
\begin{split}
&K_1(x,y)=-\frac{M }{2h_{0,N-2}}\Bigg(\epsilon\left(L_Nw\right)(y)L_{N-2}w(x)\\
&-\frac{\left<L_N,L_{N-2}\right>_1}
{\left<L_{N-1},L_{N-2}\right>_1}\left(\epsilon\left(L_{N-1}w\right)(y)L_{N-2}w(x)-
\epsilon\left(L_{N-2}w\right)(y)L_{N-1}w(x)\right)\Bigg)\\
&-\frac{M }{2h_{0,N-1}}\epsilon\left(L_{N+1}w\right)(y)L_{N-1}w(x)
+\frac{Mt-
(N+M)}{2h_{0,N-1}}\epsilon\left(L_Nw\right)(y)L_{N-1}w(x).
\end{split}
\end{equation}
\section{Derivative of the partition function}\label{se:Der}
In this section we will derive a formula for the derivative of
determinant of the matrix $\mathbb{D}$ given in (\ref{eq:M}). We
have the following.
\begin{proposition}\label{pro:derpar}
Let $\mathbb{D}$ be the matrix given by (\ref{eq:M}), where the
sequence of monic polynomials $r_j(x)$ in (\ref{eq:M}) is chosen
such that $r_j(x)$ are arbitrary degree $j$ monic polynomials that
are independent on $t$ and $r_j(x)=\pi_{j,1}(x)$ for $j=N-2,N-1$.
Then the logarithmic derivative of $\det\mathbb{D}$ with respect to
$t$ is given by
\begin{equation}\label{eq:derpar}
\frac{\p}{\p
t}\log\det\mathbb{D}=-\int_{\mathbb{R}_+}\frac{S_1(x,x)}{t- x}\D x,
\end{equation}
where $S_1(x,y)$ is the kernel given in (\ref{eq:kerS}).
\end{proposition}
\begin{proof} First let us differentiate the determinant
$\det\mathbb{D}$ with respect to $t$. We have
\begin{equation*}
\begin{split}
\frac{\p}{\p
t}\det\mathbb{D}&=\det\begin{pmatrix}\p_tD_{00}&D_{01}&\cdots&D_{0,N-1}\\
\vdots&\vdots&\ddots&\vdots\\
\p_tD_{2n-1,0}&D_{2n-1,1}&\cdots&D_{N-1,N-1}\end{pmatrix}\\&+
\det\begin{pmatrix}D_{00}&\p_tD_{01}&\cdots&D_{0,N-1}\\
\vdots&\vdots&\ddots&\vdots\\
D_{2n-1,0}&\p_tD_{2n-1,1}&\cdots&D_{N-1,N-1}\end{pmatrix}\\
&+\cdots+
\det\begin{pmatrix}D_{00}&D_{01}&\cdots&\p_tD_{0,N-1}\\
\vdots&\vdots&\ddots&\vdots\\
D_{N-1,0}&D_{N-1,1}&\cdots&\p_tD_{N-1,N-1}\end{pmatrix}.
\end{split}
\end{equation*}
Computing the individual determinants using the Laplace formula, we
obtain
\begin{equation}\label{eq:logder}
\frac{\p}{\p
t}\det\mathbb{D}=\det\mathbb{D}\sum_{i,j=0}^{N-1}\p_tD_{ij}\mu_{ji}.
\end{equation}
As $r_j(x)$ are independent on $t$ for $j<N-2$, the derivative
$\p_tD_{ij}$ is given by
\begin{equation*}
\p_tD_{ij}=-\frac{1}{2}\left(\left<\frac{r_i}{t- x},r_j\right>_1+
\left<r_i,\frac{r_j}{t- y}\right>_1\right)=
-\frac{1}{2}\left(\left<\frac{r_i}{t- x},r_j\right>_1-
\left<\frac{r_j}{t- x},r_i\right>_1\right)
\end{equation*}
for $i,j<N-2$. For either $i$ or $j$ equal to $N-2$ or $N-1$, we
have
\begin{equation*}
\begin{split}
\p_tD_{i,N-1}&=\delta_{N-2,i}\Bigg(-\frac{1}{2}\left(\left<\frac{\pi_{N-2,1}}{t-
{x}},\pi_{N-1,1}\right>_1
-\left<\frac{\pi_{N-1,1}}{t- {x}},\pi_{N-2,1}\right>_1\right)\\
&+\left<\p_t\pi_{N-2,1},\pi_{N-1,1}\right>_1+\left<\pi_{N-2,1},\p_t\pi_{N-1,1}\right>_1
\Bigg).
\end{split}
\end{equation*}
Note that by orthogonality, the last two terms in the above
expression are zero as $\p_t\pi_{N-1,1}$ is of degree $N-2$ and
$\p_t\pi_{N-2,1}$ is of degree $N-3$. Applying the same argument to
$\p_tD_{i,N-2}$, we obtain
\begin{equation*}
\begin{split}
\p_tD_{i,N-1}&=-\delta_{N-2,i}\left(\left<\frac{\pi_{N-2,1}}{t-
{x}},\pi_{N-1,1}\right>_1
-\left<\frac{\pi_{N-1,1}}{t- {x}},\pi_{N-2,1}\right>_1\right)\\
\p_tD_{i,N-2}&=-\frac{\delta_{N-1,i}}{2}\left(\left<\frac{\pi_{N-1,1}}{t-
{x}},\pi_{N-2,1}\right>_1 -\left<\frac{\pi_{N-2,1}}{t-
{x}},\pi_{N-1,1}\right>_1\right).
\end{split}
\end{equation*}
As $\mathbb{D}$ is anti-symmetric, the derivatives $\p_tD_{N-1,i}$
and $\p_tD_{N-2,i}$ are given by $\p_tD_{N-1,i}=-\p_tD_{i,N-1}$ and
$\p_tD_{N-2,i}=-\p_tD_{i,N-2}$. From these and (\ref{eq:logder}), we
obtain
\begin{equation}\label{eq:de}
\begin{split}
\frac{\p}{\p
t}\det\mathbb{D}&=\det\mathbb{D}\Bigg(\sum_{i,j=0}^{N-3}\left<\frac{r_i}{t-
x},r_j\right>_1\mu_{ij}
\\&+2\left(\left<\frac{\pi_{N-2,1}}{t- {x}},\pi_{N-1,1}\right>_1
-\left<\frac{\pi_{N-1,1}}{t-
{x}},\pi_{N-2,1}\right>_1\right)\mu_{N-2,N-1}\Bigg),
\end{split}
\end{equation}
where we have used the anti-symmetry of $\mathbb{D}$ and $\mu$ to
obtain the last term. From the structure of the matrix $\mu$ in
(\ref{eq:mu}), we see that the last term in (\ref{eq:de}) can be
written as
\begin{equation*}
\begin{split}
&2\left(\left<\frac{\pi_{N-2,1}}{t- {x}},\pi_{N-1,1}\right>_1
-\left<\frac{\pi_{N-1,1}}{t- {x}},\pi_{N-2,1}\right>_1\right)\mu_{N-2,N-1}\Bigg)=\\
&\sum_{i=0}^{N-1}\sum_{j=N-2}^{N-1}\left<\frac{r_i}{t-
x},r_j\right>_1\mu_{ij}
+\sum_{i=N-2}^{N-1}\sum_{j=0}^{N-1}\left<\frac{r_i}{t-
x},r_j\right>_1\mu_{ij}
\end{split}
\end{equation*}
From this, (\ref{eq:de}) and the expression of the kernel in
(\ref{eq:kerS}), we obtain (\ref{eq:derpar}).
\end{proof}
\section{Asymptotic analysis}\label{se:asymskew}
We will now use the asymptotics of the Laguerre polynomials to
obtain asymptotic expression for the kernel $S_1(x,y)$. We will
demonstrate the asymptotic analysis with the assumption that
$M/N\rightarrow\gamma=1$. The same analysis also be applied for
general $\gamma$.

Let us recall the asymptotics of the Laguerre polynomials in
different regions of $\mathbb{R}_+$. These asymptotics can be found
in \cite{V}. When $\gamma=1$, the end points $b_-$ and $b_+$ of the
support of $\rho$ in (\ref{eq:MP}) are $0$ and $4$ respectively.

Let $n=N-j$, where $j=O(1)$. Let us define the function $\varphi(x)$
to be
\begin{equation}\label{eq:varphi}
\varphi(x)=\int_{4}^x\frac{1}{2}\sqrt{\frac{s-4}{s}}ds,
\end{equation}
where the contour of integration is chosen such that it does not
intersect the interval $(-\infty,4)$. We can now use $\varphi(x)$ to
define two maps $\tilde{f}_{n}$ and $f_{n}$ in the neighborhoods of
$0$ and $4$.
\begin{definition}\label{de:bihomap}
Let $\varphi_{+}(x)$ be the branch of $\varphi(x)$ in the upper half
plane and let $\delta>0$ be sufficiently small, then the maps
$\tilde{f}_{n}$ is defined in a disc of radius $\delta$ center at
$b_{-,n}$ as follows.
\begin{equation}\label{eq:ftilde}
\begin{split}
\tilde{f}_{n}&=\left(\frac{n}{2}(\varphi_{+}(x)-\varphi_{+}(0))\right)^2,\quad
\gamma=1.
\end{split}
\end{equation}
In the above expressions, the map $\tilde{f}_{n}$ is defined to be
the analytic continuation of the expression in the right hand side
to the whole disc.

Similarly, the map $f_{n}$ in a neighborhood of $4$ is defined to be
\begin{equation}\label{eq:fmap}
\begin{split}
f_{n}&=\left(\frac{3}{2}n\varphi_{+}(x)\right)^{\frac{2}{3}}
\end{split}
\end{equation}
As in \cite{V}, these maps are conformal inside the small discs
around $b_{\pm,n}$, provided $\delta$ is sufficiently small. In
particular, they behave as follows as near $b_{\pm,n}$. (Recall that
$b_{-,n}=0$ and $b_{+,n}=4$ when $\gamma=1$.)
\begin{equation}\label{eq:behf}
\begin{split}
\tilde{f}_{n}&=-n^2x\tilde{\hat{f}}_n(x),\quad
f_{n}=\left(\frac{n}{4}\right)^{\frac{2}{3}}(x-4)\hat{f}_n(x),
\end{split}
\end{equation}
where the maps $\tilde{\hat{f}}_n(x)$ and $\hat{f}_n(x)$ are of the
form
\begin{equation*}
\tilde{\hat{f}}_n(x)=1+O(x),\quad\hat{f}_n(x)=1+O(x-4).
\end{equation*}
Note that $\tilde{\hat{f}}_n(x)$ and $\hat{f}_n(x)$ are independent
on $n$.
\end{definition}
We can now state the asymptotics of the Laguerre polynomials. These
can be found, for example, in \cite{V}. When $M-N=\alpha=O(1)$, the
Laguerre polynomials have the following asymptotic expressions in
different regions on $\mathbb{R}_+$ (See \cite{V}).
\begin{proposition}\label{pro:V}
Suppose $M-N=\alpha=O(1)$, then for $n=N-j$ and $j=O(1)$, there
exists sufficiently small $\delta>0$ and the Laguerre polynomials
$L_{n}(x)$ in (\ref{eq:laguerre}) have the following asymptotic
behavior on $\mathbb{R}_+$ as $n,N,M\rightarrow\infty$.
\begin{enumerate}
\item
Uniformly for $x\in(0,\delta]$,
\begin{equation}\label{eq:bessel}
\begin{split}
&L_{n}\left(\frac{n}{N}x\right)\left(w_0\left(\frac{n}{N}x\right)\right)^{\frac{1}{2}}=\frac{2\sqrt{\pi}(-1)^{n}(-\tilde{f_{n}})^{\frac{1}{4}}\left(i\kappa_{n-1}\right)^{-\frac{1}{2}}}{x^{\frac{1}{4}}(4-x)^{\frac{1}{4}}}
\Bigg(\sin\zeta J_{\alpha}\left(2(-\tilde{f}_{n})^{\frac{1}{2}}\right)\\
&\times(1+O(N^{-1}))+\cos\zeta
J_{\alpha}^{\prime}\left(2(-\tilde{f}_{n})^{\frac{1}{2}}\right)(1+O(N^{-1}))\Bigg),
\end{split}
\end{equation}
where $J_{\alpha}$ is the Bessel function and $\zeta$ is the
following
\begin{equation}\label{eq:zeta}
\zeta=\frac{1}{2}(\alpha+1)\arccos\left(\frac{x}{2}-1\right)-\frac{\pi\alpha}{2}.
\end{equation}
The branch cut of $\arccos(x)$ in the above is chosen to be
$(-\infty,-1]\cup[1,\infty)$. The constant $\kappa_{n}$ is given by
$\kappa_{n}=-\frac{2\pi i}{h_{n,0}}$.
\item Uniformly for $x\in[\delta,1-\delta]$,
\begin{equation}\label{eq:bulk}
\begin{split}
&L_{n}\left(\frac{n}{N}x\right)\left(w_0\left(\frac{n}{N}x\right)\right)^{\frac{1}{2}}=\frac{2\left(i\kappa_{n-1}\right)^{-\frac{1}{2}}}{x^{\frac{1}{4}}(4-x)^{\frac{1}{4}}}
\Bigg(\cos\left(\eta_++in\varphi_{+}-\frac{\pi}{4}\right)\\
&\times\left(1+O\left(N^{-1}\right)\right)
+\cos\left(\eta_-+in\varphi_{+}-\frac{\pi}{4}\right)O(N^{-1})\Bigg)
\end{split}
\end{equation}
where $\eta_{\pm}$ are given by
\begin{equation}\label{eq:eta}
\eta_{\pm}=\frac{1}{2}(\alpha\pm1)\arccos\left(\frac{x}{2}-1\right)
\end{equation}
\item Uniformly for $x\in[4-\delta,4+\delta]$,
\begin{equation}\label{eq:airy}
\begin{split}
&L_{n}\left(\frac{n}{N}x\right)\left(w_0\left(\frac{n}{N}x\right)\right)^{\frac{1}{2}}=\frac{2\sqrt{\pi}(i\kappa_{n-1})^{-\frac{1}{2}}}{x^{\frac{1}{4}}\left|x-4\right|^{\frac{1}{4}}}\Bigg(
\left|f_{n}\right|^{\frac{1}{4}}\cos\eta_+(x)
Ai(f_{n})\\&\times\left(1+O(N^{-1})\right)
-\left|f_{n}\right|^{-\frac{1}{4}}\sin\eta_+(x)
Ai^{\prime}(f_{n})(1+O(N^{-1}))\Bigg),
\end{split}
\end{equation}
\end{enumerate}
\end{proposition}
\begin{remark} Reader may notice that the asymptotic formula
presented in Proposition \ref{pro:V} is different from the ones in
\cite{V}. This is because the weight in \cite{V} that is relevant to
us is $w_0(x)=x^{\alpha}e^{-4nx}$ and a rescaling of the variable
from $x$ to $y=\frac{4n}{N}x$ is needed to obtain the formula in
Proposition \ref{pro:V} from the results in \cite{V}.
\end{remark}
We will now use the asymptotic formulae in Proposition \ref{pro:V}
to compute the skew products of the form
$\left<L_{N-j},L_{N-k}\right>_1$. We shall follow the ideas in
\cite{DG}, \cite{DGKV} and \cite{St}.

\subsection{Asymptotics for the skew products}
Throughout the analysis, we will assume that $t$ is of finite
distance from $[0,4)$. We will also need to consider the case where
$|t- 4|$ is of order $N^{-\frac{2}{3}}$. We shall organize the
results according to the different regions on $\mathbb{R}_+$. We
will divide $\mathbb{R}_+$ into four regions $(0,N^{-\frac{1}{2}}]$,
$[N^{-\frac{1}{2}},4-N^{-\varepsilon}]$,
$[4-N^{-\varepsilon},4+N^{-\varepsilon}]$ and
$[4+N^{-\varepsilon},\infty)$ for some $\varepsilon>0$. These
regions are called the Bessel region, the bulk region, the Airy
region and the exponential region.
\subsubsection{The Bessel Region}
Let us define $\hat{L}_n(x)$ to be
\begin{equation}\label{eq:hatL}
\begin{split}
\hat{L}_n(x)&=L_n\left(\frac{n}{N}x\right)w\left(\frac{n}{N}x\right)\left(i\kappa_{n-1}\right)^{\frac{1}{2}}\\
&=L_n\left(\frac{n}{N}x\right)\left(w_0\left(\frac{n}{N}x\right)\right)^{\frac{1}{2}}\left(i\kappa_{n-1}\right)^{\frac{1}{2}}
\left(\frac{nx}{N}\left(t-
\frac{nx}{N}\right)\right)^{-\frac{1}{2}}.
\end{split}
\end{equation}
From (\ref{eq:zeta}), we see that as $x\rightarrow 0$, $\sin\zeta$
and $\cos\zeta$ have the following behavior.
\begin{equation}\label{eq:sincoszeta}
\begin{split}
\sin\zeta=1+O(x),\quad \cos\zeta=\frac{\alpha+1}{2}\sqrt{x}(1+O(x)).
\end{split}
\end{equation}
Let us now express the integrals in the Bessel region in terms of
the variable $\tilde{f}_n$. From (\ref{eq:ftilde}), we have
\begin{equation}\label{eq:xtilde}
\begin{split}
x&=-\frac{\tilde{f}_n}{n^2}\left(1-\frac{1}{12}\frac{\tilde{f}_n}{n^2}+O\left(\frac{\tilde{f}_n^2}{n^4}\right)\right),\\
\left(\frac{nx}{N}\left(t-
\frac{nx}{N}\right)\right)^{-\frac{1}{2}}&=
\left(\frac{-\tilde{f}_n}{tn^2}\right)^{-\frac{1}{2}}\left(1+O\left(\frac{\tilde{f}_n}{n^2}\right)+O(n^{-1})\right)
\end{split}
\end{equation}
As $x<N^{-\frac{1}{2}}$ in the Bessel region, we see that the term
$\tilde{f}_n^2/n^4$ is of order $O(n^{-1})$ at most. Then from the
asymptotic formula (\ref{eq:bessel}), we have the following
asymptotic formula for $\hat{L}_n(x)$ in the Bessel region.
\begin{equation}\label{eq:behatL}
\begin{split}
\hat{L}_n(x)&=\frac{(-1)^n\sqrt{2\pi}n^{\frac{3}{2}}}{t^{\frac{1}{2}}}\Bigg(
\frac{J_{\alpha}\left(2(-\tilde{f}_n)^{\frac{1}{2}}\right)}{(-\tilde{f}_n)^{\frac{1}{2}}}
\left(1+\frac{k_{1,0}}{n}+\sum_{j=1}^{\infty}\frac{k_{1,j}\tilde{f}_n^{j}}{n^{2j}}\right)\\
&+\frac{\alpha+1}{2n}J_{\alpha}^{\prime}\left(2(-\tilde{f}_n)^{\frac{1}{2}}\right)
\left(1+\frac{k_{2,0}}{n}+\sum_{j=1}^{\infty}\frac{k_{2,j}\tilde{f}_n^{j}}{n^{2j}}\right)
\Bigg)
\end{split}
\end{equation}
for some constants $k_{1,j}$ and $k_{2,j}$ that are bounded in $n$.
By using the asymptotic formula for the Bessel function in
\cite{AbSt}, we obtain the following.
\begin{lemma}\label{le:besint}
Let $k\geq0$ $s>v\geq0$ and $\alpha>0$, then as
$s\rightarrow\infty$, we have
\begin{equation}\label{eq:besint}
\begin{split}
\int_v^{s}J_{\alpha}(u)u^k\D
u=O(1)+O\left(s^{k-\frac{1}{2}}\right),\quad
\int_v^sJ_{\alpha}^{\prime}(u)u^k\D
u=O(1)+O\left(s^{k-\frac{1}{2}}\right).
\end{split}
\end{equation}
Let $s_1=O\left(s\right)$ as $s\rightarrow\infty$, and let
$B_1=J_{\alpha}$, $B_2=J^{\prime}_{\alpha}$, then as
$s\rightarrow\infty$, we have the following for the double
integrals.
\begin{equation}\label{eq:besdou}
\int_{0}^{s}B_i(u)u^k\int_u^{s_1}B_j(v)v^l\D v\D
u=O\left(s^{k+l}\right),\quad i,j=1,2,\quad k+l>0.
\end{equation}
\end{lemma}
\begin{proof} Integrating the first formula by parts, we obtain
\begin{equation*}
\int_0^{s}J_{\alpha}(u)u^k\D u=\left(-\int_u^sJ_{\alpha}(v)\D
vu^k\right)_v^{s}+k\int_v^s\left(\int_u^sJ_{\alpha}(v)\D
v\right)u^{k-1}\D u
\end{equation*}
By (9.2.1) and (11.4.17) of \cite{AbSt}, we see that
$\int_u^{s}J_{\alpha}(v)\D v=O(s^{-\frac{1}{2}})$ as $u\rightarrow
s$ and $s\rightarrow\infty$. The function $\int_u^{s}J_{\alpha}(v)\D
v$ is therefore bounded and behaves as $O(u^{-\frac{1}{2}})$ as
$u\rightarrow s$. Therefore we have
\begin{equation*}
\int_0^{s}J_{\alpha}(u)u^k\D
u=k\int_v^s\left(\int_u^sJ_{\alpha}(v)\D v\right)u^{k-1}\D
u+O(1)+O(s^{k-\frac{1}{2}})=O(1)+O(s^{k-\frac{1}{2}}).
\end{equation*}
This proves the first equation in (\ref{eq:besint}). Integrating the
second equation by parts and using $J_{\alpha}(0)=0$ for $\alpha>0$,
we obtain
\begin{equation*}
\int_v^sJ_{\alpha}^{\prime}(u)u^k\D
u=J_{\alpha}(s)s^k-J_{\alpha}(v)v^k-k\int_{v}^sJ_{\alpha}(u)u^{k-1}\D
u.
\end{equation*}
By using the following estimates for any $C>0$,
\begin{equation}\label{eq:estbes}
\begin{split}
\sup_{y\in[0,C)}|y^{-\alpha}J_{\alpha}(y)|<\infty,\quad
\sup_{y\in[C,\infty)}|\sqrt{y}J_{\alpha}(y)|<\infty,\\
\sup_{y\in[0,C)}|y^{-\alpha+1}J_{\alpha}^{\prime}(y)|<\infty,\quad
\sup_{y\in[C,\infty)}|\sqrt{y}J_{\alpha}^{\prime}(y)|<\infty,
\end{split}
\end{equation}
we obtain the second equation in (\ref{eq:besint}). The estimate
(\ref{eq:besdou}) now follows immediately from (\ref{eq:besint}) and
(\ref{eq:estbes}).
\end{proof}
We can now compute the integrals in the skew products in the Bessel
region.
\begin{proposition}\label{pro:besin}
The single integral involving $L_n(x)w(x)$ in the Bessel region is
given by
\begin{equation}\label{eq:singleB}
\begin{split}
\left(i\kappa_{n-1}\right)^{\frac{1}{2}}\int_{0}^{N^{-\frac{1}{2}}}L_n(x)w(x)\D
x&=\mathcal{I}_{1,n}=\frac{(-1)^n\sqrt{2\pi}}{\sqrt{nt}}+O(n^{-\frac{7}{8}}).
\end{split}
\end{equation}
The double integral in the Bessel region is given by
\begin{equation}\label{eq:doubleB}
\begin{split}
&i\sqrt{\kappa_{n-1}\kappa_{m-1}}\int_0^{N^{-\frac{1}{2}}}L_n(x)w(x)\int_x^{N^{-\frac{1}{2}}}L_m(y)w(y)\D
x\D
y\\
&=\mathcal{J}_{1}=\frac{(-1)^{m+n}2\pi}{nt}+O\left(n^{-\frac{5}{4}}\right).
\end{split}
\end{equation}
\end{proposition}
\begin{proof}Let us first prove the single integral. We have, by
(\ref{eq:hatL}),
\begin{equation*}
\left(i\kappa_{n-1}\right)^{\frac{1}{2}}\int_{0}^{N^{-\frac{1}{2}}}L_n(x)w(x)\D
x=\int_0^{N^{-\frac{1}{2}}}\hat{L}_n\left(\frac{N}{n}x\right)\D
x=\frac{n}{N}\int_0^{\frac{\sqrt{N}}{n}}\hat{L}_n(x)\D x.
\end{equation*}
From (\ref{eq:behatL}), we have
\begin{equation*}
\begin{split}
\int_0^{\frac{\sqrt{N}}{n}}\hat{L}_n(x)\D x=
\frac{(-1)^n\sqrt{2\pi}}{\sqrt{tn}}
\Bigg(\int_0^{u_+}\Bigg(J_{\alpha}(u)\left(1+\frac{k_{1,0}}{n}+\sum_{j=1}^{\infty}\frac{(-1)^jk_{1,j}u^{2j}}{(2n)^{2j}}\right)\\
+\frac{\alpha+1}{4}J_{\alpha}^{\prime}(u)\left(\frac{u}{n}+\frac{k_{2,0}u}{n^2}+\sum_{j=1}^{\infty}\frac{(-1)^jk_{2,j}u^{2j+1}}{2^{2j}n^{2j+1}}\right)\Bigg)
\D u\Bigg),
\end{split}
\end{equation*}
where we have changed the integration variable from $x$ to
$u=2(-\tilde{f}_n)^{\frac{1}{2}}$ and $u_+$ is the corresponding
upper limit
$u_+=2\sqrt{-\tilde{f}_n\left(\frac{\sqrt{N}}{n}\right)}=O\left(n^{\frac{3}{4}}\right)$.
By Lemma \ref{le:besint}, we have
\begin{equation*}
\begin{split}
\int_0^{\frac{\sqrt{N}}{n}}\hat{L}_n(x)\D x&=
\frac{(-1)^n\sqrt{2\pi}}{\sqrt{tn}} \Bigg(
\int_{0}^{u_+}J_{\alpha}(u)\D u+O\left(n^{-\frac{5}{8}}\right)\Bigg)
\end{split}
\end{equation*}
By using the asymptotic formula for the Bessel function
(\cite{AbSt}, (9.2.5), (9.2.9) and (9.2.10)), we have obtain the
following estimate
\begin{equation}\label{eq:jinf}
\int_{u_+}^{\infty}J_{\alpha}(u)\D
u=O\left(u_+^{-\frac{1}{2}}\right)
\end{equation}
By using this and the fact that $\int_{0}^{\infty}J_{\alpha}(u)\D
u=1$, we arrive at (\ref{eq:singleB}).

Let us now compute the double integral. Let $n=N-j$ and $m=N-k$,
where $j$ and $k$ are finite. As in the single integrals, we will
change the integration variables into $u=2\sqrt{-\tilde{f}_n}$ and
$v=2\sqrt{-\tilde{f}_m}$. Let
$v_+=2\sqrt{-\tilde{f}_m(\sqrt{N}/n)}=O\left(m^{\frac{3}{4}}\right)$
be the upper limit in the variable $v$ and $\nu(u)$ be the value of
$v$ when $y=\frac{n}{m}x$. Then by using (\ref{eq:xtilde}), we have
\begin{equation}\label{eq:nubes}
\begin{split}
\nu(u)=\sqrt{\frac{m}{n}}u\left(1+\frac{u^2}{24m}\frac{n-m}{n^2}+O\left(\frac{u^3}{n^4}\right)\right)
\end{split}
\end{equation}
Now the double integral is given by
\begin{equation*}
i\sqrt{\kappa_{n-1}\kappa_{m-1}}\int_0^{N^{-\frac{1}{2}}}L_n(x)w(x)\int_x^{N^{-\frac{1}{2}}}L_m(y)w(y)\D
x\D
y=\frac{mn}{N^2}\int_0^{\frac{\sqrt{N}}{n}}\hat{L}_n(x)\int_{\frac{n}{m}x}^{\frac{\sqrt{N}}{m}}\hat{L}_m(y)\D
y\D x.
\end{equation*}
By using Lemma \ref{le:besint} and (\ref{eq:estbes}). We see that it
is of the order
\begin{equation*}
\begin{split}
&\int_0^{\frac{\sqrt{N}}{n}}\hat{L}_n(x)\int_{\frac{n}{m}x}^{\frac{\sqrt{N}}{m}}\hat{L}_m(y)\D
y\D x=\frac{(-1)^{m+n}2\pi
}{\sqrt{nm}t}\int_0^{u_+}J_{\alpha}(u)\int_{\nu(u)}^{v_+}J_{\alpha}(v)\D
v\D u+O\left(N^{-\frac{5}{4}}\right).
\end{split}
\end{equation*}
To compute the first term, note that, since
\begin{equation*}
\nu^{\prime}(u)J_{\alpha}(\nu(u))=-\frac{d}{du}\left(\int_{\nu(u)}^{v_+}J_{\alpha}(v)\D
v\right),
\end{equation*}
we have, by (\ref{eq:nubes}) and mean value theorem, the following
\begin{equation*}
\begin{split}
\left(1+O(n^{-1})\right)\left(
J_{\alpha}(u)+J_{\alpha}^{\prime}(\xi_u)\left(\frac{(m-n)u}{2n}+O\left(n^{-\frac{3}{4}}\right)\right)\right)
=-\frac{d}{du}\left(\int_{\nu(u)}^{v_+}J_{\alpha}(v)\D v\right),
\end{split}
\end{equation*}
where $\xi_u$ is between $u$ and $\nu(u)$. As $J_{\alpha}(u)$ is
bounded, we have
\begin{equation*}
\begin{split}
&\int_0^{u_+}J_{\alpha}(u)\int_{\nu(u)}^{v_+}J_{\alpha}(v)\D v\D
u=\frac{1}{2}\left(\left(\int_{0}^{v_+}J_{\alpha}(v)\D
v\right)^2-\left(\int_{\nu(u_+)}^{v_+}J_{\alpha}(v)\D
v\right)^2\right)\\
&-\int_0^{u_+}\left(J_{\alpha}^{\prime}(\xi_u)\left(\frac{(m-n)u}{2n}+O\left(n^{-\frac{3}{4}}\right)\right)+O\left(n^{-1}\right)
\right)\int_{\nu(u)}^{v_+}J_{\alpha}(v)\D v\D u.
\end{split}
\end{equation*}
By using Lemma \ref{le:besint} and (\ref{eq:jinf}), we obtain
\begin{equation*}
\begin{split}
&\int_{\nu(u_+)}^{v_+}J_{\alpha}(v)\D
v=O\left(n^{-\frac{3}{8}}\right),\quad \int_{0}^{v_+}J_{\alpha}(v)\D
v=1+O\left(n^{-\frac{3}{8}}\right),\\
&\int_0^{u_+}\left(J_{\alpha}^{\prime}(\xi_u)\left(\frac{(m-n)u}{2n}+O\left(n^{-\frac{3}{4}}\right)\right)+O\left(n^{-1}\right)
\right)\int_{\nu(u)}^{v_+}J_{\alpha}(v)\D v\D
u=O\left(n^{-\frac{1}{4}}\right).
\end{split}
\end{equation*}
This gives us
\begin{equation*}
\begin{split}
&\int_0^{u_+}J_{\alpha}(u)\int_{\nu(u)}^{v_+}J_{\alpha}(v)\D v\D
u=\frac{1}{2}+O\left(n^{-\frac{1}{4}}\right).
\end{split}
\end{equation*}
This proves the proposition.
\end{proof}
This concludes the analysis in the Bessel region, we will now
consider the integrals in the bulk region.
\subsubsection{The Bulk region}
First recall the following matching formula in the region
$\left[\frac{\sqrt{N}}{2n},\delta\right]$ from \cite{DGKV}.
\begin{lemma}\label{le:match}(Proposition 4.4 of \cite{DGKV})
\begin{enumerate}
\item Uniformly for $x\in\left[\frac{\sqrt{N}}{2n},\delta\right]$, as
$n\rightarrow\infty$,
\begin{equation}\label{eq:match}
\begin{split}
&\left(-\tilde{f}_n\right)^{\frac{1}{4}}\left(\sin\zeta
J_{\alpha}\left(2\left(-\tilde{f}_n\right)^{\frac{1}{2}}\right)
+\cos\zeta
J_{\alpha}^{\prime}\left(2\left(-\tilde{f}_n\right)^{\frac{1}{2}}\right)\right)\\
&=\frac{(-1)^n}{\sqrt{\pi}}\left(\cos F_n(x)+\iota_n\sin
F_n(x)\right)+O(n^{-1}),
\end{split}
\end{equation}
where $\iota_n=\frac{4\alpha^2-1}{16(-\tilde{f}_n)^{\frac{1}{2}}}$
and $F_n=\eta_++in\varphi_{+}-\frac{\pi}{4}$.
\item Uniformly for
$x\in\left[4-\varepsilon,\frac{N}{n}c_-\right]$, as
$n\rightarrow\infty$, we have
\begin{equation}\label{eq:matchairy}
\begin{split}
\cos\eta_+|f_n|^{\frac{1}{4}}Ai(f_n)-\sin\eta_+|f_n|^{-\frac{1}{4}}Ai^{\prime}(f_n)
= \frac{1}{\sqrt{\pi}}\cos
F_n+O\left(\frac{1}{n(4-x)^{\frac{3}{2}}}\right).
\end{split}
\end{equation}
\end{enumerate}
\end{lemma}
This shows that throughout the bulk region, the function
$\hat{L}_n(x)$ in (\ref{eq:hatL}) has the following asymptotic
behavior.
\begin{equation}\label{eq:buhatL}
\begin{split}
\hat{L}_n(x)&=\frac{2\sqrt{N}}{\sqrt{n}x^{\frac{3}{4}}(4-x)^{\frac{1}{4}}\left(t-
\frac{n}{N}x\right)^{\frac{1}{2}}} \left(\cos F_n(x)+E\right)
\end{split}
\end{equation}
uniformly for $x\in\left[\frac{\sqrt{N}}{2n},\frac{N}{n}c_-\right]$,
where $E$ is an error term of order
$O\left(n^{-1}(4-x)^{-\frac{3}{2}}\right)+O\left(n^{-1}x^{-\frac{1}{2}}\right)$.
From this we can now compute the single integral.
\begin{proposition}\label{pro:bulksing}
Let $c_-=4-N^{-\varepsilon}$ such that
$\varepsilon\leq\frac{1}{20}$. Suppose the distance from the point
$t$ to the interval $\left[\frac{\sqrt{N}}{n},\frac{N}{n}C_-\right]$
is greater than $CN^{-\varepsilon}$ for some $C>0$ and the distance
from $t$ to $0$ is finite. Then the single integral in the bulk
region is of the following order.
\begin{equation}\label{eq:bulksing}
\begin{split}
&\left(i\kappa_{n-1}\right)^{\frac{1}{2}}\int_{N^{-\frac{1}{2}}}^{c_-}L_n(x)w(x)\D
x=\mathcal{I}_{2,n}=O\left(n^{-\frac{7}{8}}\right)
\end{split}
\end{equation}
\end{proposition}
\begin{proof} By (\ref{eq:hatL}), we have
\begin{equation*}
\left(i\kappa_{n-1}\right)^{\frac{1}{2}}\int_{N^{-\frac{1}{2}}}^{c_-}L_n(x)w(x)\D
x=\frac{n}{N}\int_{\frac{\sqrt{N}}{n}}^{\frac{N}{n}c_-}\hat{L}_n(x)\D
x.
\end{equation*}
From this and (\ref{eq:buhatL}), we obtain
\begin{equation}\label{eq:sing1}
\begin{split}
\int_{\frac{\sqrt{N}}{n}}^{\frac{N}{n}c_-}\hat{L}_n(x)\D
x&=2\sqrt{\frac{n}{N}}\int_{\frac{\sqrt{N}}{n}}^{\frac{N}{n}c_-}
\frac{\cos F_n(x)+E}{x^{\frac{3}{4}}(4-x)^{\frac{1}{4}}\left(t-
\frac{n}{N}x\right)^{\frac{1}{2}}} \D x.
\end{split}
\end{equation}
Since $F_n=\eta_++in\varphi_{+}-\frac{\pi}{4}$, we see that its
derivative is given by
\begin{equation}\label{eq:Fprime}
F_n^{\prime}=-\frac{(\alpha+1)}{2\sqrt{x(4-x)}}+\frac{in}{2}\left(\sqrt{\frac{x-4}{x}}\right)_+=
-\frac{\alpha+1+n(4-x)}{2\sqrt{x(4-x)}}
\end{equation}
where the plus subscript indicates the boundary value as
$\sqrt{x-4}/\sqrt{x}$ approaches the real axis in the upper half
plane.

This gives us the following estimate of the order of
$1/F_n^{\prime}$.
\begin{equation}\label{eq:orderFn}
\frac{1}{F_n^{\prime}}=-\frac{2\sqrt{x}}{n\sqrt{4-x}}\left(1+O\left(\frac{1}{n(4-x)}\right)\right)
\end{equation}
We can now integrate the first term in (\ref{eq:sing1}) by parts to
obtain
\begin{equation*}
\begin{split}
&\int_{\frac{\sqrt{N}}{n}}^{\frac{N}{n}c_-} \frac{\cos
F_n(x)}{x^{\frac{3}{4}}(4-x)^{\frac{1}{4}}\left(t-
\frac{n}{N}x\right)^{\frac{1}{2}}} \D x=\frac{\sin
F_n}{F_n^{\prime}x^{\frac{3}{4}}(4-x)^{\frac{1}{4}}\left(t-
\frac{n}{N}x\right)^{\frac{1}{2}}}\Bigg|_{\frac{\sqrt{N}}{n}}^{\frac{N}{n}c_-}
\\
&-\int_{\frac{\sqrt{N}}{n}}^{\frac{N}{n}c_-} \sin
F_n(x)\frac{d}{dx}\left(\frac{1}{F_n^{\prime}x^{\frac{3}{4}}(4-x)^{\frac{1}{4}}\left(t-
\frac{n}{N}x\right)^{\frac{1}{2}}} \right)\D x
\end{split}
\end{equation*}
As the distance from $t$ to the bulk region is at least of order
$O\left(N^{-\varepsilon}\right)$, the first term in the above
equation is of order
$O\left(n^{\frac{5\varepsilon}{4}-1}\right)+O\left(n^{-\frac{7}{8}}\right)$,
which is of order $O\left(n^{-\frac{7}{8}}\right)$ as
$\varepsilon\leq 1/20$. By repeating the integration by parts
procedure to the second term, one can verify that it is of order at
most $O\left(n^{\frac{11\varepsilon}{4}-2}\right)$. This gives
\begin{equation*}
\begin{split}
\int_{\frac{\sqrt{N}}{n}}^{\frac{N}{n}c_-} \frac{\cos
F_n(x)}{x^{\frac{3}{4}}(4-x)^{\frac{1}{4}}\left(t-
\frac{n}{N}x\right)^{\frac{1}{2}}} \D
x&=O\left(n^{-\frac{7}{8}}\right),
\end{split}
\end{equation*}
Since the error term $E$ is of order
$O\left(n^{-1}(4-x)^{-\frac{3}{2}}\right)+O\left(n^{-1}x^{-\frac{1}{2}}\right)$,
we obtain the following estimate for contribution from $E$.
\begin{equation*}
2\sqrt{\frac{n}{N}}\int_{\frac{\sqrt{N}}{n}}^{\frac{N}{n}c_-}
\frac{E}{x^{\frac{3}{4}}(4-x)^{\frac{1}{4}}\left(t-
\frac{n}{N}x\right)^{\frac{1}{2}}} \D
x=O\left(n^{\frac{5\varepsilon}{4}-1}\right).
\end{equation*}
This concludes the proof of the proposition.
\end{proof}
In particular, from the proof of Proposition \ref{pro:bulksing}, we
see that
\begin{equation}\label{eq:bulkint}
\begin{split}
\int_{y}^{\frac{N}{n}c_-}\frac{\cos
F_n(x)}{x^{\frac{3}{4}}(4-x)^{\frac{1}{4}}\left(t-
\frac{n}{N}x\right)^{\frac{1}{2}}}\D
x=O\left(n^{-\frac{7}{8}}\right),\quad
y\in\left[\frac{\sqrt{N}}{n},\frac{N}{n}c_-\right].
\end{split}
\end{equation}
Let us now proceed to compute the double integrals. To begin with,
we have the following lemma that will help us to simplify the
results. (Proposition 5.12 in \cite{DGKV})
\begin{lemma}\label{le:diarg}
The function $\theta(x)$ defined by
\begin{equation}\label{eq:theta}
\theta(x)=\frac{1}{2}\int_{4}^x\sqrt{\frac{4-s}{s}}\D s=-i\varphi_+,
\quad x\in[0,4]
\end{equation}
satisfies the following
\begin{equation}\label{eq:diarg}
\theta(x)-x\theta^{\prime}(x)=-\arccos\left(\frac{x}{2}-1\right)
\end{equation}
\end{lemma}
\begin{proof} By differentiating $\theta(x)$ twice we see that
\begin{equation*}
x\theta^{\prime\prime}(x)=-\frac{1}{\sqrt{x(4-x)}}.
\end{equation*}
As $\arccos(x)$ is given by
\begin{equation}\label{eq:arccos}
\arccos(x)=\int_x^1\frac{1}{\sqrt{1-s^2}}\D s,
\end{equation}
we see that
\begin{equation*}
\begin{split}
\frac{d}{dx}\left(\theta(x)-x\theta^{\prime}(x)\right)=-\frac{d}{dx}\arccos\left(\frac{x}{2}-1\right).
\end{split}
\end{equation*}
Integrating the above, we obtain
\begin{equation*}
\theta(x)-x\theta^{\prime}(x)=-\arccos\left(\frac{x}{2}-1\right)+C.
\end{equation*}
For some integration constant $C$. By evaluating the above equation
at $x=4$, we see that $C=0$. This proves the lemma.
\end{proof}
The lemma implies the following. (Proposition 5.13, \cite{DGKV})
\begin{lemma}\label{le:diffmn}
Let $n=N-j$, $m=N-k$ where $j$ and $k$ are finite integers, then for
$x\in\left[\frac{\sqrt{N}}{n},\frac{N}{n}c_-\right]$, we have
\begin{equation}\label{eq:diffmn}
F_m\left(x\frac{n}{m}\right)-F_n(x)=\left(m-n\right)\arccos\left(\frac{x}{2}-1\right)+O\left(m^{\frac{\varepsilon}{2}-1}\right).
\end{equation}
\end{lemma}
\begin{proof} As in \cite{DGKV}, let us write the left hand side of
(\ref{eq:diffmn}) as
\begin{equation*}
F_m\left(x\frac{n}{m}\right)-F_n(x)=\left(F_m\left(x\frac{n}{m}\right)-F_m(x)\right)+\left(F_m(x)-F_n(x)\right)
\end{equation*}
By repeat application of the mean value theorem, we see that the
first term on the right hand side is given by
\begin{equation*}
F_m\left(x\frac{n}{m}\right)-F_m(x)=\frac{x(n-m)}{m}F^{\prime}_m(x)+\frac{x^2(n-m)^2}{2m^2}F^{\prime\prime}(\xi)
\end{equation*}
for some $\xi\in[x,xn/m]$. From (\ref{eq:Fprime}), we see that for
$x\in\left[\frac{\sqrt{N}}{n},\frac{N}{n}c_-\right]$, we have
\begin{equation*}
\begin{split}
\frac{x(n-m)}{m}F_m^{\prime}(x)&=(m-n)x\theta^{\prime}(x)+O\left(\frac{1}{m\sqrt{4-x}}\right),\\
\frac{x^2(m-n)^2}{2m^2}F^{\prime\prime}(\xi)&=O\left(\frac{1}{m^2(4-x)^{\frac{3}{2}}}\right)+O\left(\frac{1}{m\sqrt{4-x}}\right).
\end{split}
\end{equation*}
As $4-x=N^{-\varepsilon}$ and $\varepsilon\leq1/20$, we have
\begin{equation*}
\begin{split}
F_m\left(x\frac{n}{m}\right)-F_n(x)&=(m-n)x\theta^{\prime}(x)+\left(F_m(x)-F_n(x)\right)+O\left(m^{\frac{\varepsilon}{2}-1}\right)\\
&=(m-n)x\theta^{\prime}(x)+\left(n-m\right)\theta(x)+O\left(m^{\frac{\varepsilon}{2}-1}\right)\\
&=\left(m-n\right)\arccos\left(\frac{x}{2}-1\right)+O\left(m^{\frac{\varepsilon}{2}-1}\right),
\end{split}
\end{equation*}
where the last equality follows from lemma \ref{le:diarg}.
\end{proof}
We can now compute the double integrals in the bulk region.
\begin{proposition}
The double integral inside the bulk region is of the following
order.
\begin{equation}\label{eq:doubulk0}
\begin{split}
&i\sqrt{\kappa_{n-1}\kappa_{m-1}}\int_{N^{-\frac{1}{2}}}^{c_-}L_n(x)w(x)\int_x^{c_-}L_m(y)w(y)\D
x\D
y=\mathcal{J}_2\\
&=\frac{4}{N}\int_{\frac{\sqrt{N}}{n}}^{\frac{N}{n}c_-}
\frac{\sin\left((n-m)\arccos\left(\frac{x}{2}-1\right)\right)}{x(4-x)\left(t-
x\right)}\D x+O\left(n^{-\frac{7}{4}}\right)
\end{split}
\end{equation}
where $c_-=4-N^{-\varepsilon}$.
\end{proposition}
\begin{proof}
The proof is similar to the computation in \cite{DGKV} and
\cite{DG}. We have
\begin{equation*}
\mathcal{J}_2=\frac{mn}{N^2}\int_{\frac{\sqrt{N}}{n}}^{\frac{N}{n}c_-}\hat{L}_n(x)
\int_{\frac{n}{m}x}^{\frac{N}{n}c_-}\hat{L}_m(y)\D x\D y,
\end{equation*}
Then by using (\ref{eq:buhatL}), we see that the double integral is
given by the following.
\begin{equation*}
\begin{split}
\mathcal{J}_2
&=\frac{4\sqrt{mn}}{N}\int_{\frac{\sqrt{N}}{n}}^{\frac{N}{n}c_-}\frac{\cos
F_n(x)+E(x)}{x^{\frac{3}{4}}(4-x)^{\frac{1}{4}}\left(t-\frac{n}{N}
x\right)^{\frac{1}{2}}}
\int_{\frac{n}{m}x}^{\frac{N}{m}c_-}\frac{\cos
F_m(y)+E(y)}{y^{\frac{3}{4}}(4-y)^{\frac{1}{4}}\left(t-\frac{m}{N}
y\right)^{\frac{1}{2}}} \D x\D y
\end{split}
\end{equation*}
First let us compute the error terms. By changing the order of the
integration, we have
\begin{equation*}
\begin{split}
&\int_{\frac{\sqrt{N}}{n}}^{\frac{N}{n}c_-}\frac{\cos
F_n(x)}{x^{\frac{3}{4}}(4-x)^{\frac{1}{4}}\left(t-\frac{n}{N}
x\right)^{\frac{1}{2}}}
\int_{\frac{n}{m}x}^{\frac{N}{m}c_-}\frac{E(y)}{y^{\frac{3}{4}}(4-y)^{\frac{1}{4}}\left(t-\frac{n}{N} y\right)^{\frac{1}{2}}}\D y\D x\\
&=\int_{\frac{\sqrt{N}}{m}}^{\frac{N}{m}c_-}\frac{E(y)}{y^{\frac{3}{4}}(4-y)^{\frac{1}{4}}\left(t-\frac{n}{N}
y\right)^{\frac{1}{2}}}
\int_{\frac{\sqrt{N}}{n}}^{\frac{m}{n}y}\frac{\cos
F_n(x)}{x^{\frac{3}{4}}(4-x)^{\frac{1}{4}}\left(t-\frac{n}{N}
x\right)^{\frac{1}{2}}}
\D x\D y\\
&=\int_{\frac{\sqrt{N}}{m}}^{\frac{N}{m}c_-}\frac{E(y)}{y^{\frac{3}{4}}(4-y)^{\frac{1}{4}}\left(t-\frac{n}{N}
y\right)^{\frac{1}{2}}} O\left(n^{-\frac{7}{8}}\right)\D
y=O\left(n^{-\frac{7}{4}}\right),
\end{split}
\end{equation*}
where the last equality follows from (\ref{eq:bulkint}). The order
of the other error term can be estimated similarly. Let us now
consider the leading order term. Integrating by parts, we obtain
\begin{equation*}
\begin{split}
\mathcal{J}_{20}&=\int_{\frac{\sqrt{N}}{n}}^{\frac{N}{n}c_-}\frac{\cos
F_n(x)}{x^{\frac{3}{4}}(4-x)^{\frac{1}{4}}\left(t-\frac{n}{N}
x\right)^{\frac{1}{2}}}
\int_{\frac{n}{m}x}^{\frac{N}{m}c_-}\frac{\cos
F_m(y)}{y^{\frac{3}{4}}(4-y)^{\frac{1}{4}}\left(t-\frac{m}{N}
y\right)^{\frac{1}{2}}}\D
y\D x\\
&=\left(\frac{m}{n}\right)^{\frac{3}{4}}\int_{\frac{\sqrt{N}}{n}}^{\frac{N}{n}c_-}\frac{\cos
F_n(x)\sin
F_m\left(\frac{n}{m}x\right)}{F_m^{\prime}\left(\frac{n}{m}x\right)x^{\frac{3}{2}}(4-x)^{\frac{1}{4}}\left(4-\frac{n}{m}x\right)^{\frac{1}{4}}\left(t-\frac{n}{N}
x\right)}
\D x\\
&-\int_{\frac{\sqrt{N}}{n}}^{\frac{N}{n}c_-}\frac{\cos
F_n(x)}{x^{\frac{3}{4}}(4-x)^{\frac{1}{4}}\left(t-\frac{n}{N}
x\right)^{\frac{1}{2}}} \int_{\frac{n}{m}x}^{\frac{N}{m}c_-}\sin
F_m(y)\frac{d}{dy}\left(\frac{1}
{F_m^{\prime}y^{\frac{3}{4}}(4-y)^{\frac{1}{4}}\left(t-\frac{m}{N}
y\right)^{\frac{1}{2}}}\right)
\\&+O\left(n^{-\frac{7}{4}}\right)
\end{split}
\end{equation*}
To evaluate the second term, note that for $y$ in the interval
$\left[\frac{\sqrt{N}}{m},\frac{N}{m}c_-\right]$, we have
\begin{equation*}
\sin F_m(y)\frac{\D}{\D y}\left(\frac{1}
{F_m^{\prime}y^{\frac{3}{4}}(4-y)^{\frac{1}{4}}\left(t-\frac{m}{N}
y\right)^{\frac{1}{2}}}\right)=
O\left(\frac{1}{m^{1-\varepsilon}(4-y)^{\frac{5}{4}}y^{\frac{5}{4}}}\right)
\end{equation*}
Hence by interchanging the order of integration, we obtain
\begin{equation*}
\begin{split}
&\int_{\frac{\sqrt{N}}{n}}^{\frac{N}{n}c_-}\frac{\cos
F_n(x)}{x^{\frac{3}{4}}(4-x)^{\frac{1}{4}}\left(t-\frac{n}{N}
x\right)^{\frac{1}{2}}} \int_{\frac{n}{m}x}^{\frac{N}{m}c_-}\sin
F_m(y)\frac{d}{dy}\left(\frac{1}
{F_m^{\prime}y^{\frac{3}{4}}(4-y)^{\frac{1}{4}}\left(t-\frac{m}{N} y\right)^{\frac{1}{2}}}\right)\\
&=O\left(m^{-\frac{7}{4}}\right)+O\left(m^{-\frac{7}{8}-1+\frac{5\varepsilon}{4}}\right)=O\left(m^{-\frac{7}{4}}\right)
\end{split}
\end{equation*}
To compute the first term of $\mathcal{J}_{20}$, let us change the
integration variable from $x$ to $\frac{n}{m}x$. Then we obtain
\begin{equation*}
\begin{split}
\mathcal{J}_{20}=\left(\frac{m}{n}\right)^{\frac{1}{4}}\int_{\frac{\sqrt{N}}{m}}^{\frac{N}{m}c_-}\frac{\cos
F_n\left(\frac{m}{n}x\right)\sin
F_m\left(x\right)}{F_m^{\prime}\left(x\right)x^{\frac{3}{2}}(4-x)^{\frac{1}{4}}\left(4-\frac{m}{n}x\right)^{\frac{1}{4}}\left(t-\frac{m}{N}
x\right)} \D x+O(m^{-\frac{7}{4}})
\end{split}
\end{equation*}
Now by (\ref{eq:orderFn}) and the fact that both $4-x$ and
$t-\frac{n}{N} x$ are of order at least $n^{-\varepsilon}$, we
obtain the follow estimate
\begin{equation*}
\begin{split}
\mathcal{J}_{20}=-2\left(\frac{m}{n}\right)^{\frac{1}{4}}\int_{\frac{\sqrt{N}}{m}}^{\frac{N}{m}c_-}\frac{\cos
F_n\left(\frac{m}{n}x\right)\sin F_m\left(x\right)}{mx(4-x)\left(t-
x\right)}\D x+O(m^{-2+3\varepsilon})+O\left(m^{-\frac{7}{4}}\right).
\end{split}
\end{equation*}
To evaluate the integral, we use the angle addition formula for sine
to obtain
\begin{equation*}
\begin{split}
\mathcal{J}_{20}&=-\left(\frac{m}{n}\right)^{\frac{1}{4}}\left(\int_{\frac{\sqrt{N}}{m}}^{\frac{N}{m}c_-}\frac{\sin\left(
F_m(x)-F_n\left(\frac{m}{n}x\right)\right)}{mx(4-x)\left(t-
x\right)}\D x+
\int_{\frac{\sqrt{N}}{m}}^{\frac{N}{m}c_-}\frac{\sin\left(
F_n\left(\frac{m}{n}x\right)+
F_m\left(x\right)\right)}{mx(4-x)\left(t- x\right)}\D
x\right)\\
&+O(m^{-\frac{7}{4}}).
\end{split}
\end{equation*}
The first term can be simplified using (\ref{eq:diffmn}) while
integration by parts shows that the second term is of order
$O\left(m^{-\frac{7}{4}}\right)$. This gives
\begin{equation*}
\mathcal{J}_{20}=\int_{\frac{\sqrt{N}}{m}}^{\frac{N}{m}c_-}
\frac{\sin\left((n-m)\arccos\left(\frac{x}{2}-1\right)\right)}{mx(4-x)(t-
x)}\D x+O(m^{-\frac{4}{7}}).
\end{equation*}
As $\frac{N}{n}c_--\frac{N}{m}c_-$ is of order $m^{-1}$, we see that
\begin{equation*}
\int_{\frac{N}{n}c_-}^{\frac{N}{m}c_-}\frac{\sin\left((m-n)\arccos\left(\frac{x}{2}-1\right)\right)}{mx(4-x)(t-
x)}\D x=O\left(m^{-2+2\varepsilon}\right).
\end{equation*}
Similarly, we can change the lower limit to $\frac{\sqrt{N}}{n}$ and
add an error term of order $O\left(m^{-2}\right)$. This proves the
proposition.
\end{proof}
As we are only going to consider the values of $m$ and $n$ with
$m-n=\pm 1$ or $\pm 2$, we can simplify the double integrals
further. First let us consider the case when $m-n=\pm 1$. A simple
calculation shows that
\begin{equation}\label{eq:sinarc}
\sin\left(\arccos\left(\frac{x}{2}-1\right)\right)=\frac{\sqrt{x(4-x)}}{2}.
\end{equation}
From this we obtain
\begin{equation}\label{eq:sinarc2}
\begin{split}
\sin\left(2\arccos\left(\frac{x}{2}-1\right)\right)&=\frac{\sqrt{x(4-x)}(x-2)}{2}.
\end{split}
\end{equation}
We can now use these and residue calculation to compute the double
integrals.
\begin{proposition}\label{pro:doubulk}
Let $\mathcal{F}_l$ be
\begin{equation}\label{eq:fl}
\mathcal{F}_l=\int_{\frac{N}{n}c_-}^{4}\frac{4\sin\left(l\arccos\left(\frac{x}{2}-1\right)\right)}{Nx(4-x)\left(t-
x\right)}\D x
\end{equation}
Then the asymptotics of the double integrals in the bulk region are
given by the followings.
\begin{equation}\label{eq:doubulk1}
\begin{split}
\mathcal{J}_2&=
                   \pm\frac{2\pi}{N\sqrt{t(t-4 )}}\mp\mathcal{F}_1+O\left(N^{-\frac{5}{4}}\right), \quad n-m=\pm1, \\
\mathcal{J}_2&=\pm\frac{2\pi}{N }\left(\frac{t-2 }{\sqrt{t(t-4
)}}-1\right) \mp\mathcal{F}_2+O\left(N^{-\frac{5}{4}}\right), \quad
n-m=\pm2.
\end{split}
\end{equation}
\end{proposition}
\begin{proof} We shall compute the integrals using Cauchy theorem.
Since $\left(\sqrt{x-4}\right)_+=i\sqrt{4-x}$ on $\left[0,4\right]$,
we have, by using (\ref{eq:sinarc}), the following
\begin{equation*}
\int_{0}^{4}\frac{\sin\left(\arccos\left(\frac{x}{2}-1\right)\right)}{x(4-x)\left(t-
x\right)}\D x= i\int_{0}^4\frac{1}{2\sqrt{x(x-4)}\left(t-
x\right)}\D x,
\end{equation*}
and similar relations for $n-m=2$. The right hand side can be
computed using Cauchy's theorem and we obtain
\begin{equation*}
\int_{0}^{4}\frac{\sin\left((n-m)\arccos\left(\frac{x}{2}-1\right)\right)}{x(4-x)(t-
x)}\D x=\left\{
    \begin{array}{ll}
       \frac{\pi}{2\sqrt{t(t-4 )}}
, & \hbox{$n-m=1$;} \\
      \frac{\pi}{2 }\left(\frac{t-2 }{\sqrt{t(t-4 )}}-1\right), & \hbox{$n-m=2$;} \\
    \end{array}
  \right.
\end{equation*}
By (\ref{eq:sinarc}) and (\ref{eq:sinarc2}), we see that
\begin{equation*}
\int_{0}^{\frac{\sqrt{N}}{n}}\frac{4\sin\left(l\arccos\left(\frac{x}{2}-1\right)\right)}{Nx(4-x)\left(t-
x\right)}\D x=O\left(N^{-\frac{5}{4}}\right).
\end{equation*}
From this and (\ref{eq:doubulk0}), (\ref{eq:sinarc}) and
(\ref{eq:sinarc2}), we have
\begin{equation*}
\begin{split}
\mathcal{J}_2&=\int_0^4\frac{\sin\left((n-m)\arccos\left(\frac{x}{2}-1\right)\right)}{x(4-x)\left(t-
x\right)}\D x-\mathcal{F}_{n-m}+O\left(N^{-\frac{5}{4}}\right).
\end{split}
\end{equation*}
This completes the proof of the Proposition.
\end{proof}
This completes the analysis in the bulk region. We will now move
onto the Airy region.
\subsubsection{The Airy region}\label{se:Airy}
The analysis in the Airy region is more difficult compare to the
other regions as we will need to consider the case where the point
$t$ is inside the Airy region. First let us compute the asymptotics
of the function $\hat{L}_n(x)$ in the Airy region.

As the asymptotics contain the functions $Ai(f_n)$ and
$Ai^{\prime}(f_n)$, we shall make a change of variable and write the
asymptotics in terms of the function $f_n$. First, from the
expression of $f_n$ in Definition \ref{de:bihomap}, we see that the
map $f_n$ is of the following order inside the Airy region.
\begin{equation}\label{eq:fnexp}
f_n=\left(\frac{n}{4}\right)^{\frac{2}{3}}(x-4)\left(1-\frac{x-4}{20}
+O((x-4)^2)\right),
\end{equation}
and hence $x-4$ is of order
\begin{equation}\label{eq:x4}
x-4=\left(\frac{4}{n}\right)^{\frac{2}{3}}f_n\left(1+\frac{1}{20}f_n\left(\frac{4}{n}\right)^{\frac{2}{3}}
+O\left(\frac{f_n^2}{n^{\frac{4}{3}}}\right)\right).
\end{equation}
Let us introduce the scaled variable $T$ to be
\begin{equation}\label{eq:T}
T=\left(\frac{N}{4}\right)^{\frac{2}{3}}\left(t-4 \right).
\end{equation}
If $T$ is close to the real axis, say $T=T_R+iT_I$, $T_R$,
$T_I\in\mathbb{R}$ and $|T_I|<\delta$, then we will deform the
integration contour $\mathbb{R}$ in $\left<L_n,L_m\right>_1$ inside
a small neighborhood of $T_R$ into the upper/lower half plane
depending on the sign of $T_I$, so that the distance between $T$ and
the integration contour is always greater than some finite
$\delta>0$. If $T\in\mathbb{R}$, then for definiteness, we will
deform the contour into the lower half plane and also deform the
branch cut of $(T-f_n)^{\frac{1}{2}}$ into the upper half plane so
that the integrals are well-defined. As pointed out in the
introduction, such deformations will not affect our final result.
For simplicity of the notations, we will denote the integration
contour in the Airy region by $\mathbb{R}$ in all cases, but with
the understand that appropriate deformations of the contours are
performed when $T$ is close to $\mathbb{R}$.

As $T$ is of finite distance from the integration contour, we have
\begin{equation}\label{eq:calB}
\begin{split}
&\left(t-\frac{n}{N}
x\right)^{-\frac{1}{2}}=\left(\frac{4}{N}\right)^{-\frac{1}{3}}
\left(T- f_n\right)^{-\frac{1}{2}}\\
&\times\left(1+\frac{4^{\frac{1}{3}} (N-n)}{N^{\frac{1}{3}}(T-
u)}+\frac{1}{T-
u}G_1\left(\frac{u(n-N)}{N}\right)+\frac{n}{4^{\frac{2}{3}}N^{\frac{1}{3}}(T-
u)}
G_2\left(\frac{u}{n^{\frac{2}{3}}}\right)\right)^{-\frac{1}{2}}\\
&=\left(\frac{4}{N}\right)^{-\frac{1}{3}} \left(T-
f_n\right)^{-\frac{1}{2}}\mathcal{G}_n^{-\frac{1}{2}}(f_n),
\end{split}
\end{equation}
where $G_1$ and $G_2$ are power series expansions in their arguments
with coefficients independent on $n$ and $N$. Note that the series
expansion $B$ starts at power 1 while the expansion $C$ starts at
power 2.

Let us write $f_n=u$, then by using (\ref{eq:calB}) and
(\ref{eq:x4}), we can write $\hat{L}_n(x)/u^{\prime}$ as a series in
terms of $u$.
\begin{equation}\label{eq:airyL}
\begin{split}
&\frac{\hat{L}_n(x)}{u^{\prime}(x)}=\frac{N^{\frac{5}{6}}\sqrt{\pi}}{\sqrt{2}n^{\frac{5}{6}}(T-
u)^{\frac{1}{2}}}\Bigg(
\left(\frac{n}{4}\right)^{-\frac{1}{6}}Ai(u)\left(a_0+\sum_{k=1}^{\infty}\frac{a_{k}u^k}{n^{\frac{2k}{3}}}
\right)
\\&+\frac{\alpha+1}{2}
\left(\frac{n}{4}\right)^{-\frac{1}{2}}Ai^{\prime}(u)\Bigg(b_0+\sum_{k=1}^{\infty}\frac{b_{k}u^k}{n^{\frac{2k}{3}}}
\Bigg)\Bigg)\mathcal{G}_n(u)
\end{split}
\end{equation}
For some constants $a_0$, $b_0$, $a_{k}$ and $b_{k}$ of the form
\begin{equation*}
\begin{split}
a_0&=1+O(n^{-1}),\quad b_0=1+O\left(n^{-1}\right),
a_{k}=a_{k,0}+O\left(n^{-1}\right),\\
b_{k}&=b_{k,0}+O\left(n^{-1}\right),\quad
\end{split}
\end{equation*}
where $a_{k,0}$ and $b_{k,0}$ are independent on $n$.

As the functions $\hat{L}_n$ are close to each other inside the Airy
region, we will introduce the following function $\mathcal{A}(U)$.
\begin{equation}\label{eq:A}
\mathcal{A}(U)=\frac{\hat{L}_N(x(U))}{u^{\prime}(x(U))},
\end{equation}
where $x(U)$ is the inverse function of $U=f_N(x)$. The function
$\mathcal{A}(U)$ is regarded as a function in the variable $U$. We
shall express the double integrals in terms of the function
$\mathcal{A}$.

As in (\ref{eq:x4}) we can write $x(U)$ as a power series in $U$
(but with $n$ replaced by $N$). In particular, we have
\begin{equation*}
\begin{split}
\mathcal{A}(U)&=\frac{\sqrt{\pi}}{\sqrt{2}(T-
U)^{\frac{1}{2}}}\Bigg(
\left(\frac{N}{4}\right)^{-\frac{1}{6}}Ai(U)\left(\tilde{a}_0+\sum_{k=1}^{\infty}\frac{\tilde{a}_{k}U^k}{N^{\frac{2k}{3}}}
\right)
\\&+\frac{\alpha+1}{2}
\left(\frac{N}{4}\right)^{-\frac{1}{2}}Ai^{\prime}(U)\Bigg(\tilde{b}_0+\sum_{k=1}^{\infty}\frac{\tilde{b}_{k}U^k}{N^{\frac{2k}{3}}}
\Bigg)\Bigg)\mathcal{G}_N(U)
\end{split}
\end{equation*}
where $\tilde{a}_k=a_k+O\left(n^{-1}\right)$ and
$\tilde{b}_k=b_k+O\left(n^{-1}\right)$. If we now let $x$ be the
inverse function of $u=f_n(x)$ and use (\ref{eq:calB}) and
(\ref{eq:airyL}) to represent $\hat{L}_n(x(u))/u^{\prime}(x(u))$ as
a function in $u$, then we see that, as a function of $u$, we have
\begin{equation}\label{eq:AL}
\begin{split}
\frac{\hat{L}_n(x(u))}{u^{\prime}(x(u))}&=\left(\mathcal{A}(u)+\frac{Ai(u)}{(T-
u)^{\frac{1}{2}}}\left(\frac{N}{4}\right)^{-\frac{1}{6}}\mathcal{C}_{1,n}(u)
+\frac{Ai^{\prime}(u)}{(T- u)^{\frac{1}{2}}N^{\frac{1}{2}}}\mathcal{C}_{2,n}(u)\right)\\
&\times\left(\mathcal{E}_n(u)+ \mathcal{D}_n(u)\right),
\end{split}
\end{equation}
where $\mathcal{C}_{i,n}(u)$ and $\mathcal{D}_n(u)$ are power series
of the form
\begin{equation}\label{eq:CD}
\begin{split}
\mathcal{C}_{i,n}(u)&=\sum_{k,j}^{\infty}\frac{c_{i,k,j,n}u^k}{N^{\frac{2(k-j+3)}{3}}(T- u)^j},\quad i=1,2,\\
\mathcal{D}_n(u)&=\sum_{k,j}^{\infty}\frac{d_{k,j,l,n}u^k}{N^{\frac{2(k+j)+l}{3}}(T-
u)^{j+l}},
\end{split}
\end{equation}
where $d_{k,j,l,n}$ and $c_{i,k,j,n}$ are bounded and $k\geq 2j$. In
the series $\mathcal{C}_{i,n}(u)$, both indices $k$ and $j$ start
from $0$, while in $\mathcal{D}_n(u)$, $k$ and $l$ start from $1$
and $j$ from $0$. The term $\mathcal{E}_n(u)$ is of the form
\begin{equation}\label{eq:E}
\mathcal{E}_n(u)=1-\frac{1}{2}v_{0,n}+\frac{3}{8}v_{0,n}^2+O\left(N^{-1}\right),
\end{equation}
where $v_0$ is given by
\begin{equation}\label{eq:v0}
v_{0,n}=\frac{4^{\frac{1}{3}} (N-n)}{N^{\frac{1}{3}}(T- u)}
\end{equation}
We are now ready to compute the single integrals. First let us show
the following
\begin{lemma}\label{le:airyint}
Let $k,j\geq 0$, then as $s_{\pm}\rightarrow +\infty$ and
$|s_-|/|s_+|=O(1)$, we have
\begin{equation}\label{eq:airyint}
\begin{split}
\int_{s_-}^{s_+}\frac{u^kAi(u)}{(T- u)^{\frac{j}{2}}}\D
u&=O\left(|s_-|^{k-\frac{j}{2}-\frac{3}{4}}\right)+O\left(1\right),\\
\int_{s_-}^{s_+}\frac{u^kAi^{\prime}(u)}{(T- u)^{\frac{j}{2}}}\D
u&=O\left(|s_-|^{k-\frac{j}{2}-\frac{1}{4}}\right)+O\left(1\right),\\
\end{split}
\end{equation}
Let $\nu(u)$ be a function of $u$ such that $\nu(u)=O(u)$ as
$u\rightarrow\infty$ and let $v_+=O(s_+)$, then we have
\begin{equation}\label{eq:airydou}
\begin{split}
\int_{s_-}^{s_+}\frac{u^{k_1}Ai^{(i_1)}(u)}{(T-
u)^{\frac{j_1}{2}}}\int_{\nu(u)}^{v_+}
\frac{u^{k_2}Ai^{(i_2)}(u)}{(T- u)^{\frac{j_2}{2}}}\D v \D
u=O\left(|s_-|^{k_1+k_2-\frac{j_1+j_2-i_1-i_2}{2}}\right)+O(1),
\end{split}
\end{equation}
where $Ai^{(0)}=Ai$ and $Ai^{(1)}=Ai^{\prime}$.
\end{lemma}
\begin{proof} First note that, as the Airy
function decays exponentially as $u\rightarrow \infty$,
\begin{equation}\label{eq:airyasymp}
\begin{split}
Ai(u)&=\frac{1}{2\sqrt{\pi}}u^{-\frac{1}{4}}e^{-\frac{2}{3}u^{\frac{3}{2}}}\left(1+O\left(u^{-\frac{3}{2}}\right)\right),\quad
u\rightarrow +\infty,\\
Ai(-u)&=\frac{\sin\left(\frac{2}{3}u^{\frac{3}{2}}+\frac{\pi}{4}\right)}
{\sqrt{\pi}u^{\frac{1}{4}}}\left(1+O\left(u^{-\frac{3}{2}}\right)\right),\quad
u\rightarrow-\infty.
\end{split}
\end{equation}
From this we have
\begin{equation*}
\int_0^{s_+}\frac{u^kAi(u)}{(T- u)^{\frac{j}{2}}}\D
u=O\left(1\right),
\end{equation*}
as $s_+\rightarrow +\infty$. Therefore let us consider the integral
in the negative real axis. Integrating by parts, we obtain
\begin{equation}\label{eq:intneg}
\begin{split}
\int_{s_-}^0\frac{u^kAi(u)}{(T- u)^{\frac{j}{2}}}\D u&=
\left(\frac{u^k\int_{-\infty}^uAi(v)\D v}{(T-
u)^{\frac{j}{2}}}\right)_{s_-}^0-k\int_{s_-}^0\frac{u^{k-1}\int_{-\infty}^uAi(v)\D
v}{(T- u)^{\frac{j}{2}}}\D u\\
&-\frac{j }{2}\int_{s_-}^0\frac{u^k\int_{-\infty}^uAi(v)\D v}{(T-
u)^{\frac{j}{2}+1}}\D u
\end{split}
\end{equation}
Now by (\ref{eq:airyasymp}), we see that $\int_{-\infty}^uAi(v)\D
v=O\left(u^{-\frac{3}{4}}\right)$ as $u\rightarrow-\infty$ and is
bounded for $u\in\mathbb{R}$. This gives us the following estimate
\begin{equation*}
\int_{s_-}^0\frac{u^kAi(u)}{(T- u)^{\frac{j}{2}}}\D
u=O\left(|s_-|^{k-\frac{j}{2}-\frac{3}{4}}\right)+O(1)
\end{equation*}
For integrals involving the derivative $Ai^{\prime}(u)$, we again
note that $Ai^{\prime}(u)$ is decaying exponentially as
$u\rightarrow+\infty$ and we again have
\begin{equation*}
\int_0^{s_+}\frac{u^kAi^{\prime}(u)}{(T- u)^{\frac{j}{2}}}\D
u=O\left(1\right).
\end{equation*}
For the integral on the negative real axis, we perform integration
by parts again and use the estimate
\begin{equation}\label{eq:airyest}
|Ai(s)|\leq C\left(1+|s|\right)^{-\frac{1}{4}},\quad s<0
\end{equation}
for some constant $C$. This shows that the integral on the negative
real axis is of order
\begin{equation*}
\int_{s_-}^0\frac{u^kAi^{\prime}(u)}{(T- u)^{\frac{j}{2}}}\D
u=O\left(|s_-|^{k-\frac{j}{2}-\frac{1}{4}}\right)+O\left(1\right).
\end{equation*}
This proves (\ref{eq:airyint}). The estimate (\ref{eq:airydou}) now
follows immediately from (\ref{eq:airyint}) and
(\ref{eq:airyasymp}).
\end{proof}
Let us now transform the limits in the Airy region into the variable
$u$. As in the previous cases, we are interested in the integral
\begin{equation*}
\left(i\kappa_{n-1}\right)^{\frac{1}{2}}\int_{c_-}^{c_+}L_n(x)w(x)\D
x=\mathcal{I}_{3,n}=\frac{n}{N}\int_{\frac{N}{n}c_-}^{\frac{N}{n}c_+}\hat{L}_n(x)\D
x,
\end{equation*}
The following is an immediate consequence of (\ref{eq:AL}) and the
estimates (\ref{eq:airyint}).
\begin{proposition}\label{pro:singairy} Let $T$ be defined as in
(\ref{eq:T}), then the single integral in the Airy region is given
by
\begin{equation}\label{eq:singairy}
\begin{split}
&\frac{N}{n}\left(i\kappa_{n-1}\right)^{\frac{1}{2}}\int_{c_-}^{c_+}L_n(x)w(x)\D
x=\int_{u_-}^{u_+}\mathcal{A}\mathcal{E}_n\D u
+O\left(N^{-\frac{7}{6}}\right)
\end{split}
\end{equation}
\end{proposition}
This completes the analysis of the single integral in the Airy
region. We will now analyze the double integrals.

By (\ref{eq:hatL}), we have
\begin{equation}\label{eq:j3}
\begin{split}
&i\sqrt{\kappa_{n-1}\kappa_{m-1}}\int_{c_-}^{c_+}L_n(x)w(x)\int_x^{c_-}L_m(y)w(y)\D
x\D
y\\
&=\mathcal{J}_3=\frac{mn}{N^2}\int_{\frac{N}{n}c_-}^{\frac{N}{n}c_+}\hat{L}_n(x)
\int_{\frac{n}{m}x}^{\frac{N}{m}c_+}\hat{L}_m(y)\D x\D y,
\end{split}
\end{equation}
We will change the integration variables to $u=f_n(x)$ and
$v=f_m(y)$. We will denote the limits of the outer integral by
$u_{\pm}$ and the upper limit of the inner integral by $v_+$.

Let us now compute the lower limit of the inner integral. As both
$u=f_n(x)$ and $v=f_m(y)$ are conformal inside the Airy region, they
can be written as a series of each other. Then by (\ref{eq:x4}) and
the analogue for $f_m$, together with the fact that at the lower
bound, $y=\frac{n}{m}x$, we obtain
\begin{equation*}
\begin{split}
&4\left(\frac{n}{m}-1\right)+\frac{n}{m}\left(\left(\frac{4}{n}\right)^{\frac{2}{3}}u
+\frac{1}{20}\left(\frac{4}{n}\right)^{\frac{4}{3}}u^2+O\left(\frac{u^3}{n^2}\right)\right)\\
&=\left(\frac{4}{m}\right)^{\frac{2}{3}}v+\frac{1}{20}\left(\frac{4}{m}\right)^{\frac{4}{3}}v^2
+O\left(\frac{v^3}{m^2}\right)
\end{split}
\end{equation*}
The following can easily be seen by writing $v$ as a series
expansion in $u$.
\begin{lemma}\label{le:vcoef}
Let $u=f_n(x)$ and let $\nu(u)$ be the value of $v=f_{m}$ at
$y=\frac{n}{m}x$, then as $n,m\rightarrow\infty$ with $m-n$ finite,
we have
\begin{equation}\label{eq:vcoef}
\begin{split}
\nu(u)&=v_0+\left(1+v_1\right)
u+\sum_{l=2}^{\infty}v_lu^l,\quad v_0=(n-m)\left(\frac{4}{m}\right)^{\frac{1}{3}}+O\left(m^{-\frac{4}{3}}\right),\\
\delta v_1&=\frac{1}{15}\frac{m-n}{m}+O\left(m^{-2}\right),\quad
v_l=O\left(m^{-\frac{2l+1}{3}}\right),\quad l\geq 2,
\end{split}
\end{equation}
\end{lemma}
By (\ref{eq:airydou}) and (\ref{eq:AL}), we see that the double
integral is given by
\begin{equation}\label{eq:douint}
\begin{split}
&i\sqrt{\kappa_{n-1}\kappa_{m-1}}\int_{c_-}^{c_+}L_n(x)w(x)\int_x^{c_-}L_m(y)w(y)\D
x\D
y=\mathcal{J}_3\\
&=\frac{mn}{N^2}\Bigg(
\int_{u_-}^{u_+}\mathcal{E}_n\mathcal{A}(u)\int_{\nu(u)}^{v_+}\mathcal{E}_m\mathcal{A}(v)\D
v\D u\Bigg)+O\left(N^{-\frac{4}{3}}\right).
\end{split}
\end{equation}
Let us consider the terms in
$\int_{u_-}^{u_+}\mathcal{E}_n\mathcal{A}(u)\int_{\nu(u)}^{v_+}\mathcal{E}_m\mathcal{A}(v)\D
v\D u$ in (\ref{eq:douint}). This term can be written as
\begin{equation}\label{eq:termdouint}
\begin{split}
&\int_{u_-}^{u_+}\mathcal{E}_n\mathcal{A}(u)\int_{\nu(u)}^{v_+}\mathcal{E}_m\mathcal{A}(v)\D
v\D
u\\
&=\sum_{j,k=0}^{2}\frac{h_{kj}}{N^{\frac{j+k}{3}}}\left((N-n)^j(N-m)^k\mathcal{A}(j,k)+(N-n)^k(N-m)^j\mathcal{A}(k,j)\right)
\\
&+O\left(N^{-\frac{4}{3}}\right)
\end{split}
\end{equation}
for some constants $h_{kj}$ that are independent on $N$, $n$ and
$m$, where $\mathcal{A}(j,k)$ are defined as follows.
\begin{equation}\label{eq:calH}
\mathcal{A}(j,k)=\int_{u_-}^{u_+}\frac{\mathcal{A}(u)}{(T-
u)^{j}}\int_{\nu(u)}^{v_+}\frac{\mathcal{A}(v)} {(T- v)^{k}}\D v\D
u.
\end{equation}
First let us consider the leading order term in
(\ref{eq:termdouint}), which is given by $\mathcal{A}(0,0)$.

Let $\mathcal{L}(u)$ be the following function in $u$.
\begin{equation}\label{eq:Lu}
\mathcal{L}(u)=\frac{\D x_N}{\D
u}\frac{2^{\frac{3}{2}}(-u)^{\frac{1}{4}}(T-
u)^{\frac{1}{2}}}{x_N(u)^{\frac{3}{4}}(4-x_N(u))^{\frac{1}{4}}\left(t-
x_N(u)\right)^{\frac{1}{2}}}
\end{equation}
where $x_N(u)$ is defined as the inverse function to $u=f_N(x_N)$.

Note that $\mathcal{L}(u)$ is of order
\begin{equation}\label{eq:Luord}
\mathcal{L}(u)=\left(\frac{N}{4}\right)^{-\frac{1}{6}}\left(1+O\left(\frac{u}{N^{\frac{2}{3}}}\right)\right)
\end{equation}
where the error term is uniform in $T$ as $T\rightarrow
N^{\frac{2}{3}}$. Then $\mathcal{A}(u)$ can be written as
\begin{equation*}
\mathcal{A}(u)=\sqrt{\frac{\pi}{2}}\frac{u^{\frac{1}{4}}\cos\eta_+(x_N)Ai(u)\left(1+O\left(N^{-1}\right)\right)
-u^{-\frac{1}{4}}\sin\eta_+(x_N)Ai^{\prime}(u)\left(1+O\left(N^{-1}\right)\right)}{(-u)^{\frac{1}{4}}
(T- u)^{\frac{1}{2}}}\mathcal{L}(u)
\end{equation*}
The following can be derived using (\ref{eq:matchairy}).
\begin{lemma}\label{le:Aasym}
The derivatives of $\mathcal{A}(u)$ behave as follows as
$u\rightarrow -CN^{\frac{2}{3}-\varepsilon}$ for some $C>0$.
\begin{equation}\label{eq:Aasym}
\begin{split}
\mathcal{A}^{(2k-1)}(u)&=\frac{1}{\sqrt{2}}\left(\frac{(-1)^{k-1}(-u)^{\frac{2k-1}{2}}\sin
F_N(x_N(u))}{(-u)^{\frac{1}{4}}(T- u)^{\frac{1}{2}}}\mathcal{L}(u)+O\left(\frac{u^{k-\frac{9}{4}}}{(T- u)^{\frac{1}{2}}}\right)\right),\\
\mathcal{A}^{(2k)}(u)&=\frac{1}{\sqrt{2}}\left(\frac{u^{k}\cos
F_N(x_N(u))}{(-u)^{\frac{1}{4}}(T-
u)^{\frac{1}{2}}}\mathcal{L}(u)+\left(\frac{u^{k-\frac{7}{4}}}{(T-
u)^{\frac{1}{2}}}\right)\right),
\end{split}
\end{equation}
where $k\geq 0$ and $x_N(u)$ is the inverse of $u=f_N(x)$.
\end{lemma}
Changing the upper limit $v_+$ into $u_+$, we can write the term
$\mathcal{A}(0,0)$ as
\begin{equation}\label{eq:Adoub}
\begin{split}
\mathcal{A}(0,0)&=\frac{1}{2}\left(\int_{u_-}^{u_+}\mathcal{A}(u)\D
u\right)^2-\sum_{j=0}^{\infty}\int_{u_-}^{u_+}\frac{\mathcal{A}(u)\mathcal{A}^{(j)}(u)}{(j+1)!}(\nu-u)^{j+1}\D
u+O\left(e^{-\frac{2}{3}u_+^{\frac{3}{2}}}\right)
\end{split}
\end{equation}
where $\nu$ is treated as a function of $u$ by using Lemma
\ref{le:vcoef}.

Let us estimate the order of these terms.
\begin{lemma}\label{le:jkprime}
Let $u_-/N^{\frac{2}{3}}=O\left(N^{-\varepsilon}\right)$. If $j$ is
odd, then we have
\begin{equation}\label{eq:jkodd}
\int_{u_-}^{u_+}\mathcal{A}(u)\mathcal{A}^{(j)}(u)(\nu-u)^{j+1}\D
u=\left\{
    \begin{array}{ll}
      O\left(m^{-2+\frac{3\varepsilon}{2}}\right)
    \end{array}
  \right.
\end{equation}
If $j$ is even, then we have
\begin{equation}\label{eq:jkeven}
\begin{split}
&\int_{u_-}^{u_+}\mathcal{A}(u)\mathcal{A}^{(j)}(u)(\nu-u)^{j+1}\D
u=\frac{1}{4}\int_{u_-}^{0}\frac{u^{\frac{j}{2}}\mathcal{L}^2(u)(\nu-u)^{j+1}}{(-u)^{\frac{1}{2}}(T-
u)}\D
u\\
&+\left\{
    \begin{array}{ll}
      O\left(N^{-\frac{4}{3}}\right), & \hbox{$j>0$;} \\
      O\left(N^{-\frac{2}{3}}\right), & \hbox{$j=0$.}
    \end{array}
  \right.
\end{split}
\end{equation}
\end{lemma}
\begin{proof} The first equation can be proven using integration by
parts. By (\ref{eq:vcoef}), (\ref{eq:Aasym}) and (\ref{eq:Luord}),
we have
\begin{equation*}
\begin{split}
&\int_{u_-}^{u_+}\mathcal{A}(u)\mathcal{A}^{(j)}(u)(\nu-u)^{j+1}\D
u=-\mathcal{A}(u_-)\mathcal{A}^{(j-1)}(u_-)(\nu(u_-)-u_-)^{j+1}\\
&-\int_{u_-}^{u_+}\mathcal{A}^{\prime}(u)\mathcal{A}^{(j-1)}(u)(\nu-u)^{j+1}\D
u+O\left(\left(\frac{u_-^{\frac{j-2}{2}}}{m^{\frac{j-2}{3}}}\right)m^{-2}\right)
\end{split}
\end{equation*}
From (\ref{eq:Aasym}), (\ref{eq:vcoef}) and (\ref{eq:Luord}), the
first term is of order
$O\left(\left(u_-/m^{\frac{2}{3}}\right)^{\frac{j-4}{2}}m^{-2}\right)$.
Repeating the integration by parts, we obtain (\ref{eq:jkodd}).

To prove (\ref{eq:jkeven}), we have, by (\ref{eq:Aasym}), the
following
\begin{equation}\label{eq:AAj}
\begin{split}
\mathcal{A}(u)\mathcal{A}^{(j)}(u)-\frac{1}{4}\frac{u^{\frac{j}{2}}\mathcal{L}^2(u)}
{(-u)^{\frac{1}{2}}(T- u)} -\frac{1}{4}\frac{u^{\frac{j}{2}}\cos
2F_N(x_N)}{(-u)^{\frac{1}{2}}(T- u)}\mathcal{L}^2(u)
=O\left(u^{\frac{j-6}{2}}m^{-\frac{1}{3}}\right)
\end{split}
\end{equation}
as $u\rightarrow u_-$. After multiplying by $(\nu(u)-u)^{j+1}$ and
integrate, the error term
$O\left(u^{\frac{j-6}{2}}m^{-\frac{1}{3}}\right)$ gives a
contribution of order $O\left(N^{-\frac{4}{3}}\right)$ for $j>0$ and
$O\left(N^{-\frac{2}{3}}\right)$ for $j=0$.

The term with the $\cos 2F_N$ factor can be integrated by parts
using
\begin{equation*}
\begin{split}
F_N&=\eta_++iN\varphi_{+}-\frac{\pi}{4}=\frac{2}{3}\left(-u\right)^{\frac{3}{2}}+O\left(\frac{u^{\frac{1}{2}}}{N^{\frac{1}{3}}}\right)
-\frac{\pi}{4}.\,\\
\frac{\D F_N}{\D u}&=-(-u)^{\frac{1}{2}}+O\left(\frac{1}{u^{\frac{1}{2}}N^{\frac{1}{3}}}\right)\,\\
\end{split}
\end{equation*}
Then by splitting the interval $[u_-,0]$ into $[u_-,u_0]$ and
$[u_0,0]$ for some $u_0$ between $u_-$ and $0$, and integrating the
integral on $[u_-,u_0]$ by parts, we have
\begin{equation*}
\begin{split}
&\int_{u_-}^{0}\left(\mathcal{A}(u)\mathcal{A}^{(j)}(u)-\frac{1}{4}\frac{u^{\frac{j}{2}}}
{(-u)^{\frac{1}{2}}(T-
u)}\mathcal{L}^2(u)\right)\left(\nu(u)-u\right)^{j+1} \D
u\\
&=O\left(\frac{u_-^{\frac{\max(0,j-4)}{2}}}{m^{\frac{j-4}{3}}}m^{-2}\right)+O\left(m^{-\frac{j+2}{3}}\right)
\end{split}
\end{equation*}
Hence we have
\begin{equation}\label{eq:halford}
\begin{split}
&\int_{u_-}^{0}\left(\mathcal{A}(u)\mathcal{A}^{(j)}(u)(\nu-u)^{j+1}
-\frac{1}{4}\frac{u^{\frac{j}{2}}(\nu-u)^{j+1}}{(-u)^{\frac{1}{2}}(T-
u)}\mathcal{L}^2(u) \right)\D
u\\
&=\left\{
    \begin{array}{ll}
      O\left(N^{-\frac{4}{3}}\right), & \hbox{$j>0$;} \\
      O\left(N^{-\frac{2}{3}}\right), & \hbox{$j=0$.}
    \end{array}
  \right.
\end{split}
\end{equation}
As $\mathcal{A}(u)\mathcal{A}^{(j)}(u)(\nu-u)^{j+1}$ is integrable
in $[0,\infty)$ and have exponential decay at $+\infty$,
(\ref{eq:jkeven}) now follows from (\ref{eq:halford}).
\end{proof}
We would now like to combine the terms in (\ref{eq:Adoub}) with the
end point terms $\mathcal{F}_l$ in (\ref{eq:fl}) from the double
integral in the bulk region.

\begin{lemma}\label{le:Fmn}
Let $\mathcal{F}_{n-m}$ be given by (\ref{eq:fl}), then we have
\begin{equation}\label{eq:Fmn}
\begin{split}
&\mathcal{F}_{n-m}+\frac{1}{4}\sum_{k=0}^{\infty}\int_{u_-}^0\frac{u^k\left(\nu(u)-u\right)^{2k+1}}{(-u)^{\frac{1}{2}}
(T- u)(2k+1)!}\mathcal{L}^2(u)\D u=\\
 &4\int_{\frac{N}{n}c_-}^4\frac{
(n-m)\phi} {Nx(4-x)(t- x)}\D
x+\frac{1}{4}\int_{u_-}^0\frac{\left(\nu(u)-u\right)}{(-u)^{\frac{1}{2}}
(T- u)}\mathcal{L}^2(u)\D u + O\left(N^{-\frac{4}{3}}\right).
\end{split}
\end{equation}
\end{lemma}
\begin{proof}
From (\ref{eq:fl}), we can write $\mathcal{F}_{n-m}$ in the
following form. (After renaming the integration variable from $x$ to
$x_N$)
\begin{equation}\label{eq:Fmn1}
\begin{split}
\mathcal{F}_{n-m}&=4\int_{\frac{N}{n}c_-}^4\frac{
\sin\left((n-m)\phi(x_N)\right)} {Nx_N(4-x_N)(t- x_N)}\D x_N \\
&=4\sum_{k=0}^{\infty}\int_{\frac{N}{n}c_-}^4\frac{
(-1)^k\left((n-m)\phi(x_N)\right)^{2k+1}} {Nx_N(4-x_N)(t-
x_N)(2k+1)!}\D x_N
\end{split}
\end{equation}
where $\phi(x_N)=\arccos\left(\frac{x_N}{2}-1\right)$. Let
$u=f_N(x_N)$. Then from the proof of Lemma \ref{le:diffmn}, we can
express $\phi(x_N)$ as follows when $u\rightarrow u_-$.
\begin{equation*}
(n-m)\phi(x_N)=F_{n}(x_N)-F_m\left(x_N\frac{n}{m}\right)+O\left(\frac{1}{N}\left(\frac{u}{N^{\frac{2}{3}}}\right)^{-\frac{1}{2}}\right)
\end{equation*}
Now $F_n(x_N)$ is given by
\begin{equation*}
\begin{split}
F_n(x_N)&=in\varphi_{+}(x_N)+\eta_+(x_N)-\frac{\pi}{4}=\frac{n}{N}\frac{2}{3}(-u)^{\frac{3}{2}}+\eta_+(x_N)-\frac{\pi}{4}.
\end{split}
\end{equation*}
Similarly, we have
$F_m(x_N)=\frac{2}{3}(-f_m)^{\frac{3}{2}}+\eta_+(x_N)-\frac{\pi}{4}$
and $f_m(x_Nn/m)=\nu(f_n)$. Then by Lemma \ref{le:vcoef} and
$f_n(x_N)=\left(\frac{n}{N}\right)^{\frac{2}{3}}f_N(x_N)=\left(\frac{n}{N}\right)^{\frac{2}{3}}u$,
we obtain
\begin{equation*}
\begin{split}
F_m\left(\frac{n}{m}x_N\right)&=\frac{2}{3}\left((-\nu(u))^{\frac{3}{2}}+\frac{n-N}{N}u(-\nu(u))^{\frac{1}{2}}\right)
\\
&+\eta_+\left(\frac{n}{m}x_N\right)-\frac{\pi}{4}+O\left(\frac{1}{N}\left(\frac{u}{N^{\frac{2}{3}}}\right)^{-\frac{1}{2}}\right)
\end{split}
\end{equation*}
From (\ref{eq:eta}), we obtain
\begin{equation*}
F_n(x_N)-F_m\left(\frac{n}{m}x_N\right)=(-u)^{\frac{1}{2}}(\nu(u)-u)+O\left(\frac{1}{N}\left(\frac{u}{N^{\frac{2}{3}}}\right)^{-\frac{1}{2}}\right)
\end{equation*}
as $u\rightarrow u_-$. From (\ref{eq:Lu}), we see that
$\mathcal{L}^2(u)$ is given by
\begin{equation*}
\mathcal{L}^2(u)=\frac{8(-u)^{\frac{1}{2}} (T-
u)}{x_N^{\frac{3}{2}}(4-x_N)^{\frac{1}{2}}\left(t-
x_N\right)}\left(\frac{\D x_N}{\D u}\right)^2.
\end{equation*}
By using $u=f_N(x_N)$, (\ref{eq:fmap}) and (\ref{eq:varphi}), we
obtain $\D_u x_N=\frac{2u^{\frac{1}{2}}}{N}\sqrt{x_N/x_N-4}$.
Therefore $\mathcal{L}^2(u)$ is can be written as
\begin{equation*}
\mathcal{L}^2(u)=\frac{16u (T- u)}{Nx_N(4-x_N)\left(t-
x_N\right)}\frac{\D x_N}{\D u}.
\end{equation*}
This gives us
\begin{equation*}
\frac{4(-1)^k\left((n-m)\phi(x_N)\right)^{2k+1}} {Nx_N(4-x_N)(t-
x_N)}=-\frac{1}{4}\frac{u^k\left(\nu(u)-u+O\left(N^{-\frac{2}{3}}u^{-1}\right)\right)^{2k+1}}{(-u)^{\frac{1}{2}}
(T- u)}\mathcal{L}^2(u)\frac{\D u}{\D x_N}
\end{equation*}
as $u\rightarrow u_-$. As $\nu-u=O\left(N^{-\frac{1}{3}}\right)$ and
$u_-/N^{\frac{2}{3}}=O\left(N^{-\varepsilon}\right)$, we see that
for $k>1$,
\begin{equation}\label{eq:dif}
\begin{split}
&\frac{4(-1)^k\left((n-m)\phi(x_N)\right)^{2k+1}} {Nx_N(4-x_N)(t-
x_N)}+\frac{1}{4}\frac{u^k\left(\nu(u)-u\right)^{2k+1}}{(-u)^{\frac{1}{2}}
(T- u)}\mathcal{L}^2(u)\frac{\D u}{\D x_N}
=O\left(\frac{u^{k-\frac{5}{2}}}{N^{\frac{2(k+1)}{3}}}\right).
\end{split}
\end{equation}
as $u\rightarrow u_-$. Since
$\phi(x_N)=\arccos\left(\frac{x_N}{2}-1\right)$ it behaves as
$\sqrt{4-x_N}$ as $x\rightarrow 4$. Therefore the function on the
left hand side of (\ref{eq:dif}) is integrable for $u\in (u_-,0)$.
Therefore by using $\varepsilon\leq1/20$, we have
\begin{equation*}
\begin{split}
\int_{\frac{N}{n}c_-}^{4}\frac{4(-1)^k\left((n-m)\phi(x)\right)^{2k+1}}
{Nx(4-x)(t- x)}\D
x&=-\frac{1}{4}\int_{u_-}^0\frac{u^k\left(\nu(u)-u\right)^{2k+1}}{(-u)^{\frac{1}{2}}
(T- u)}\mathcal{L}^2(u)\D u+O\left(N^{-\frac{4}{3}}\right)
\end{split}
\end{equation*}
for $k>0$ and of order $O\left(N^{-\frac{2}{3}}\right)$ for $k=0$.
This, together with (\ref{eq:Fmn1}), implies the Lemma.
\end{proof}
By (\ref{eq:Adoub}), (\ref{eq:Fmn1}) and Lemma \ref{le:jkprime},
\ref{le:Fmn}, we see that $\mathcal{A}(0,0)$ can be written as
\begin{equation}\label{eq:H00}
\begin{split}
\mathcal{A}(0,0)&-\mathcal{F}_{n-m}=\frac{1}{2}\left(\int_{u_-}^{u_+}\mathcal{A}(u)\D
u\right)^2-\int_{u_-}^{u_+}\mathcal{A}^2(u)(\nu-u)\D u
\\
&-4\int_{\frac{N}{n}c_-}^{4}\frac{\left((n-m)\phi\right)}{Nx(4-x)(t-
x)}\D x+O\left(N^{-\frac{4}{3}}\right)
\end{split}
\end{equation}
By (\ref{eq:vcoef}), we see that the terms in the sum involving
$\mathcal{A}$ are of the form
\begin{equation}\label{eq:Acoef}
\begin{split}
\int_{u_-}^{u_+}\mathcal{A}^2(u)(\nu-u)\D
u=\frac{4^{\frac{1}{3}}(n-m)}{N^{\frac{1}{3}}}\int_{u_-}^{u_+}\mathcal{A}^2(u)\D
u+O\left(N^{-\frac{4}{3}+\frac{\varepsilon}{2}}\right).
\end{split}
\end{equation}
Let us now compute the other terms in (\ref{eq:termdouint}). These
terms are of the form
\begin{equation*}
\frac{h_{kj}}{N^{\frac{j+k}{3}}}\left((N-n)^j(N-m)^k\mathcal{A}(j,k)+(N-n)^k(N-m)^j\mathcal{A}(k,j)\right)
\end{equation*}
for some for some constants $h_{kj}$ that are independent on $N$,
$n$ and $m$.

To compute these terms, let us first prove the following.
\begin{lemma}\label{le:FG}
Let $F$ and $G$ be integrable functions on $[u_-,u_+]$ such that $F$
and $G$ are of order $e^{-ku^{\frac{3}{2}}}$ as
$u\rightarrow+\infty$ for some $k>0$. Then we have
\begin{equation}\label{eq:FG}
\begin{split}
&\int_{u_-}^{u_+}F(u)\int_{\nu(u)}^{v_+}G(v)\D v\D
u+\int_{u_-}^{u_+}G(u)\int_{\nu(u)}^{v_+}F(v)\D v\D
u\\
&=\int_{u_-}^{u_+}F(u)\D u\int_{u_-}^{u_+}G(u)\D u +
\int_{u_-}^{u_+}F(u)\int_{\nu(u)}^{u}G(v)\D v\D u\\
&+ \int_{u_-}^{u_+}G(u)\int_{\nu(u)}^{u}F(v)\D v\D
u+O\left(e^{-ku_+^{\frac{3}{2}}}\right)
\end{split}
\end{equation}
\end{lemma}
\begin{proof} The lemma can be proved by integration by parts.
We have
\begin{equation}\label{eq:FG1}
\int_{u_-}^{u_+}F(u)\int_{\nu(u)}^{v_+}G(v)\D v\D u=
\int_{u_-}^{u_+}F(u)\int_{u}^{v_+}G(v)\D v\D
u+\int_{u_-}^{u_+}F(u)\int_{\nu(u)}^{u}G(v)\D v\D u
\end{equation}
Integrating the first term by parts, we have
\begin{equation*}
\begin{split}
&\int_{u_-}^{u_+}F(u)\int_{u}^{v_+}G(v)\D v\D
u=\int_{u_-}^{u_+}F(u)\D u\int_{u_-}^{u_+}G(u)\D u\\
&- \int_{u_-}^{u_+}G(u)\int_{u}^{v_+}F(v)\D v\D
u+O\left(e^{-ku_+^{\frac{3}{2}}}\right)
\end{split}
\end{equation*}
as $F$ is of order $e^{-ku^{\frac{3}{2}}}$ when
$u\rightarrow+\infty$. This implies
\begin{equation*}
\begin{split}
&\int_{u_-}^{u_+}F(u)\int_{u}^{v_+}G(v)\D v\D
u=\int_{u_-}^{u_+}F(u)\D u\int_{u_-}^{u_+}G(u)\D u-
\int_{u_-}^{u_+}G(u)\int_{\nu(u)}^{v_+}F(v)\D v\D
u\\&+\int_{u_-}^{u_+}G(u)\int_{\nu(u)}^{u}F(v)\D v\D
u+O\left(e^{-ku_+^{\frac{3}{2}}}\right)
\end{split}
\end{equation*}
This, together with (\ref{eq:FG1}), implies the lemma.
\end{proof}
By taking $F=\mathcal{A}/(T- u)^{j}$ and $G=\mathcal{A}/(T- u)^{k}$
for $\mathcal{A}(j,k)$, we have the following.
\begin{corollary}\label{cor:Hjk}
The terms in (\ref{eq:termdouint}) are given by
\begin{equation}\label{eq:Hjk}
\begin{split}
&(N-m)^{k-j}\mathcal{A}(j,k)+(N-n)^{k-j}\mathcal{A}(k,j)=
AH(j,k)\left((N-m)^{k-j}-(N-n)^{k-j}\right)\\
&+SH(j,k)\left((N-m)^{k-j}+(N-n)^{k-j}\right)\\
&+\frac{1}{2}\left((N-m)^{k-j}+(N-n)^{k-j}\right)\int_{u_-}^{u_+}\mathcal{A}_j\D
u\int_{u_-}^{u_+}\mathcal{A}_k\D
u+O\left(e^{-ku_+^{\frac{3}{2}}}\right),
\end{split}
\end{equation}
where $AH(j,k)$ and $SH(j,k)$ and $\mathcal{A}_j$ are given by
\begin{equation*}
\begin{split}
\mathcal{A}_j(u)&=\frac{\mathcal{A}(u)}{(T- u)^{j}},\quad
AH(j,k)=\frac{1}{2}\left(\mathcal{A}(j,k)-\mathcal{A}(k,j)\right),\\
SH(j,k)&=\frac{1}{2}\left(\int_{u_-}^{u_+}\mathcal{A}_j(u)\int_{\nu(u)}^{u}\mathcal{A}_k(v)\D
v\D
u+\int_{u_-}^{u_+}\mathcal{A}_k(u)\int_{\nu(u)}^{u}\mathcal{A}_j(v)\D
v\D u\right)
\end{split}
\end{equation*}
\end{corollary}
We can now consider the terms in (\ref{eq:termdouint}). We shall
first consider the terms where $j$ and $k$ are not both zero. From
(\ref{eq:airydou}), we see that if $j+k>2$, then the term will be of
order $O\left(N^{-\frac{4}{3}}\right)$. Therefore we shall only
consider the cases where $j+k\leq 2$. Let us now compute the terms
$\mathcal{A}(j,k)$.
\begin{lemma}\label{le:calHjk}
Let $\mathcal{A}_0(j,k)$ be
\begin{equation}\label{eq:calH0}
\mathcal{A}_0(j,k)=\int_{u_-}^{u_+}\mathcal{A}_j(u)\int_{u}^{v_+}\mathcal{A}_k(v)\D
v\D u,
\end{equation}
Then for $j+k\geq 1$, we have
\begin{equation}\label{eq:calHjk}
\begin{split}
\mathcal{A}(j,k)&=\mathcal{A}_0(j,k)+\frac{4^{\frac{1}{3}}(m-n)}{N^{\frac{1}{3}}}\int_{u_-}^{u_+}\mathcal{A}_j(u)\mathcal{A}_k(u)\D
u+O\left(N^{-1}\right)
\end{split}
\end{equation}
\end{lemma}
\begin{proof} This can be proved by using the mean value theorem. By
repeat use of mean value theorem, we have
\begin{equation}\label{eq:MVH}
\int_{\nu(u)}^u\mathcal{A}_k(v)\D
v=\mathcal{A}_k(u)(u-\nu(u))-\frac{1}{2}\mathcal{A}^{\prime}(\xi_u)(u-\nu(u))^2
\end{equation}
for some $\xi_u$ between $u$ and $\nu(u)$. Therefore we have
\begin{equation*}
\begin{split}
\mathcal{A}(j,k)=\mathcal{A}_0(j,k)+\int_{u_-}^{u_+}\mathcal{A}_j(u)\mathcal{A}_k(u)(u-\nu(u))\D
u-\frac{1}{2}\int_{u_-}^{u_+}\mathcal{A}_j(u)\mathcal{A}_k^{\prime}(\xi_u)(u-\nu(u))^2\D
u
\end{split}
\end{equation*}
Note that $\mathcal{A}_j\mathcal{A}_k$ is of order
$N^{-\frac{1}{3}}u^{-\frac{3}{2}-j-k}$ as $u\rightarrow-\infty$ and
decays exponentially as $u\rightarrow+\infty$. Then by Lemma
\ref{le:vcoef}, we see that
\begin{equation*}
\int_{u_-}^{u_+}\mathcal{A}_j(u)\mathcal{A}_k(u)(u-\nu(u))^i\D u
=\frac{4^{\frac{i}{3}}(m-n)^i}{N^{\frac{i}{3}}}\int_{u_-}^{u_+}\mathcal{A}_j\mathcal{A}_k\D
u+O\left(N^{-\frac{4}{3}}\right),
\end{equation*}
where $i=1,2$. Similarly, since $j+k\geq 1$, the function
$\mathcal{A}_j(u)\mathcal{A}_k^{\prime}(\xi_u)$ is of order
$N^{-\frac{1}{3}}u^{-2}$ as $u\rightarrow-\infty$ and decays
exponentially as $u\rightarrow\infty$. Hence the integral involving
$\mathcal{A}_j(u)\mathcal{A}_k^{\prime}(\xi_u)$ is of order
$N^{-1}$. This proves the lemma.
\end{proof}
From this, (\ref{eq:MVH}) and (\ref{eq:vcoef}), we can compute the
terms $AH(j,k)$ and $SH(j,k)$.
\begin{corollary}\label{cor:ASH} Let $AH_0(i,j)$ be
\begin{equation*}
AH_0(i,j)=\frac{1}{2}\left(\mathcal{A}_0(i,j)-\mathcal{A}_0(j,i)\right)
\end{equation*}
Then the terms $AH(j,k)$ and $SH(j,k)$ are of order
\begin{equation*}
\begin{split}
AH(j,k)=AH_0(j,k)+O\left(N^{-1}\right),\quad
SH(j,k)=\frac{4^{\frac{1}{3}}(m-n)}{N^{\frac{1}{3}}}\int_{u_-}^{u_+}\mathcal{A}_j\mathcal{A}_k\D
u+O\left(N^{-1}\right)
\end{split}
\end{equation*}
\end{corollary}
By using this, (\ref{eq:airydou}) and (\ref{eq:vcoef}), together
with the formula for $\mathcal{A}(0,0)$ (\ref{eq:H00}) and
(\ref{eq:Acoef}), we arrive at the following.
\begin{lemma}\label{le:corr1}
Let $n-m=k$. Then we have
\begin{equation}\label{eq:j3i}
\begin{split}
&\mathcal{J}_3-\frac{1}{2}\mathcal{I}_{n,3}\mathcal{I}_{m,3}-\mathcal{F}_{n-m}=k\mathcal{J}_{31}
+k(2(N-n)+k)\mathcal{J}_{32}\\
&+\frac{1}{2}\mathcal{I}_{n,3}\int_{u_-}^{v_-}\mathcal{A}(u)\mathcal{E}_m\D
u +O\left(N^{-\frac{4}{3}+\frac{\varepsilon}{2}}\right)
\end{split}
\end{equation}
where the terms $\mathcal{J}_{3j}$ are given by
\begin{equation}\label{eq:j3j}
\begin{split}
\mathcal{J}_{31}&=-\frac{4^{\frac{1}{3}}
}{2N^{\frac{1}{3}}}AH_0(0,1)
-\frac{4}{N^{\frac{1}{3}}}\int_{\frac{N}{n}c_-}^{4}\frac{\phi}{Nx(4-x)(t-
x)}\D
x-\frac{4^{\frac{1}{3}}}{N^{\frac{1}{3}}}\int_{u_-}^{u_+}\mathcal{A}^2(u)\D
u,\\
\mathcal{J}_{32}&=-\frac{4^{\frac{2}{3}}
}{2N^{\frac{2}{3}}}\left(\int_{u_-}^{u_+}\mathcal{A}_0\mathcal{A}_1\D
u+\frac{3 }{4}AH_0(0,2)\right)
\end{split}
\end{equation}
\end{lemma}
As we shall see, the terms $\mathcal{J}_{31}$ and $\mathcal{J}_{32}$
will in fact not enter into the expression of the kernel $K_1$. What
is important is the structure of equation (\ref{eq:j3i}).

\begin{remark}\label{re:uniformT}
Throughout this section, all the dependence on $T$ are through
factors of $(T- u)^{-\frac{j}{2}}$ for $j>0$, with the exception of
the function $\mathcal{L}(u)$. From the definition of
$\mathcal{L}(u)$ in (\ref{eq:Lu}), we see that it can be written in
the form (\ref{eq:Luord})
\begin{equation*}
\mathcal{L}(u)=\left(\frac{N}{4}\right)^{-\frac{1}{6}}\left(1+O\left(\frac{u}{N^{\frac{2}{3}}}\right)\right)
\end{equation*}
with an error term that is uniform in $T$ as long as $T$ is not on
the integration contour of $u$ and $v$. Therefore all error terms in
this section is uniform in $T$ as $|T|\rightarrow N^{\frac{2}{3}}$
on the contour $\Xi_+$ in Theorem \ref{thm:phdis}.
\end{remark}
\subsubsection{The exponential region} Inside the exponential
region $[4+N^{-\varepsilon},\infty)$, we have the following matching
formula for the polynomials (See Lemma 4.8 of \cite{DGKV}).
\begin{lemma}\label{le:expmatch}
For any $d>0$ there exists a constant $k>0$ such that, uniformly for
$x\in[4+N^{-d},\infty)$, we have
\begin{equation*}
\hat{L}_n(x)=O\left(e^{-k(x-4)n^{\frac{2}{3}}}\right).
\end{equation*}
\end{lemma}
From this, it is clear that the contribution from the exponential
region is of order $O\left(e^{-kN^{\frac{2}{3}-\varepsilon}}\right)$
for some $k>0$.
\subsection{Asymptotics of the skew inner
product}\label{se:asyminner}

We can now compute the asymptotics of the skew inner products
$\left<L_n,L_m\right>_1$. We have
\begin{equation}\label{eq:innersd}
\begin{split}
\left<L_n,L_m\right>_1&=\frac{1}{2}\int_{-\infty}^{\infty}L_n(x)w(x)\D
x\int_{-\infty}^{\infty}L_m(y)w(y)\D
y\\
&-\int_{-\infty}^{\infty}L_n(x)w(x)\int_{x}^{\infty}L_m(y)w(y)\D y\D
x.
\end{split}
\end{equation}
By breaking the range of these integrals into different regions, we
obtain
\begin{equation}\label{eq:innersd1}
\begin{split}
&\frac{N^2i\sqrt{\kappa_{n-1}\kappa_{m-1}}}{nm}\left<L_n,L_m\right>_1=
\sum_{j=1}^3\left(\frac{1}{2}\mathcal{I}_{n,j}\mathcal{I}_{m,j}-\mathcal{J}_j\right)
+\frac{1}{2}\Big(\mathcal{I}_{n,3}\left(\mathcal{I}_{m,2}+\mathcal{I}_{m,1}\right)\\
&-\left(\mathcal{I}_{n,1}+\mathcal{I}_{n,2}\right)\mathcal{I}_{m,3}\Big)+\frac{1}{2}\left(\mathcal{I}_{n,2}\mathcal{I}_{m,1}-\mathcal{I}_{n,1}\mathcal{I}_{m,2}\right)+O\left(e^{-kN^{\frac{2}{3}-\varepsilon}}\right)
\end{split}
\end{equation}
for some $k>0$. Then, as $\mathcal{I}_{n,2}$ and $\mathcal{I}_{m,2}$
are of order $O\left(N^{-\frac{7}{8}}\right)$ (see
(\ref{eq:bulksing})), we see that the last term in the above sum is
of order $O\left(N^{-\frac{11}{8}}\right)$. By using mean value
theorem and the asymptotic formulae for $\mathcal{A}(u)$, we see
that $\int_{u_-}^{v_-}\mathcal{A}(u)\mathcal{E}_m\D
u=O\left(N^{-\frac{7}{8}}\right)$. From these, (\ref{eq:singleB}),
(\ref{eq:bulksing}) and (\ref{eq:singairy}), we obtain
\begin{equation*}
\begin{split}
&\mathcal{I}_{n,3}\left(\mathcal{I}_{m,2}+\mathcal{I}_{m,1}\right)
-\left(\mathcal{I}_{n,1}+\mathcal{I}_{n,2}\right)\mathcal{I}_{m,3}-\mathcal{I}_{n,3}\int_{u_-}^{v_-}\mathcal{A}(u)\mathcal{E}_m\D
u\\
&=\sqrt{\frac{2\pi}{Nt}}\left((-1)^m\mathcal{I}_{n,3}-(-1)^n\mathcal{I}_{m,3}\right)+O\left(N^{-\frac{25}{24}}\right).
\end{split}
\end{equation*}
From this and (\ref{eq:singleB}), (\ref{eq:doubleB}),
(\ref{eq:doubulk1}) and Lemma \ref{le:corr1}, we arrive at the
following.
\begin{proposition}\label{pro:asyminner}
The product $\left<L_n,L_m\right>_1$ is given by
\begin{equation}\label{eq:asyminner}
\begin{split}
&\frac{N^2i\sqrt{\kappa_{n-1}\kappa_{m-1}}}{nm}\left<L_n,L_m\right>_1=
-\Bigg(k\mathcal{J}_{31}+k(2(N-n)+k)\mathcal{J}_{32}\Bigg)+k\mathcal{J}_{B1}+\\
&\mathcal{J}_{B2}+\mathcal{J}_{B3,k}+\left\{
   \begin{array}{ll}
     (-1)^{n}\Bigg(I_0+\left(2(n-N)-k\right)I_1\Bigg), & \hbox{$k$ odd;} \\
     (-1)^{n+1}kI_1, & \hbox{$k$ even.}
   \end{array}
 \right.+O\left(N^{-\frac{25}{24}}\right)
\end{split}
\end{equation}
where $\mathcal{J}_{3j}$ are given in (\ref{eq:j3j}) and
$\mathcal{J}_{B1}$, $\mathcal{J}_{B2}$ and $I_j$ are given by
\begin{equation}\label{eq:jBI}
\begin{split}
I_0&=-\sqrt{\frac{2\pi}{Nt}}\int_{u_-}^{u_+}\mathcal{A}(u)\D u,\quad
I_1=-\frac{\sqrt{2\pi}
}{4\sqrt{Nt}}\left(\frac{4}{N}\right)^{\frac{1}{3}}\int_{u_-}^{u_+}\mathcal{A}_1(u)\D
u,\\
\mathcal{J}_{B1}&=\frac{2\pi}{N^{\frac{2}{3}}4^{\frac{1}{3}}\sqrt{tT}}-\frac{2\pi}{N
},\quad
\mathcal{J}_{B2}=\frac{2\pi}{N },\\
\end{split}
\end{equation}
while $\mathcal{J}_{B3,k}$ is given by
\begin{equation}\label{eq:jB3}
\begin{split}
\mathcal{J}_{B3,1}&=0,\quad \mathcal{J}_{B3,2}=\frac{2\pi
4^{\frac{1}{3}}}{N^{\frac{4}{3}} }\sqrt{\frac{T}{t}}.
\end{split}
\end{equation}
\end{proposition}
We can simplify the expression further by the observation that
$\left<L_n,L_m\right>_1=0$ whenever $n-m=1$ and $n$ is even. (See
Corollary \ref{cor:linear}) Taking $k=1$ and $n$ even in
(\ref{eq:asyminner}), we have
\begin{equation*}
\begin{split}
-\Bigg(\mathcal{J}_{31}
+(2(N-n)+1)\mathcal{J}_{32}\Bigg)+\mathcal{J}_{B1}+\mathcal{J}_{B2}
+I_0+\left(2(n-N)-1\right)I_1=O\left(N^{-\frac{25}{24}}\right)
\end{split}
\end{equation*}
As this holds for any finite $N-n$ as long as $n$ is even, we obtain
the following relations.
\begin{equation}\label{eq:IJrel}
\begin{split}
\mathcal{J}_{32}=-I_1+O\left(N^{-\frac{25}{24}}\right),\quad
-\mathcal{J}_{31}+\mathcal{J}_{B1}+I_0=-\mathcal{J}_{B2}+O\left(N^{-\frac{25}{24}}\right).
\end{split}
\end{equation}
Note that although in the above equation, integration limits
$u_{\pm}$ depend on $n$, the effect of changing these limits will
only result in terms of order $O\left(N^{-\frac{25}{24}}\right)$ and
hence we can consider them as fixed under the change of $n$.

Let us now compute the factor $\kappa_{n-1}\kappa_{m-1}$. We have
$\kappa_{n}=-\frac{2\pi i}{h_{n,0}}$. By (\ref{eq:hn}), we see that
$h_{n,0}/h_{m,0}=1+O\left(N^{-1}\right)$. From this, Proposition
\ref{pro:asyminner} and (\ref{eq:IJrel}), we see that the skew inner
products that we need have the following asymptotics.
\begin{equation}\label{eq:prodasym}
\begin{split}
&\frac{2\pi}{h_{N,0}}\left<L_{N-1},L_{N-2}\right>_1=
-2I_0+6I_1+O\left(N^{-\frac{25}{24}}\right),\\
&\frac{2\pi}{h_{N,0}}\left<L_{N},L_{N-2}\right>_1=
-2I_0+2I_1-\mathcal{J}_{B2}
+\mathcal{J}_{B3,2}+O\left(N^{-\frac{25}{24}}\right).
\end{split}
\end{equation}
\begin{remark} From (\ref{eq:prodasym}), we see that for large
enough $N$, the products $\left<L_{N-1},L_{N-2}\right>_1$ and
$\left<L_{N-3},L_{N-4}\right>_1$ will be non-zero if
$\int_{-\infty}^{\infty}\frac{Ai}{(T-u)^{\frac{1}{2}}}\D u\neq 0$.
Since this is an analytic function with jump on $\mathbb{R}$ that is
not identically zero, we can choose the contour $\Xi_+$ in Theorem
\ref{thm:phdis} such that
$\int_{-\infty}^{\infty}\frac{Ai}{(T-u)^{\frac{1}{2}}}\D u\neq 0$ on
$\Xi_+$. We can therefore assume that
$\int_{-\infty}^{\infty}\frac{Ai}{(T-u)^{\frac{1}{2}}}\D u\neq 0$.
\end{remark}

Note that, by Remark \ref{re:uniformT}, these error terms are
uniform in $T$ as $|T|\rightarrow N^{\frac{2}{3}}$. Let us look at
the behavior of the skew products when $|T|\rightarrow
N^{\frac{2}{3}}$ and when $t$ remain finite.

First by (\ref{eq:jBI}), we see that as $T\rightarrow\infty$, the
integrals $I_j$ are of the following orders as $T\rightarrow\infty$.
\begin{equation}\label{eq:ITinfty}
\begin{split}
I_j=O\left(T^{-\frac{2j+1}{2}}t^{-\frac{1}{2}}N^{-\frac{2+j}{3}}\right).
\end{split}
\end{equation}
(The statement is clear if $|T|>u_-$. If $|T|<u_-$, then by breaking
the range of the integral into $(u_-,-|T|^{\frac{1}{2}})$ and
$(-|T|^{\frac{1}{2}},u_+)$ and integrate by parts the integral over
$(u_-,-|T|^{\frac{1}{2}})$ using (\ref{eq:Aasym}), one can check
that (\ref{eq:ITinfty}) is correct.)

By using $t=4 +\left(\frac{4}{N}\right)^{\frac{2}{3}}T$ in
(\ref{eq:jB3}), we see that $\mathcal{J}_{B3,2}$ is of the following
order.
\begin{equation}\label{eq:TjB3}
\begin{split}
\mathcal{J}_{B3,2}= O\left(N^{-1}\left(\frac{T}{N^{\frac{2}{3}}}\right)^{\frac{1}{2}}\right), \quad \hbox{$|T|\rightarrow N^{\frac{2}{3}}$ and $t$ finite.} \\
\end{split}
\end{equation}
From these we obtain the behavior of the skew products as
$T\rightarrow\infty$.
\begin{lemma}\label{le:Tskew}
The skew products $\left<L_n,L_m\right>_1$ in (\ref{eq:prodasym})
are of the following orders as $T\rightarrow\infty$.
\begin{equation}\label{eq:Tskew}
\begin{split}
&\frac{2\pi}{h_{N,0}}\left<L_{n},L_{m}\right>_1=
                                                    O\left(T^{-\frac{1}{2}}N^{-\frac{2}{3}}\right), \quad \hbox{$|T|\rightarrow N^{\frac{2}{3}}$ and $t$ finite.} \\
\end{split}
\end{equation}
\end{lemma}

We can now compute the asymptotic of the kernel.
\subsection{Asymptotics of the kernel $S_1(x,y)$} First let us used
the results in the previous sections to compute the correction term
to the kernel. Let us define the variables $\xi_1$ and $\xi_2$ to be
\begin{equation}\label{eq:scalevar}
\xi_1=\left(\frac{N}{4}\right)^{\frac{2}{3}}\left(x-4\right),\quad
\xi_2=\left(\frac{N}{4}\right)^{\frac{2}{3}}\left(y-4\right).
\end{equation}
We will assume that $\xi_1$ and $\xi_2$ are bounded from below.
First we need to compute the integrals $\epsilon(\pi_{n,1}w)$, which
is a linear combinations of the integrals of the $L_nw$. We have
\begin{equation*}
\epsilon\left(L_nw\right)(y)=-\int_{y}^{\infty}L_n(s)w(s)\D
s+\frac{1}{2}\int_{-\infty}^{\infty}L_n(s)w(s)\D s.
\end{equation*}
To compute the first term, let us first assume
$\xi_2<N^{\frac{1}{8}}$. Then we have, by the asymptotic formula of
the $L_n(x)$ inside the Airy region, and the estimates \cite{AbSt},
\begin{equation}\label{eq:aest}
|Ai(t)|\leq Ce^{-(2/3)t^{3/2}},\quad
|Ai^{\prime}(t)|<C(1+t^{1/4})e^{-2/3t^{3/2}},\quad t>0,
\end{equation}
the following.
\begin{equation}\label{eq:eL}
\frac{\sqrt{i\kappa_{n-1}}N}{n}\int_{y}^{\infty}L_n(s)w(s)\D
s=\int_{u\left(\frac{N}{n}y\right)}^{\infty}
\mathcal{A}\mathcal{E}_n\D
u+O\left(e^{-k\xi_2}N^{-\frac{7}{6}}\right)
\end{equation}
for some $k>0$, where $u=f_n(y)$. Note that by (\ref{eq:aest}), we
in fact have an exponential decay of order $e^{-k\xi_2^{3/2}}$ in
the error term above. However, the decay $e^{-k\xi_2}$ will be
sufficient for our purpose. The lower limit
$u\left(\frac{N}{n}y\right)$ can be expressed in terms of $\xi_2$ as
follows.
\begin{equation*}
u\left(\frac{N}{n}y\right)=\xi_2+\left(\frac{4}{N}\right)^{\frac{1}{3}}(N-n)
-\left(\frac{4}{N}\right)^{\frac{2}{3}}\frac{\xi_2^2}{20}+O\left(\left(\frac{\xi_2}{N^{\frac{2}{3}}}\right)^jN^{-\frac{1}{3}}\right).
\end{equation*}
where $j>0$. By using mean value theorem and (\ref{eq:aest}), we can
write the integral as
\begin{equation}\label{eq:asymint}
\begin{split}
\int_{u\left(\frac{N}{n}y\right)}^{\infty}
\mathcal{A}\mathcal{E}_n\D
u&=\mathcal{L}_0(\xi_2)+\left(\frac{4}{N}\right)^{\frac{1}{3}}(N-n)\mathcal{L}_1(\xi_2)\\
&+\frac{1}{2}\left(\frac{4}{N}\right)^{\frac{2}{3}}(N-n)^2\mathcal{L}_2(\xi_2)+O\left(e^{-k\xi_2}/N^{\frac{7}{8}}\right),
\end{split}
\end{equation}
for some $k>0$. The $\mathcal{L}_j$ in (\ref{eq:asymint}) are given
by
\begin{equation}\label{eq:calA}
\begin{split}
\mathcal{L}_0&=\int_{\xi_2}^{\infty}\mathcal{A}\D u
+\frac{4^{\frac{2}{3}}\xi_2^2}{20N^{\frac{2}{3}}}\mathcal{A}(\xi_2),\quad
\mathcal{L}_1=-\mathcal{A}(\xi_2)-\int_{\xi_2}^{\infty}\frac{
\mathcal{A}(u)}{2(T- u)}\D
u,\\
\mathcal{L}_2&=-\mathcal{A}^{\prime}(\xi_2)+\frac{
\mathcal{A}(\xi_2)}{T- \xi_2}
+\frac{3}{4}\int_{\xi_2}^{\infty}\frac{\mathcal{A}(u)}{(T- u)^2}\D
u.
\end{split}
\end{equation}
On the other hand, if $\xi_2>N^{\frac{1}{8}}$, then from Lemma
\ref{le:expmatch} and the asymptotic formula for the Airy function,
we see that
\begin{equation*}
\frac{\sqrt{i\kappa_{n-1}}N}{n}\int_{y}^{\infty}L_n(s)w(s)\D
s=O\left(e^{-k\xi_2}\right),\quad
\int_{u\left(\frac{N}{n}y\right)}^{\infty}
\mathcal{A}\mathcal{E}_n\D u=O\left(e^{-k\xi_2}\right)
\end{equation*}
for some constant $k>0$. Therefore if we choose the constant $k$ in
(\ref{eq:asymint}) to be small enough, then (\ref{eq:asymint})
remains valid for all $\xi_2$ bounded below.

Similarly, the second term in (\ref{eq:eL}) is given by
\begin{equation*}
\frac{\sqrt{i\kappa_{n-1}}}{2}\int_{-\infty}^{\infty}L_n(s)w(s)\D
s=\frac{(-1)^n\sqrt{\pi}}{\sqrt{2Nt}}+\frac{1}{2}\int_{u_-}^{u_+}\mathcal{A}\mathcal{E}_n(u)\D
u+O\left(N^{-\frac{7}{8}}\right).
\end{equation*}
Let $u_{-,N}$ be $f_N(c_-)$, then changing the lower limit in the
above integral into $u_{-,N}$ will only result in an error term of
order
$O\left(N^{-1+\frac{3\varepsilon}{4}}\right)=O\left(N^{-\frac{7}{8}}\right)$.
Similarly, changing the upper limit to $+\infty$ will only result in
an exponentially small error term. We can therefore change the lower
limit in the integration to $u_{-,N}$ and the upper limit to
$+\infty$. Let $\xi_{2,N}$ be the value of $\xi_2$ at $u=u_{-,N}$,
then we have
\begin{equation}\label{eq:asymint1}
\begin{split}
\sqrt{i\kappa_{n-1}}\epsilon\left(L_nw\right)(y)&=
\Psi_0(\xi_2)+\left(\frac{4}{N}\right)^{\frac{1}{3}}(N-n)\Psi_1(\xi_2)+\frac{1}{2}\left(\frac{4}{N}\right)^{\frac{2}{3}}(N-n)^2\Psi_2(\xi_2)\\
&+\frac{(-1)^n\sqrt{\pi}}{\sqrt{2Nt}}+O\left(N^{-\frac{7}{8}}\right),
\end{split}
\end{equation}
where $\Psi_j(\xi_2)$ is given by
\begin{equation*}
\begin{split}
\Psi_j(\xi_2)=\frac{1}{2}\mathcal{L}_j(\xi_{2,N})-\mathcal{L}_j(\xi_2),\quad
j=0,1,2.
\end{split}
\end{equation*}
Similarly, the orthogonal polynomials $L_n(x)$ are given by
\begin{equation}\label{eq:opasym}
\begin{split}
&L_n(x)w(x)=\frac{\sqrt{h_{N,0}}}{2\left(T-
\xi_1\right)^{\frac{1}{2}}}
\Bigg(\left(\frac{N}{4}\right)^{\frac{1}{2}}Ai(\xi_1)-\left(\frac{N}{4}\right)^{\frac{1}{6}}\left(\frac{\alpha+1}{2}+
(n-N)\right)Ai^{\prime}(\xi_1)\\
&+\left(\frac{N}{4}\right)^{-\frac{1}{6}}\left(\frac{(N-n)^2}{2}Ai^{\prime\prime}(\xi_1)
+f_1(\xi_1)+(N-n)f_2(\xi_2)\right)
+O\left(\frac{e^{-k\xi_1}}{N^{\frac{1}{6}}}\right)\Bigg)
\end{split}
\end{equation}
where $f_1$ and $f_2$ are functions that are independent on $N$, $n$
and $T$. Then by using (\ref{eq:k1form}), we obtain the asymptotics
for the correction kernel $K_1(x,y)$.
\begin{equation}\label{eq:mainker}
\begin{split}
&K_1(x,y)= \frac{Ai(\xi_1)\Psi_0(\xi_2)}{(T-
\xi_1)^{\frac{1}{2}}}\mathcal{Q}_{00}
+\frac{Ai^{\prime}(\xi_1)\Psi_0(\xi_2)}{(T- \xi_1)^{\frac{1}{2}}}\mathcal{Q}_{10}\\
&+\frac{Ai(\xi_1)\Psi_1(\xi_2)}{(T-
\xi_1)^{\frac{1}{2}}}\mathcal{Q}_{01}
+\frac{Ai^{\prime}(\xi_1)\Psi_1(\xi_2)}{(T-
\xi_1)^{\frac{1}{2}}}\mathcal{Q}_{11}
+\frac{Ai(\xi_1)\Psi_2(\xi_2)}{(T-
\xi_1)^{\frac{1}{2}}}\mathcal{Q}_{02}+
\frac{Ai^{\prime\prime}(\xi_1)\Psi_0(\xi_2)}{(T- \xi_1)^{\frac{1}{2}}}\mathcal{Q}_{20}\\
&+\frac{Ai(\xi_1)}{(T- \xi_1)^{\frac{1}{2}}}\mathcal{Q}_{03}
+\frac{Ai^{\prime}(\xi_1)}{(T-
\xi_1)^{\frac{1}{2}}}\mathcal{Q}_{13}+O\left(\frac{e^{-k\xi_1}N^{\frac{5}{8}}}{1+\sqrt{|T|}}\right)
+O\left(e^{-k\xi_1}N^{\frac{1}{3}}\right)
\end{split}
\end{equation}
The coefficients $\mathcal{Q}_{jk}$ are given by
\begin{equation}\label{eq:qjk}
\begin{split}
\mathcal{Q}_{00}&=\left(\frac{N}{4}\right)^{-\frac{1}{6}}\frac{NT
}{4\sqrt{2\pi}} ,\quad
\mathcal{Q}_{10}=\left(\frac{N}{4}\right)^{\frac{1}{6}}\frac{N
\left(4I_1+\mathcal{J}_{B2}-\mathcal{J}_{B3,2}\right)}
{8\sqrt{2\pi}I_0}+\left(\frac{4}{N}\right)^{\frac{1}{2}}\frac{NT}{4\sqrt{2\pi}},\\
\mathcal{Q}_{01}&=-\left(\frac{N}{4}\right)^{\frac{1}{6}}\frac{N
(4I_1+\mathcal{J}_{B2}-\mathcal{J}_{B3,2})} {8\sqrt{2\pi}I_0},\quad
\mathcal{Q}_{11}=\left(\frac{N}{4}\right)^{-\frac{1}{6}}\frac{N }{4\sqrt{2\pi}},\\
\mathcal{Q}_{02}&=-\left(\frac{N}{4}\right)^{-\frac{1}{6}}\left(\frac{N
}{2\sqrt{2\pi}}
+\frac{3N (4I_1+\mathcal{J}_{B2}-\mathcal{J}_{B3,2})}{16\sqrt{2\pi}I_0}\right),\\
\mathcal{Q}_{20}&=\left(\frac{4}{N}\right)^{\frac{5}{6}}\frac{NT}{8\sqrt{2\pi}}
+\left(\frac{N}{4}\right)^{-\frac{1}{6}}\frac{3N \left(4I_1+\mathcal{J}_{B2}-\mathcal{J}_{B3,2}\right)}{16I_0\sqrt{2\pi}},\\
\mathcal{Q}_{03}&=-\frac{N
(4I_1+\mathcal{J}_{B2}-\mathcal{J}_{B3,2})}{16\sqrt{t}I_0}
+\left(\frac{4}{N}\right)^{\frac{2}{3}}\frac{TN}{16\sqrt{t}},\\
\mathcal{Q}_{13}&=\left(\frac{4}{N}\right)^{\frac{1}{3}}\frac{N}{8\sqrt{t}}\left(-1+\frac{T}{2}\left(\frac{4}{N}
\right)^{\frac{2}{3}}-\frac{3
(4I_1+\mathcal{J}_{B2}-\mathcal{J}_{B3,2})}{4I_0}\right).
\end{split}
\end{equation}
From (\ref{eq:ITinfty}) and (\ref{eq:TjB3}), we see that as
$T\rightarrow\infty$, the kernel $K_1$ has the following behavior
\begin{equation*}
K_1(\xi_1,\xi_2)=O\left(N^{\frac{2}{3}}\right), \quad
\hbox{uniformly as $|T|\rightarrow N^{\frac{2}{3}}$ and $t$ finite.}
\end{equation*}
In fact, by using (\ref{eq:prodasym}), (\ref{eq:Tskew}) and
(\ref{eq:k1form}), we obtain the following.
\begin{lemma}\label{le:kTinfty}
As $T\rightarrow\infty$ and $t$ remains bounded, the kernel $K_2$
and $K_1$ become the Airy kernels in the large $N$ limit.
\begin{equation}\label{eq:kTinfty}
\begin{split}
\left(\frac{4}{N}\right)^{\frac{2}{3}}K_2(\xi_1,\xi_2)&=\frac{Ai(\xi_1)Ai^{\prime}(\xi_2)-Ai(\xi_2)Ai^{\prime}(\xi_1)}{\xi_1-\xi_2}+O\left(\frac{e^{-k\left(\xi_1+\xi_2\right)}}{TN^{\frac{1}{3}}}\right),\\
\left(\frac{4}{N}\right)^{\frac{2}{3}}K_1(\xi_1,\xi_2)&=\frac{1}{2}Ai(\xi_1)\int_{-\infty}^{\xi_2}Ai(u)\D
u+o(1)e^{-k\xi_1}.
\end{split}
\end{equation}
for some $k>0$.
\end{lemma}
\begin{proof} The statement for the kernel $K_2$ follows immediately
from the representation (\ref{eq:k2}) and the Airy asymptotics of
the Laguerre polynomials. To prove the statement for the correction
term $K_1$, first note that, by (\ref{eq:prodasym}), we see that
\begin{equation}\label{eq:lnn}
\begin{split}
\frac{2\pi}{h_{N,0}}\left<L_{N},L_{N-2}\right>_1&=\frac{2\pi(t-2
)}{N \sqrt{t(t-4 )}}
-\frac{2\pi}{N }+O\left(N^{-\frac{25}{24}}\right),\\
\frac{2\pi}{h_{N,0}}\left<L_{N-1},L_{N-2}\right>_1&=\frac{4\pi}{N\sqrt{t(t-4
)}} +O\left(N^{-\frac{25}{24}}\right)
\end{split}
\end{equation}
uniformly for $|T|<cN^{\frac{2}{3}}$. By dividing the range of
integration in $\int_{u_-}^{u_+}\frac{Ai}{(T- u)^{\frac{1}{2}}}\D u$
into $[u_-,-|T|^{\frac{1}{2}}]$ and $[-|T|^{\frac{1}{2}},u_+]$, we
see that
\begin{equation*}
\int_{u_-}^{u_+}\frac{Ai}{(T- u)^{\frac{1}{2}}}\D
u=\frac{1}{T^{\frac{1}{2}}}\left(\int_{-\infty}^{\infty}Ai(u)\D
u+O\left(T^{-\frac{3}{8}}\right)\right)=\frac{1}{T^{\frac{1}{2}}}\left(1+O\left(T^{-\frac{3}{8}}\right)\right),
\end{equation*}
From this, (\ref{eq:lnn}), (\ref{eq:asymint1}), (\ref{eq:opasym})
and (\ref{eq:k1form}), we see that in this limit, $K_1$ is given by
\begin{equation*}
\begin{split}
&\left(\frac{4}{N}\right)^{\frac{2}{3}}K_1(\xi_1,\xi_2)=\frac{1}{2}Ai(\xi_1)\left(\frac{1}{2}-\int_{\xi_2}^{\infty}Ai(u)\D
u+O\left(T^{-\frac{3}{8}}\right)+O\left(N^{-\frac{1}{3}}\right)\right)\\
&+\left(\frac{1}{4}+O\left(N^{-\frac{9}{24}}T^{\frac{1}{2}}\right)+O\left(T^{-\frac{3}{8}}\right)\right)
Ai(\xi_1)\left(1+O\left(N^{-\frac{1}{3}}\right)\right)\\
&=\frac{1}{2}Ai(\xi_1)\left(\int_{-\infty}^{\xi_2}Ai(u)\D
u+O\left(T^{-\frac{3}{8}}\right)+O\left(N^{-\frac{9}{24}}T^{\frac{1}{2}}\right)+O\left(N^{-\frac{1}{3}}\right)\right)
\end{split}
\end{equation*}
This proves the lemma.
\end{proof}

By using the explicit expressions for $\mathcal{A}$ and $\Psi_j$, we
obtain the asymptotic formula for the kernel when $T$ is finite.
\begin{proposition}\label{pro:asymker}
Let $T$ be of order $N^{\frac{1}{3}-c}$ for some positive $0<c\leq
\frac{1}{3}$, then for $\xi_1$ and $\xi_2$ bounded from below, the
kernel $K_1(\xi_1,\xi_2)$ is given by
\begin{equation}\label{eq:asymk1}
\begin{split}
&\left(\frac{4}{N}\right)^{\frac{2}{3}}K_1(\xi_1,\xi_2)=K_{1,\infty}(\xi_1,\xi_2)+O\left(N^{-\frac{1}{24}}e^{-k\xi_1}\right),\\
&K_{1,\infty}(\xi_1,\xi_2)=\Bigg(\frac{T}{2}
H_0(\xi_1)\int_{-\infty}^{\xi_2}H_0\D u
+\frac{1}{2}H_1(\xi_1)\int_{-\infty}^{\xi_2}H_1(u)\D u\\
&\mathcal{B}_1\left(H_1(\xi_1)\int_{-\infty}^{\xi_2}H_0(u)\D
u-H_0(\xi_1)\int_{-\infty}^{\xi_2}H_1(u)\D
u\right)\\
&-H_0(\xi_1)\int_{-\infty}^{\xi_2}H_2(u)\D u
+\mathcal{B}_2H_0(\xi_1)\Bigg).
\end{split}
\end{equation}
for some $k>0$, where $H_j$, $\mathcal{B}_1$ and $\mathcal{B}_2$ are
given by
\begin{equation}\label{eq:C1}
\begin{split}
H_j(u)&=\frac{Ai^{(j)}(u)}{(T-u)^{\frac{1}{2}}},\quad
\mathcal{B}_1=-\frac{\int_{-\infty}^{\infty}H_1(u)\D
u+1}{2\int_{-\infty}^{\infty}H_0(u)\D u},\\
\mathcal{B}_2&=-\frac{\mathcal{B}_1}{2}-\frac{T}{4}\int_{-\infty}^{\infty}H_0\D
u+\frac{\mathcal{B}_1}{2}\int_{-\infty}^{\infty}H_1\D
u+\frac{1}{2}\int_{-\infty}^{\infty}H_2\D u,\quad
\end{split}
\end{equation}
\end{proposition}
As pointed out in \cite{DG}, \cite{DG2} and \cite{DGKV}, the
eigenvalue statistics will not be affected by the rescaling
\begin{equation}\label{eq:matker2}
K\mapsto
\left(\frac{N}{4}\right)^{-\frac{2}{3}\sigma_3}K\left(\frac{N}{4}\right)^{\frac{2}{3}\sigma_3}
\end{equation}
of the matrix kernel. By rescaling the kernel in this way, all the
entries will have the same order in the large $N$ limit. From now
on, we shall use this rescaled kernel and denote it also by $K$.

From (\ref{eq:asymk1}), we obtain the following estimate for the
rescaled matrix kernel $K$.
\begin{proposition}\label{pro:asyker}
Let $K_{\infty}$ be the $2\times 2$ matrix whose entries are given
by
\begin{equation}\label{eq:kinfty}
\begin{split}
K_{\infty,11}(\xi_1,\xi_2)&=K_{\infty,22}(\xi_2,\xi_1)=K_{1,\infty}(\xi_1,\xi_2)+K_{2,\infty}(\xi_1,\xi_2),\\
K_{\infty,12}(\xi_1,\xi_2)&=-\frac{\p K_{11,\infty}}{\p\xi_2},\quad
K_{\infty,21}(\xi_1,\xi_2)=-\int_{\xi_1}^{\xi_2}K_{11,\infty}(u,\xi_2)\D
u,
\end{split}
\end{equation}
where $K_{2,\infty}(\xi_1,\xi_2)$ is given by
\begin{equation*}
\begin{split}
K_{2,\infty}(\xi_1,\xi_2)=\left(\frac{T- \xi_2}{T-
\xi_1}\right)^{\frac{1}{2}}\frac{Ai(\xi_1)Ai^{\prime}(\xi_2)-Ai(\xi_2)Ai^{\prime}(\xi_1)}{\xi_1-\xi_2}.
\end{split}
\end{equation*}
Suppose $T$ is of order $N^{\frac{1}{3}-c}$ for some positive
$0<c\leq \frac{1}{3}$ and that $\xi_1$ and $\xi_2$ are bounded from
below. Let $K(\xi_1,\xi_2)$ be the rescaled matrix kernel in
(\ref{eq:matker2}), then there exists $k>0$ such that
\begin{equation*}
\begin{split}
\left(\frac{4}{N}\right)^{\frac{2}{3}}K(\xi_1,\xi_2)&=K_{\infty}(\xi_1,\xi_2)+\begin{pmatrix}O\left(N^{-\frac{1}{24}}e^{-k\xi_1}\right)&
O\left(N^{-\frac{1}{24}}e^{-k\left(\xi_1+\xi_2\right)}\right)\\
O\left(N^{-\frac{1}{24}}\right)&O\left(N^{-\frac{1}{24}}e^{-k\xi_2}\right)
\end{pmatrix}
\end{split}
\end{equation*}
\end{proposition}
\begin{proof} The statement for the $11^{th}$ and $22^{th}$ entries
follows immediately from (\ref{eq:asymk1}), the representation
(\ref{eq:k2}) and the asymptotics of the Laguerre polynomials inside
the Airy region. The statement for the $12^{th}$ entry follows by
replacing $\epsilon(L_nw)$ in (\ref{eq:asymint1}) by the asymptotics
of the polynomials in (\ref{eq:opasym}) in the derivation of
(\ref{eq:mainker}). The computation is the same as the derivation of
(\ref{eq:mainker}) and we shall not carry out the details here. To
obtain the results for the $21^{th}$ entry, we use the fact that
$\epsilon\left(S_1\right)(x,y)$ is skew symmetric to obtain (See
\cite{DG}, \cite{DG2}, \cite{DGKV})
\begin{equation*}
\epsilon\left(S_1\right)(x,y)=\epsilon\left(S_1\right)(x,y)-\epsilon\left(S_1\right)(y,y)=
-\int_{x}^{y}S_1(t,y)\D t.
\end{equation*}
The statement for the $21^{th}$ entry then follows from integrating
(\ref{eq:asymk1}) and the asymptotic formula for $K_2$.
\end{proof}
A similar statement can be obtained for the Airy kernels when
$T\rightarrow\infty$.
\begin{lemma}\label{le:airymatker}
Let $K_{airy}(\xi_1,\xi_2)$ be the following matrix kernel
\begin{equation*}
\begin{split}
K_{airy,11}(\xi_1,\xi_2)&=K_{airy,22}(\xi_2,\xi_1)=\frac{Ai(\xi_1)Ai^{\prime}(\xi_2)-Ai(\xi_2)Ai^{\prime}(\xi_1)}{\xi_1-\xi_2}
+\frac{1}{2}Ai(\xi_1)\int_{-\infty}^{\xi_2}Ai(u)\D u,\\
K_{airy,12}(\xi_1,\xi_2)&=-\frac{\p}{\p\xi_2}K_{airy,11}(\xi_2,\xi_1),\quad
K_{airy,21}(\xi_1,\xi_2)=-\int_{\xi_1}^{\xi_2}K_{airy,11}(u,\xi_2)\D
u.
\end{split}
\end{equation*}
Then for $T\rightarrow\infty$ with $t$ finite and $\xi_1$ and
$\xi_2$ bounded from below, there exists $k>0$ such that the
rescaled matrix kernel $K(\xi_1,\xi_2)$ in (\ref{eq:matker2}) is of
the following order as $N\rightarrow\infty$.
\begin{equation*}
\begin{split}
\left(\frac{4}{N}\right)^{\frac{2}{3}}K(\xi_1,\xi_2)&=K_{airy}(\xi_1,\xi_2)+
\begin{pmatrix}o\left(1\right)e^{-k\xi_1}&
o\left(1\right)e^{-k\left(\xi_1+\xi_2\right)}\\
o\left(1\right)&o\left(1\right)e^{-k\xi_2}
\end{pmatrix}
\end{split}
\end{equation*}
\end{lemma}
In order to show that the convergence of the Fredholm determinant,
we need the following bounds on the derivatives of the kernel $K_2$.
\begin{lemma}\label{le:derker}
For $T$ of order $N^{\frac{1}{3}-c}$ for some $0<c\leq\frac{1}{3}$,
we have
\begin{equation*}
\left|\frac{\p^l}{\p\xi_1^l}\frac{\p^j}{\p\xi_1^j}\left(\left(\frac{4}{N}\right)^{\frac{2}{3}}
K_2(\xi_1,\xi_2)-K_{2,\infty}(\xi_1,\xi_2)\right)\right|=O\left(N^{-\frac{2}{3}}e^{-k\left(\xi_1+\xi_2\right)}\right)
\end{equation*}
for some $k>0$ and $l$, $j=0,1$.
\end{lemma}
\begin{proof} The lemma is an immediate consequence of the following results in
\cite{DG2} (\cite{DG2}, (3.8))
\begin{equation}\label{eq:DG2}
\left|\frac{\p^l}{\p\xi_1^l}\frac{\p^j}{\p\xi_1^j}\left(\left(\frac{4}{N}\right)^{\frac{2}{3}}
K_{lag}(\xi_1,\xi_2)-K_{2,airy}(\xi_1,\xi_2)\right)\right|=O\left(N^{-\frac{2}{3}}e^{-k\left(\xi_1+\xi_2\right)}\right)
\end{equation}
where $K_{2,airy}$ is the Airy kernel
\begin{equation}\label{eq:airyker}
K_{2,airy}(\xi_1,\xi_2)=\frac{Ai(\xi_1)Ai^{\prime}(\xi_2)-Ai(\xi_2)Ai^{\prime}(\xi_1)}{\xi_1-\xi_2}.
\end{equation}
and $K_{lag}$ is the Christoffel Darboux kernel of the Laguerre
polynomials
\begin{equation*}
K_{lag}(x,y)=w_0^{\frac{1}{2}}(x)w_0^{\frac{1}{2}}(y)\kappa_{N-1}^2\frac{L_{N}(x)L_{N-1}(y)-L_N(y)L_{N-1}(x)}{x-y}
\end{equation*}
and $x=4+\xi_1\left(4/N\right)^{\frac{2}{3}}$,
$y=4+\xi_2\left(4/N\right)^{\frac{2}{3}}$. As $K_2$ is the conjugate
to $K_{lag}$ and $K_{2,\infty}$ is the conjugate to $K_{2,airy}$,
\begin{equation*}
K_{2}=\left(\frac{y(t- y)}{x(t-
x)}\right)^{\frac{1}{2}}K_{lag},\quad K_{2,\infty}=\left(\frac{T-
\xi_2}{T- \xi_1}\right)^{\frac{1}{2}}K_{2,airy},
\end{equation*}
the lemma follows from (\ref{eq:DG2}).
\end{proof}
The corresponding statement when $T\rightarrow\infty$ follows from
the same argument but with $K_{2,\infty}$ replaced by $K_{2,airy}$.

With the estimates in Proposition \ref{pro:asyker} and Lemma
\ref{le:derker}, we can obtain the following asymptotic result for
the determinant $\det_2\left(I-\chi K\chi\right)$.
\begin{proposition}\label{pro:largest}
Let $\zeta=(z-4)\left(N/4\right)^{\frac{2}{3}}$,
$\xi_1=(x-4)\left(N/4\right)^{\frac{2}{3}}$,
$\xi_2=(y-4)\left(N/4\right)^{\frac{2}{3}}$, $g(\xi)=\sqrt{1+\xi^2}$
and $G=\diag(g,g^{-1})$, then as $N\rightarrow\infty$ and $T$ of
order up to $o\left(N^{\frac{1}{3}}\right)$, we have
\begin{equation*}
\sqrt{\mathrm{det}_2\left(I-\chi_zG(\xi_1)KG^{-1}(\xi_2)\chi_z\right)}=
    \sqrt{\mathrm{det}_2\left(I-\chi_{\zeta}G(\xi_1)K_{\infty}G^{-1}(\xi_2)\chi_{\zeta}\right)}+o(1),
\end{equation*}
and for $T\rightarrow\infty$ while $t$ remain finite, we have
\begin{equation*}
\sqrt{\mathrm{det}_2\left(I-\chi_zG(\xi_1)KG^{-1}(\xi_2)\chi_z\right)}=
    \sqrt{\mathrm{det}_2\left(I-\chi_{\zeta}G(\xi_1)K_{airy}G^{-1}(\xi_2)\chi_{\zeta}\right)}+o(1),
\end{equation*}
where $K$ is the rescaled kernel in (\ref{eq:matker2}) and $\det_2$
is the regularized 2-determinant
$\det_2(I+A)=\det((I+A)e^{-A})e^{\tr(A_{11}+A_{22})}$ for the
$2\times 2$ matrix kernel $A$ with entries $A_{ij}$. The
characteristic functions $\chi_z$ and $\chi_{\zeta}$ are
$\chi_z=\chi_{[z,\infty)}$ and $\chi_{\zeta}=\chi_{[\zeta,\infty)}$
respectively.
\end{proposition}
The proof of this proposition is exactly the same as the proof of
Corollary 1.4 in \cite{DG2}. We shall not repeat the details of the
proof here. As in \cite{DG2}, the function $g(\xi)=\sqrt{1+\xi^2}$
is to ensure that the 2-determinant exists and there is a great
freedom in the choice of the $g(\xi)$.

We will now analyze the asymptotics of the derivative $\frac{\p\log
\det \mathbb{D}}{\p t}$ in (\ref{eq:derpar}).
\subsection{Asymptotics of the derivative $\frac{\p\log \det
\mathbb{D}}{\p t}$}\label{se:asymder}

We will now compute the asymptotics of the derivative $\frac{\p\log
\det \mathbb{D}}{\p t}$ in (\ref{eq:derpar}). Let us write the
integral in (\ref{eq:derpar}) in the following form
\begin{equation}\label{eq:derpar1}
\frac{\p}{\p
t}\log\det\mathbb{D}=-\int_{\mathbb{R}_+}\frac{K_2(x,x)}{t- x}\D
x-\int_{\mathbb{R}_+}\frac{K_1(x,x)}{t- x}\D x,
\end{equation}
where $K_2(x,y)$ is the kernel given by the Laguerre polynomials
(\ref{eq:k2}) and $K_1(x,y)$ is the correction term on the right
hand side of (\ref{eq:kerform1}).

To compute the contribution from the kernel $K_2$, we will use the
following differential identity \cite{BEH}. (See also \cite{MM}
Lemma 2.1) Let $\hat{Z}$ be the matrix related to the matrix $Z$ in
(\ref{eq:Phimat}) by $\hat{Z}=Zw_0^{-\frac{\sigma_3}{2}}$. Then we
have
\begin{lemma}\label{le:K2t}
Let $\hat{Z}=Zw_0^{-\frac{\sigma_3}{2}}$, where $Z$ is the matrix in
(\ref{eq:Phimat}). Then we have
\begin{equation}\label{eq:K2t}
\int_{\mathbb{R}}\frac{K_2(x,x)}{t- x}\D x=\frac{1}{2
}\tr\left(\hat{Z}^{-1}\left(t\right)\hat{Z}^{\prime}\left(t\right)\sigma_3\right)
\end{equation}
\end{lemma}
The lemma can be proven in exactly the same way as Lemma 2.1 of
\cite{MM}. The key is to use (\ref{eq:K2form}), the jump condition
of $\hat{Z}$ on $\mathbb{R}_+$ and L'Hopital Rule to write
$K_2(x,x)$ as
\begin{equation*}
K_2(x,x)=\frac{1}{4\pi
i}\left(\tr\left(\hat{Z}_-^{-1}\hat{Z}_-^{\prime}\sigma_3\right)-\tr\left(\hat{Z}_+^{-1}\hat{Z}_+^{\prime}\sigma_3\right)\right),\quad
x\in\mathbb{R}_+
\end{equation*}
and then deform the contour of integration to obtain (\ref{eq:K2t}).

The asymptotics of the matrix $\hat{Z}$ can be found in \cite{V}.
For $x\notin[0,4]$, the asymptotics of $\hat{Z}$ is given by
\begin{equation}\label{eq:Youtside}
\hat{Z}(x)=(i\kappa_{N-1})^{-\frac{\sigma_3}{2}}R(x)P_{\infty}(x)e^{N\left(\frac{x}{2}-\varphi(x)\right)\sigma_3},
\end{equation}
where $R(x)$ is of the form $I+O\left(N^{-1}\right)$ and
$P_{\infty}(x)$ is a matrix bounded in $x$ for $x\notin[0,4]$. Near
the point $4$, both the matrix $P_{\infty}$ and $P_{\infty}^{-1}$
have a forth root singularity. The derivative of $R(x)$ is of order
$O\left(N^{-1}\right)$ and the derivative of $P_{\infty}(x)$ remains
bounded. From this, we have
\begin{equation}\label{eq:K2out}
\begin{split}
\tr\left(\hat{Z}^{-1}\left(t\right)\hat{Z}^{\prime}\left(t\right)\sigma_3\right)
=N\left(1-\sqrt{\frac{t-4 }{t}}\right)+O\left(1\right).
\end{split}
\end{equation}
As we will see, this is the part that determines where the saddle
point is. Once we have done the saddle point analysis later on in
this section, we will see that at the phase transition, the saddle
point will be inside the Airy region. We therefore also need the
asymptotics of the matrix $\hat{Z}(x)$ inside the Airy region. This
again, can be found in \cite{V}.
\begin{lemma}\label{le:Vair}(\cite{V}, Section 5.3)
The asymptotics of the matrix $\hat{Z}(x)$ inside a small disc
$U_{\delta}$ of radius $\delta$ around $4$ is given by
\begin{equation*}
\hat{Z}(x)=\frac{\sqrt{2}e^{\frac{i\pi}{4}}(i\tilde{\kappa}_{N-1})^{-\frac{\sigma_3}{2}}}{x^{\frac{1}{4}}(x-4)^{\frac{1}{4}}}R(x)
\begin{pmatrix}\cos\eta_+&-i\sin\eta_+\\
-i\cos\eta_-&-\sin\eta_-\end{pmatrix}f_N(x)^{\frac{\sigma_3}{4}}A(f_N(x))w_0^{-\frac{\sigma_3}{2}},
\end{equation*}
where $\eta_{\pm}$ is given by (\ref{eq:eta}) and $A(\xi)$ is the
matrix
\begin{equation*}
\begin{split}
A(\xi)&=              \sqrt{2\pi}e^{-\frac{\pi i}{12}}\begin{pmatrix}Ai(\xi)&Ai(\omega^2\xi)\\
Ai^{\prime}(\xi)&\omega^2Ai^{\prime}(\omega^2\xi)\end{pmatrix}e^{-i\frac{\pi}{6}\sigma_3}, \quad\xi\in I, \\
A(\xi)&=              \sqrt{2\pi}e^{-\frac{\pi i}{12}}\begin{pmatrix}Ai(\xi)&Ai(\omega^2\xi)\\
Ai^{\prime}(\xi)&\omega^2Ai^{\prime}(\omega^2\xi)\end{pmatrix}e^{-i\frac{\pi}{6}\sigma_3}\begin{pmatrix}1&0\\
-1&1\end{pmatrix}, \quad \xi\in II, \\
A(\xi)&= \sqrt{2\pi}e^{-\frac{\pi i}{12}}\begin{pmatrix}Ai(\xi)&-\omega^2Ai(\omega \xi)\\
Ai^{\prime}(\xi)&-Ai^{\prime}(\omega
\xi)\end{pmatrix}e^{-i\frac{\pi}{6}\sigma_3}\begin{pmatrix}1&0\\
1&1\end{pmatrix},  \quad\xi\in
III,\\
A(\xi)&=\sqrt{2\pi}e^{-\frac{\pi i}{12}}\begin{pmatrix}Ai(\xi)&-\omega^2Ai(\omega \xi)\\
Ai^{\prime}(\xi)&-Ai^{\prime}(\omega
\xi)\end{pmatrix}e^{-i\frac{\pi}{6}\sigma_3}, \quad\xi\in IV.
\end{split}
\end{equation*}
where $\omega=e^{\frac{2\pi i}{3}}$ and the regions $I$, $II$, $III$
and $IV$ are given by
\begin{equation*}
I=\left\{\xi|\quad 0<\arg(\xi)<2\pi/3\right\},\quad
II=\left\{\xi|\quad 2\pi/3<\arg(\xi)<\pi\right\},\quad
III=\overline{II},\quad IV=\overline{I},
\end{equation*}
where the overline indicates complex conjugation. The matrix $R(x)$
is again of the form $I+O\left(N^{-1}\right)$. Its derivative is of
order $R^{\prime}(x)=O\left(N^{-1}\right)$.
\end{lemma}
From the behavior of the functions $\eta_{\pm}$, we see that the
asymptotics of $\hat{Z}(x)$ inside the Airy region is of the form
\begin{equation}\label{eq:Yrep}
\hat{Z}(x)=(i\kappa_{N-1})^{-\frac{\sigma_3}{2}}R(x)P_a(x)N^{\frac{\sigma_3}{6}}A(f_N(x))w_0^{-\frac{\sigma_3}{2}},
\end{equation}
where both $P_a(x)$ and $P_a^{-1}(x)$ are bounded and analytic
inside $U_{\delta}$ (See \cite{V}). Moreover, as
$f_N(x)\rightarrow\infty$, there is a matching condition between the
formula for $\hat{Z}(x)$ in the Airy region and its formula in the
outside region (\ref{eq:Youtside}).
\begin{equation}\label{eq:Ymatch}
P_a(x)N^{\frac{\sigma_3}{6}}A(f_N(x))w_0^{-\frac{\sigma_3}{2}}
=P_{\infty}(x)\left(I+O\left(f_N^{-\frac{3}{2}}\right)\right)e^{N\left(\frac{x}{2}-\varphi(x)\right)\sigma_3}.
\end{equation}
From this and the fact that $P_{\infty}$ is of order
$O\left((x-4)^{-\frac{1}{4}}\right)$ as $x\rightarrow 4$, we see
that (\ref{eq:K2out}) remains valid if $T=\left(t-4
\right)\left(N/4\right)^{\frac{2}{3}}$ is large, but with some
modifications in the error term.
\begin{equation}\label{eq:K2out1}
\begin{split}
\tr\left(\hat{Z}^{-1}\left(t\right)\hat{Z}^{\prime}\left(t\right)\sigma_3\right)
=N\left(1-\sqrt{\frac{t-4}{t}}\right)+O\left(N^{\frac{2}{3}}T^{-\frac{5}{2}}\right)+O\left(NT^{-\frac{3}{2}}\right).
\end{split}
\end{equation}
We shall divide the range of $T$ into the regimes $0<|T|\leq
N^{1/5}$ and $|T|\geq N^{1/5}$ and use (\ref{eq:K2out1}) to for
$|T|\geq N^{1/5}$. Let us now assume $|T|\leq N^{1/5}$. Then from
(\ref{eq:Yrep}) we obtain
\begin{equation}\label{eq:Yasym1}
\begin{split}
&\tr\left(\hat{Z}^{-1}\left(t\right)\hat{Z}^{\prime}\left(t\right)\sigma_3\right)=
\tr\left(A^{-1}\left(f_N\left(t\right)\right)A^{\prime}\left(f_N\left(t\right)\right)f_N^{\prime}\sigma_3\right)\\
&+\tr\left(P_a^{-1}P_a^{\prime}N^{\frac{\sigma_3}{6}}A\sigma_3A^{-1}N^{-\frac{\sigma_3}{6}}\right)
+M-\frac{M-N}{t}+O\left(N^{-\frac{2}{3}}\right),
\end{split}
\end{equation}
The error term is uniform in $T$ throughout $U_{\delta}$. By using
the asymptotic formula for the Airy functions, we see that the
matrix $A(\xi)$ has the following behavior as
$\xi\rightarrow\infty$.
\begin{equation}\label{eq:Amatasym}
A(\xi)=\frac{1}{\sqrt{2}}\xi^{-\frac{\sigma_3}{4}}\begin{pmatrix}1&1\\
-1&1\end{pmatrix}e^{-\frac{\pi
i}{4}\sigma_3}\left(I+O(\xi^{-\frac{3}{2}})\right)e^{-\frac{2}{3}\xi^{\frac{3}{2}}\sigma_3},
\quad   \xi\rightarrow\infty
\end{equation}
Therefore uniformly in $U_{\delta}$, the second term in
(\ref{eq:Yasym1}) is of order
\begin{equation*}
\tr\left(P_a^{-1}P_a^{\prime}N^{\frac{\sigma_3}{6}}A\sigma_3A^{-1}N^{-\frac{\sigma_3}{6}}\right)=
O\left((N/T)^{\frac{1}{3}}\right)
\end{equation*}
By (\ref{eq:Amatasym}), the first term in (\ref{eq:Yasym1}) has the
following behavior as $f_N\rightarrow\infty$.
\begin{equation*}
\tr\left(A^{-1}\left(f_N\left(t\right)\right)A^{\prime}\left(f_N\left(t\right)\right)\sigma_3\right)
=-2f_N^{\frac{1}{2}}+O\left(f_N^{-1}\right),\quad
f_N\rightarrow\infty.
\end{equation*}
Since
$f_N\left(t\right)=T\left(1+O\left(T/N^{\frac{2}{3}}\right)\right)$,
we have, by mean value theorem,
\begin{equation*}
\begin{split}
&\tr\left(A^{-1}\left(f_N\left(t\right)\right)A^{\prime}\left(f_N\left(t\right)\right)f_N^{\prime}\sigma_3\right)
=\tr\left(A^{-1}\left(T\right)A^{\prime}\left(T\right)f_N^{\prime}\sigma_3\right)+O\left(T^{\frac{3}{2}}/N^{\frac{2}{3}}\right)\\
&=\left(\frac{N}{4}\right)^{\frac{2}{3}}\tr\left(A^{-1}\left(T\right)A^{\prime}\left(T\right)\sigma_3\right)+O\left(T^{\frac{3}{2}}\right)
\end{split}
\end{equation*}
Therefore after integration, we have
\begin{equation*}
\int^t_{t_0}\frac{K_2(x,x)}{t- x}\D
x=\frac{1}{2}\int^{T}_{T_0}\tr\left(A^{-1}\left(T\right)
A^{\prime}\left(T\right)\sigma_3\right)\D
T+\frac{M(t-t_0)}{2}+O\left(N^{-\frac{2}{15}}\right),
\end{equation*}
where $|T|\leq N^{\frac{1}{5}}$ and $|T_0|=N^{\frac{1}{5}}$. For
$|T|\geq N^{\frac{1}{5}}$ and $|T_0|=N^{\frac{1}{5}}$, we have, by
(\ref{eq:K2out1}), the following
\begin{equation*}
\begin{split}
\int_{t_0}^{t}\tr\left(\hat{Z}^{-1}\left(t\right)\hat{Z}^{\prime}\left(t\right)\sigma_3\right)
=N\int_{t_0}^t\left(1-\sqrt{\frac{t-4}{t}}\right)\D
t+O\left(N^{\frac{7}{30}}\right).
\end{split}
\end{equation*}
Since the first term is of order
$O\left(T^{\frac{3}{2}}\right)+O\left(N^{\frac{1}{3}}T\right)$, we
see that for $|T|\geq N^{\frac{1}{5}}$, the first term will dominate
over the error term. Summarizing, we have the following.
\begin{proposition}\label{pro:K2cont}
Let $T=(t-4 )\left(N/4\right)^{\frac{2}{3}}$. Then the contribution
of the kernel $K_2$ to the logarithmic derivative of
$\det\mathbb{D}$ is given by
\begin{enumerate}
\item Uniformly for $|T|\leq N^{\frac{1}{5}}$, we have
\begin{equation}\label{eq:K2cont1}
\begin{split}
\int^t_{t_0}\frac{K_2(x,x)}{t- x}\D x&=\frac{1}{2
}\int^{T}_{T_0}\tr\left(A^{-1}\left(T\right)
A^{\prime}\left(T\right)\sigma_3\right)\D
T\\
&+\frac{M}{2 }(t-t_0)+O\left(N^{-\frac{2}{15}}\right),
\end{split}
\end{equation}
for any $|T_0|=N^{\frac{1}{5}}$, where the integration contour does
not cross the jump contours of the matrix $A$.
\item Uniformly for $|T|>N^{\frac{1}{5}}$, we have
\begin{equation}\label{eq:K2cont2}
\begin{split}
\int^t_{t_0}\frac{K_2(x,x)}{t- x}\D x =\frac{N}{2
}\int_{t_0}^t\left(1-\sqrt{\frac{t-4 }{t}}\right)\D
t+O\left(N^{\frac{7}{30}}\right).
\end{split}
\end{equation}
for any $|T_0|=N^{\frac{1}{5}}$.
\end{enumerate}
where the integration contour does not cross $(-\infty,4]$.
\end{proposition}
Let us now compute the contribution from the correction term
$K_1(x,y)$. Before we compute the asymptotics, let us make a further
simplification using (\ref{eq:k1form}). First by using the
recurrence relation of the Laguerre polynomials (\ref{eq:recur}), we
see that
\begin{equation*}
L_N=\frac{tL_N}{t- x}-\frac{ xL_N}{t- x} =\frac{L_N}{M(t-
x)}\left(Mt-(M+N+1) \right)-\frac{ L_{N+1}}{t- x} -\frac{N}{M}\frac{
L_{N-1}}{t- x}.
\end{equation*}
Substituting this back into (\ref{eq:k1form}), we obtain
\begin{equation}\label{eq:K1}
\begin{split}
&\int_{\mathbb{R}_+}\frac{K_1(x,y)}{t- x}\D x= -\frac{M
}{2h_{0,N-2}}\Bigg(\left<\frac{L_{N-2}}{t-
x},L_N\right>_1-2\frac{\left<L_N,L_{N-2}\right>_1}
{\left<L_{N-1},L_{N-2}\right>_1}\frac{\p}{\p t}\left<L_{N-1},L_{N-2}\right>_1\Bigg)\\
&-\frac{M }{h_{0,N-1}}\frac{\p}{\p
t}\left<L_{N+1},L_{N-1}\right>_1-\frac{1}{2h_{0,N-1}}\left<\frac{L_N}{t-
x},L_{N-1}\right>_1 +\frac{N }{2h_{0,N-1}}\left<\frac{L_{N-1}}{t-
x},L_{N-1}\right>_1.
\end{split}
\end{equation}
where we have used the fact that $\left<L_N,L_{N-1}\right>_1=0$.

The main task is to compute the products $\left<L_n/(t-
x),L_m\right>_1$. The analysis is very similar to those in Section
\ref{se:asymskew}. First note that, outside of the Airy region, the
analysis in Section \ref{se:asymskew} remains the same and we have
\begin{lemma}\label{le:besselbulk}
The single and double integrals in $\left<L_n/(t- x),L_m\right>_1$
have the following contributions in the Bessel region.
\begin{equation}\label{eq:derB}
\begin{split}
\left(i\kappa_{n-1}\right)^{\frac{1}{2}}\int_{0}^{N^{-\frac{1}{2}}}\frac{L_n(x)w(x)}{t-
x}\D
x&=\tilde{\mathcal{I}}_{1,n}=\frac{(-1)^n\sqrt{2\pi}}{\sqrt{n}t^{\frac{3}{2}}}+O\left(n^{-\frac{7}{8}}\right).
\end{split}
\end{equation}
The double integral in the Bessel region is given by
\begin{equation}\label{eq:derB2}
\begin{split}
&i\sqrt{\kappa_{n-1}\kappa_{m-1}}\int_0^{N^{-\frac{1}{2}}}\frac{L_n(x)w(x)}{t-
x}\int_x^{N^{-\frac{1}{2}}}L_m(y)w(y)\D x\D
y\\
&=\tilde{\mathcal{J}}_{1}=\frac{(-1)^{m+n}2\pi}{nt^2}+O\left(n^{-\frac{5}{4}}\right).
\end{split}
\end{equation}
The single integral in the bulk region is given by
\begin{equation}\label{eq:bulksingder}
\begin{split}
&\left(i\kappa_{n-1}\right)^{\frac{1}{2}}\int_{N^{-\frac{1}{2}}}^{c_-}\frac{L_n(x)w(x)}{t-
x}\D x=\tilde{\mathcal{I}}_{2,n}=O\left(n^{-\frac{7}{8}}\right)
\end{split}
\end{equation}
Let $n-m=k$, then the asymptotics of the double integrals in the
bulk region is given by the following.
\begin{equation}\label{eq:doubulkder}
\begin{split}
\tilde{\mathcal{J}}_2=O\left(N^{-\frac{5}{4}}\right), \quad
k=0,\quad \tilde{\mathcal{J}}_2=-\frac{\p\mathcal{J}_{2}}{\p t}
+O\left(N^{-\frac{5}{4}}\right), \quad k=1,2,
\end{split}
\end{equation}
where $\tilde{\mathcal{J}}_2$ is given by
\begin{equation*}
\tilde{\mathcal{J}}_2=i\sqrt{\kappa_{n-1}\kappa_{m-1}}\int_{N^{-\frac{1}{2}}}^{c_-}\frac{L_n(x)w(x)}{t-
x}\int_x^{c_-}L_m(y)w(y)\D x\D y.
\end{equation*}
and $\mathcal{J}_2$ is given in (\ref{eq:doubulk1}).
\end{lemma}
Note that the choice of $\varepsilon\leq 1/20$ ensures that the
error terms are of the same order as before.

In the Airy region, the analysis is more different. Let us now
compute these integrals inside the Airy region. As in Section
\ref{se:Airy}, we have
\begin{equation}\label{eq:Atilde}
\begin{split}
&\frac{N}{m}\left(i\kappa_{m-1}\right)^{\frac{1}{2}}\int_{\frac{n}{m}x}^{c_+}\hat{L}_m(y)\D
y=\int_{\nu(u)}^{v_+}\mathcal{A}\mathcal{E}_m\D
v+O\left(N^{-\frac{7}{6}}\right),\\
&\frac{N}{n}\left(i\kappa_{n-1}\right)^{\frac{1}{2}}\int_{c_-}^{c_+}\frac{\hat{L}_n(x)}{t-\frac{n}{N}
x}\D
x=\left(\frac{N}{4}\right)^{\frac{2}{3}}\int_{u_-}^{u_+}\mathcal{A}_1\tilde{\mathcal{E}}_n\D
u+O\left(N^{-\frac{1}{2}}\right),
\end{split}
\end{equation}
where $\mathcal{A}_1=\mathcal{A}/(T- u)$ and $\tilde{\mathcal{E}}_n$
is of the form
\begin{equation}\label{eq:Etilde}
\tilde{\mathcal{E}}_n=1-\frac{3}{2}v_{0,n}+\frac{15}{8}v_{0,n}^2+O\left(N^{-1}\right),
\end{equation}
where $v_{0,n}$ is given in (\ref{eq:v0}). By using mean value
theorem, (\ref{eq:vcoef}), we obtain the following estimate.
\begin{equation}\label{eq:Amint}
\begin{split}
&\int_{\nu\left(u\right)}^{v_+}\mathcal{A}\mathcal{E}_m\D v=
\int_{u}^{v_+}\mathcal{A}\mathcal{E}_m\D
v+\Bigg(\left(\frac{4}{N}\right)^{\frac{1}{3}}(m-n)\mathcal{A}(u)\mathcal{E}_m(u)\\
&-\left(\frac{4}{N}\right)^{\frac{2}{3}}\frac{(m-n)^2}{2}\mathcal{A}^{\prime}(u)\mathcal{E}_m(u)\Bigg)
\left(1+O\left(\frac{u}{N}\right)\right)+\frac{1}{3!}\mathcal{A}^{\prime\prime}(\xi_u)\left(u-\nu\right)^3\mathcal{E}_m(\xi_u)
\end{split}
\end{equation}
for some $\xi_u$ between $u$ and $\nu(u)$. As
\begin{equation*}
\begin{split}
&\left|\int_{u_-}^{u_+}\mathcal{A}_1(u)\mathcal{A}^{\prime\prime}(\xi_u)(\nu-u)^3\mathcal{E}_m(\xi_u)\D
u\right|\\
&<\frac{1}{N^{\frac{4}{3}}}\left|\int_{u_-}^{-|T|^{\frac{1}{2}}}\frac{C}{u^{\frac{1}{2}}(|T|-
u)}\D u\right|+O\left(\frac{1}{(1+|T|^2)N^{\frac{4}{3}}}\right),
\end{split}
\end{equation*}
we see that this term is of order
$O\left(N^{-\frac{4}{3}}T^{-\frac{1}{2}}\right)$ as
$T\rightarrow\infty$ and $|T|<|u_-|$. Applying similar argument to
the other error terms in (\ref{eq:Amint}), we see that
\begin{equation}\label{eq:airydouder}
\begin{split}
&\int_{u_-}^{u_+}\mathcal{A}_1\tilde{\mathcal{E}}_n\int_{\nu(u)}^{v_+}\mathcal{A}\mathcal{E}_m\D
v\D
u\\
&=\int_{u_-}^{u_+}\mathcal{A}_1\tilde{\mathcal{E}}_n\int_{u}^{v_+}\mathcal{A}\mathcal{E}_m\D
v\D u+\left(\frac{4}{N}\right)^{\frac{1}{3}}(m-n)
\int_{u_-}^{u_+}\mathcal{A}_1\mathcal{A}\tilde{\mathcal{E}}_n\mathcal{E}_m\D u\\
&-\left(\frac{4}{N}\right)^{\frac{2}{3}}\frac{(m-n)^2}{2}
\int_{u_-}^{u_+}\mathcal{A}_1\mathcal{A}^{\prime}\D
u+O\left(N^{-\frac{4}{3}}T^{-\frac{1}{2}}\right).
\end{split}
\end{equation}
Let us now determine the behavior of these terms as $|T|\rightarrow
N^{\frac{2}{3}}$. We have
\begin{lemma}\label{le:HijT}
As $|T|\rightarrow |u_-|$, the terms $\mathcal{A}_0(i,j)$ is of
order $O\left(N^{-\frac{1}{3}}T^{-i-j-1}\right)$ in $T$.
\end{lemma}
\begin{proof} Let us write the integral in the following form
\begin{equation*}
\mathcal{A}_0(i,j)=\int_{u_-}^{-\left|T^{\frac{1}{2}}\right|}\mathcal{A}_i\int_{u}^{-\left|T^{\frac{1}{2}}\right|}\mathcal{A}_j\D
v\D u+\int_{u_-}^{-\left|T^{\frac{1}{2}}\right|}\mathcal{A}_i\D
u\int_{-\left|T^{\frac{1}{2}}\right|}^{v_+}\mathcal{A}_j\D v
+\int_{-\left|T^{\frac{1}{2}}\right|}^{u_+}\mathcal{A}_i\int_{u}^{v_+}\mathcal{A}_j\D
v\D u
\end{equation*}
Then it is clear that the third term is of order
$O\left(N^{-\frac{1}{3}}T^{-i-j-1}\right)$ in $N$ and $T$ as
$|T|\rightarrow |u_-|$. By using the asymptotic formula
(\ref{eq:Aasym}) to integrate the second term by parts, we see that
this term is also of order
$O\left(N^{-\frac{1}{3}}T^{-i-j-1}\right)$. For the first term, we
have, from (\ref{eq:Aasym}), the following
\begin{equation*}
\left|\int_{u_-}^{-\left|T^{\frac{1}{2}}\right|}\mathcal{A}_i\int_{u}^{-\left|T^{\frac{1}{2}}\right|}\mathcal{A}_j\D
v\D
u\right|<\left|\int_{u_-}^{-\left|T^{\frac{1}{2}}\right|}\frac{C}{N^{\frac{1}{3}}u(|T|-
u)^{i+j+1}}\D u\right|
\end{equation*}
for some constant $C>0$ independent on $T$ and $N$. Integrating this
gives us the result.
\end{proof}
From Lemma \ref{le:HijT} and the asymptotic behavior of
$\mathcal{A}$, we see that the various terms in the above expansion
are of the following behavior as $|T|\rightarrow N^{\frac{2}{3}}$.
\begin{equation}\label{eq:Hord}
\begin{split}
&\int_{u_-}^{u_+}\mathcal{A}_1\mathcal{A}\D
u=O\left(N^{-\frac{1}{3}}T^{-\frac{3}{2}}\right),\quad
\int_{u_-}^{u_+}\mathcal{A}_1\mathcal{A}^{\prime}\D
u=O\left(N^{-\frac{1}{3}}T^{-2}\right),\\
&\int_{u_-}^{u_+}\mathcal{A}_1^2\D
u=O\left(N^{-\frac{1}{3}}T^{-\frac{5}{2}}\right),\quad
\mathcal{A}_0(i,j)=O\left(N^{-\frac{1}{3}}T^{-i-j-1}\right).
\end{split}
\end{equation}
From (\ref{eq:airydouder}), we obtain the double integral inside of
the Airy region.

Let us now compute the order of the following term in (\ref{eq:K1})
inside the Airy region.
\begin{equation}\label{eq:M1}
D_1=-\frac{M }{2h_{0,N-2}}\left<\frac{L_{N-2}}{t- x},L_N\right>_1
+\frac{N }{2h_{0,N-1}}\left<\frac{L_{N-1}}{t- x},L_{N-1}\right>_1.
\end{equation}

First by differentiating the identity
$\left<L_N,L_{N-1}\right>_1=0$, we can establish the following.
\begin{lemma}\label{le:dLNN}
Let $PH(i,j)$ be
\begin{equation}\label{eq:PH}
PH(i,j)=\frac{1}{2}\int_{u_-}^{u_+}\mathcal{A}_i\D
u\int_{u_-}^{u_+}\mathcal{A}_j\D u-\mathcal{A}_0(i,j),
\end{equation}
and let $\tilde{\mathcal{J}}_{3k}$ be the followings.
\begin{equation}\label{eq:tildeJ}
\begin{split}
\tilde{\mathcal{J}}_{31}&=\left(\frac{N}{4}\right)^{\frac{1}{3}}\int_{u_-}^{u_+}\mathcal{A}_1\mathcal{A}\D
u,\quad
\tilde{\mathcal{J}}_{32}=-\left(\frac{N}{4}\right)^{\frac{1}{3}}\frac{3 }{2}PH(2,0),\\
\tilde{\mathcal{J}}_{33}&=\frac{3}{8}PH(1,2)-\frac{1}{2}\left(
\int_{u_-}^{u_+}\mathcal{A}_1^2\D
u-\int_{u_-}^{u_+}\mathcal{A}_1\mathcal{A}^{\prime}\D
u\right),\\
\tilde{\mathcal{J}}_{34}&=\frac{3}{8}\left(5PH(3,0)-PH(1,2)\right),\quad
\tilde{\mathcal{J}}_{35}=-2 \int_{u_-}^{u_+}\mathcal{A}_1^2\D u,
\end{split}
\end{equation}
and let $\tilde{I}_j$ be
\begin{equation}\label{eq:tildeI}
\begin{split}
\tilde{I}_1=\left(\frac{N}{4}\right)^{\frac{1}{6}}\frac{\sqrt{2\pi}}{4\sqrt{t}}
\int_{u_-}^{u_+}\mathcal{A}_1\D u,\quad
\tilde{I}_2=\left(\frac{N}{4}\right)^{\frac{1}{3}}\frac{3\sqrt{2\pi}
}{4\sqrt{t}} \int_{u_-}^{u_+}\mathcal{A}_2\D u.
\end{split}
\end{equation}
Let $k=n-m$, then $\left<\frac{L_{n}}{t- x},L_m\right>_1$ is given
by
\begin{equation}\label{eq:DL}
\begin{split}
&\frac{2\pi}{h_{N,0}}\left<\frac{L_{n}}{t- x},L_m\right>_1
=-PH(1,0)+k\tilde{\mathcal{J}}_{31}+(N-n)\tilde{\mathcal{J}}_{32}+k^2\tilde{\mathcal{J}}_{33}
+(N-n)^2\tilde{\mathcal{J}}_{34}\\
&+k(N-n)\tilde{\mathcal{J}}_{35}+(-1)^n\Big((-1)^k\tilde{I}_1+(-1)^k(N-n)\tilde{I}_2\Big)
+\frac{\p}{\p
t}\mathcal{J}_2\\
&+O\left(N^{-\frac{9}{24}}/\left(1+|T|^{\frac{3}{2}}\right)\right)+O\left(N^{-\frac{2}{3}}/\left(1+|\sqrt{T}|\right)\right)
\end{split}
\end{equation}
\end{lemma}
The lemma follows from a straightforward calculation using
(\ref{eq:airydouder}) and Lemma \ref{le:besselbulk}.

As in the computation of the skew products
$\left<L_{2k},L_{2k-1}\right>_1=0$ can be used to simplify the
expressions $\left<\frac{L_n}{t- x},L_m\right>_1$. In this case, we
differentiate the identity $\left<L_{2k},L_{2k-1}\right>_1=0$ with
respect to $t$. Then we obtain
\begin{equation*}
-2\frac{\p}{\p t}\left<L_{n},L_{n-1}\right>_1=\left<\frac{L_{n}}{t-
x},L_{n-1}\right>_1- \left<\frac{L_{n-1}}{t- x},L_{n}\right>_1=0.
\end{equation*}
By using Lemma \ref{le:dLNN}, we obtain the following identities.
\begin{equation}\label{eq:idder}
\begin{split}
2\tilde{\mathcal{J}}_{31}-\tilde{\mathcal{J}}_{32}-2\tilde{I}_1+2\frac{\p
\mathcal{J}_{2}}{\p t}=E_1,\quad
-\tilde{\mathcal{J}}_{34}+\tilde{\mathcal{J}}_{35}-\tilde{I}_2=E_2,
\end{split}
\end{equation}
where $E_1$ and $E_2$ are error terms that behave as the error terms
in (\ref{eq:DL}) and $\mathcal{J}_2$ in the above is taken to be the
expression in (\ref{eq:doubulk1}) with $n-m=1$.

We can now compute the term $D_1$ in (\ref{eq:M1}). By using Lemma
\ref{le:dLNN}, we obtain the following
\begin{equation}\label{eq:M1asy}
\begin{split}
&\frac{4\pi D_1}{M
}=\tilde{\mathcal{J}}_{35}-4\tilde{\mathcal{J}}_{33} +\frac{2\pi}{N
}\frac{\p }{\p t}\sqrt{\frac{t-4 }{t}}-\frac{\p}{\p
t}\left(\mathcal{F}_{2}-2\mathcal{F}_1\right)+E.
\end{split}
\end{equation}
where the error term $E$ has the same order as the ones in
(\ref{eq:DL}).

Let us now show that $\mathcal{F}_{2}-2\mathcal{F}_1$ is of order
$O\left(N^{-1-\frac{\varepsilon}{2}}\right)$.

From (\ref{eq:fl}), we see that $\mathcal{F}_{2}-2\mathcal{F}_1$ is
given by
\begin{equation*}
\begin{split}
&\mathcal{F}_{2}-2\mathcal{F}_1=
\int_{\frac{N}{n}c_-}^{4}\frac{4\left(\sin\left(2\phi\right)-2\sin\phi\right)}{Nx(4-x)\left(t-
x\right)}\D
x=\int_{\frac{N}{n}c_-}^{4}\frac{8\left(\sin\phi\left(\cos\phi-1\right)\right)}{Nx(4-x)\left(t-
x\right)}\D
x\\
&=-\int_{\frac{N}{n}c_-}^{4}\frac{4\phi^3\left(1+O\left(\phi\right)\right)}{Nx(4-x)\left(t-
x\right)}\D x.
\end{split}
\end{equation*}
where again, $\phi=\arccos\left(\frac{x}{2}-1\right)$. As
$\phi=\sqrt{4-x}\left(1+O\left(x-4\right)\right)$, we have
\begin{equation*}
\begin{split}
&\mathcal{F}_{2}-2\mathcal{F}_1=
-\int_{\frac{N}{n}c_-}^{4}\frac{4\sqrt{4-x}\left(1+O\left(x-4\right)\right)}{Nx\left(t-
x\right)}\D x.
\end{split}
\end{equation*}
By the change of variable
$x=4+\left(\frac{4}{N}\right)^{\frac{2}{3}}\xi$ and $t=4
+\left(\frac{4}{N}\right)^{\frac{2}{3}}T$, it is easy to see that
the above integral is of order
$O\left(N^{-1-\frac{\epsilon}{2}}\right)$. Hence after integration,
we obtain
\begin{equation}\label{eq:M1int}
\begin{split}
&\int^tD_1\D t=\frac{1}{2}\sqrt{\frac{t-4
}{t}}-\frac{1}{4}\int_{-\infty}^{\infty}\frac{Ai^2(u)}{(T- u)^2}\D
u\\
&-\frac{3}{4}\int^T\left(\frac{Ai}{(T-
u)^{\frac{3}{2}}},\frac{Ai}{(T- u)^{\frac{5}{2}}}\right)_1\D
T+\int^TO\left(E\right)\D T
\end{split}
\end{equation}
where
$\left(f,g\right)_1=\int_{\mathbb{R}}\int_{\mathbb{R}}\epsilon(x-y)f(x)g(y)\D
x\D y$ and the error $E$ is of the order indicated in (\ref{eq:DL}).
In particular, the error term is of order $o(1)$ uniformly in
$T=o\left(N^{\frac{2}{3}}\right)$. For
$T=O\left(N^{\frac{2}{3}}\right)$, the error term is of order
$O(1)$.

Let us now consider the other terms in (\ref{eq:K1}). Let $D_2$ be
\begin{equation*}
D_2=\frac{M
}{h_{0,N-2}}\frac{\left<L_{N},L_{N-2}\right>_1}{\left<L_{N-1},L_{N-2}\right>_1}\frac{\p}{\p
t}\left<L_{N-1},L_{N-2}\right>_1-\frac{M }{h_{0,N-1}}\frac{\p}{\p
t}\left<L_{N+1},L_{N-1}\right>_1
\end{equation*}
By using (\ref{eq:prodasym}) and (\ref{eq:relprod}), we see that
$D_2$ is given by
\begin{equation*}
\begin{split}
&\frac{2\pi D_2}{M
}=\left(\frac{4I_1+\mathcal{J}_{B2}-\mathcal{J}_{B3,2}}{2I_0}+O\left(N^{-\frac{9}{24}}T^{\frac{1}{2}}\right)\right)
\frac{\p}{\p
t}\left(-2I_0+6I_1+O\left(N^{-\frac{25}{24}}\right)\right)\\
&+\frac{\p}{\p
t}\left(4I_1-\mathcal{J}_{B3,2}+O\left(N^{-\frac{25}{24}}\right)\right).
\end{split}
\end{equation*}
Note that the error term $O\left(N^{-\frac{25}{24}}\right)$ is
uniform in $T$ as $|T|\rightarrow N^{\frac{2}{3}}$ and its
derivative is of the same order and is also uniform in $T$. From
this, we obtain
\begin{equation*}
\begin{split}
&\frac{2\pi D_2}{M
}=\left(\frac{4I_1+\mathcal{J}_{B2}-\mathcal{J}_{B3,2}}{2I_0}+O\left(N^{-\frac{9}{24}}T^{\frac{1}{2}}\right)\right)
\frac{\p}{\p
t}\left(-2I_0+O\left(N^{-\frac{1}{3}}\right)\right)\\
&+\frac{\p}{\p
t}\left(4I_1-\mathcal{J}_{B3,2}+O\left(N^{-\frac{25}{24}}\right)\right).
\end{split}
\end{equation*}
By using the definitions of $I_0$, $I_1$ and $\mathcal{J}_{B2}$,
$\mathcal{J}_{B3,2}$ in Proposition \ref{pro:asyminner}, we obtain
the following for $D_2$.
\begin{equation}\label{eq:M2asy}
\begin{split}
\int^tD_2\D t&=\frac{2 }{\sqrt{t}}\int_{-\infty}^{\infty}H_1\D
u-\sqrt{\frac{t-4 }{t}}-\log
\left(t^{-\frac{1}{2}}\int_{-\infty}^{\infty}H_0\D
u\right)\\
&-2 \int^T\frac{\left(\int_{-\infty}^{\infty}H_1\D
u\right)^2}{\sqrt{t}\int_{-\infty}^{\infty}H_0\D u}\D T
+O\left(N^{-\frac{1}{24}}\log
T\right)+O\left(N^{-\frac{1}{3}}T^{\frac{1}{2}}\right).
\end{split}
\end{equation}
The orders of $T$ in the error terms indicates their behavior as
$|T|\rightarrow N^{\frac{2}{3}}$. Summarizing, we obtain the
contribution to the derivative from the correction kernel
$K_1(x,y)$.
\begin{proposition}\label{pro:corrder}
Let $|T|<N^{\frac{1}{3}}$. Then the contribution of $K_1(x,y)$ to
the logarithmic derivative is given by
\begin{equation}\label{eq:intcorr}
\begin{split}
&\int^t\int_{\mathbb{R}}\frac{K_1(x,x)}{t- x}\D x\D t=
-\frac{1}{4}\int_{-\infty}^{\infty}\frac{Ai^2}{(T- u)^2}\D
u+\int_{-\infty}^{\infty}H_1\D u-\log
\left(\int_{-\infty}^{\infty}H_0\D
u\right)\\&-\frac{3}{4}\int^T\left(\frac{Ai}{(T-
u)^{\frac{3}{2}}},\frac{Ai}{(T- u)^{\frac{5}{2}}}\right)_1\D
T-\int^T\frac{\left(\int_{-\infty}^{\infty}H_1\D
u\right)^2}{\int_{-\infty}^{\infty}H_0\D u}\D T+o(1).
\end{split}
\end{equation}
where the error term is uniform in $T$. As $T\rightarrow\infty$ but
$t$ remains finite, the function
$\int^t\int_{\mathbb{R}}\frac{K_1(x,x)}{t- x}\D x\D t$ will be of
order $O(1)$ in $N$.
\end{proposition}
The fact that the integral remains bounded for $t$ finite follows
from the estimates (\ref{eq:Hord}) and the expression
(\ref{eq:intcorr}).
\subsection{Steepest descent analysis}
Before we apply steepest descent analysis to the contour $\Gamma$,
let us take a closer look at the integration constants in
Proposition \ref{pro:K2cont} and \ref{pro:corrder}. Since the
integration constants must be the same for all $T$ while the
functions involved have jump discontinuities on $\mathbb{R}$, care
must be taken when choosing these integration constants.

First note that the expressions in Proposition \ref{pro:K2cont} and
\ref{pro:corrder} can be simplified further through integration by
parts. By repeat use of integration by parts, we can write the
following terms in (\ref{eq:M1int}) as
\begin{equation*}
\begin{split}
&-\frac{3}{4}\left(\frac{Ai}{(T- u)^{\frac{3}{2}}}, \frac{Ai}{(T-
u)^{\frac{5}{2}}}\right)_1+\frac{1}{2}\int_{-\infty}^{\infty}\frac{Ai^2}
{(T- u)^3}\D u=\\
&2\int_{-\infty}^0\left(H_1\int_{-\infty}^uH_2\D
v-\frac{(-u)^{\frac{1}{2}}}{2\pi(T- u)}\right)\D u +
2\int_0^{\infty}H_1\int_{-\infty}^uH_2\D v\D
u-\int_{-\infty}^{\infty}H_1\D u\int_{-\infty}^{\infty}H_2\D u\\
&-\int_{-\infty}^0\left(\frac{\left(Ai^{\prime}\right)^2-AiAi^{\prime\prime}}{T-
u} -\frac{(-u)^{\frac{1}{2}}}{\pi(T- u)}\right)\D
u-\int_0^{\infty}\frac{\left(Ai^{\prime}\right)^2-AiAi^{\prime\prime}}{T-
u}\D u.
\end{split}
\end{equation*}
The term $\frac{(-u)^{\frac{1}{2}}}{\pi(T- u)}$ is there to ensure
the convergence of the respective integrals. From the jump
discontinuities and the behavior at $T\rightarrow\infty$, one can
check that
\begin{equation*}
\begin{split}
&\int_{-\infty}^0\left(\frac{\left(Ai^{\prime}\right)^2-AiAi^{\prime\prime}}{T-
u} -\frac{(-u)^{\frac{1}{2}}}{\pi (T- u)}\right)\D u
+\int_0^{\infty}\frac{\left(Ai^{\prime}\right)^2-AiAi^{\prime\prime}}{T-
u}\D
u\\
&=\frac{1}{2}\tr\left(A^{-1}\left(T\right)A^{\prime}\left(T\right)\sigma_3\right)
+T^{\frac{1}{2}}
\end{split}
\end{equation*}
Therefore we have, for $|T|<N^{\frac{1}{5}}$,
\begin{equation*}
\begin{split}
&\int^t\int_{\mathbb{R}}\frac{S_1(x,x)}{t- x}\D x\D t=
2\int^T\int_{-\infty}^0\left(H_1\int_{-\infty}^uH_2\D
v-\frac{(-u)^{\frac{1}{2}}}{2\pi(T- u)}\right)\D u \D T\\&+
2\int^T\int_0^{\infty}H_1\int_{-\infty}^uH_2\D v\D u\D T
-\int^T\int_{-\infty}^{\infty}H_1\D u\int_{-\infty}^{\infty}H_2\D
u\D T\\&+\int_{-\infty}^{\infty}H_1\D u-\log
\left(\int_{-\infty}^{\infty}H_0\D
u\right)-\int^T\frac{\left(\int_{-\infty}^{\infty}H_1\D
u\right)^2}{\int_{-\infty}^{\infty}H_0\D u}\D
T+\frac{Mt}{2}-\frac{2}{3}T^{\frac{3}{2}}+C+o(1),
\end{split}
\end{equation*}
where $C$ is an integration constant that has no jump on
$\mathbb{R}$. By Lemma \ref{le:HijT}, we see that the term
$\left(\int_{-\infty}^{\infty}H_1\D
u\right)^2/\int_{-\infty}^{\infty}H_0\D u$ is of order
$T^{-\frac{5}{2}}$ as $|T|>N^{\frac{1}{5}}$. Hence the we can choose
the base point of this integral to be $+\infty$. For the other
terms, integration by parts shows that
\begin{equation}\label{eq:S11}
\begin{split}
&\mathcal{S}_{11}= 2\int_{-\infty}^0\left(H_1\int_{-\infty}^uH_2\D
v-\frac{(-u)^{\frac{1}{2}}}{2\pi(T- u)}\right)\D u+
2\int_0^{\infty}H_1\int_{-\infty}^uH_2\D v\D u
\\&-\int_{-\infty}^{\infty}H_1\D u\int_{-\infty}^{\infty}H_2\D u
=-\frac{1}{\pi}\int_{-\infty}^{0}\frac{2\pi(Ai^{\prime})^2-(-u)^{\frac{1}{2}}}{(T-
u)}\D
u\\
&-2\int_{0}^{\infty}\frac{(Ai^{\prime})^2}{(T- u)}\D
u+O\left(T^{-2}\right)
\end{split}
\end{equation}
as $T\rightarrow\infty$. Therefore if we let $0_{\pm}$ be the origin
on the positive and negative sides of $\mathbb{R}$, then
$\int_{0_+}^{0_-}\mathcal{S}_{11}\D T\in i\mathbb{R}$ is bounded,
where the integration contour goes from $0_+$ to $\infty$ along the
positive side of $\mathbb{R}_+$ and then from $\infty$ back to $0_-$
along the negative side of $\mathbb{R}_+$. Therefore we can fix the
integration constant by defining the integral of $\mathcal{S}_{11}$
to be $\int_{0_+}^T\mathcal{S}_{11}\D
T+1/2\int_{0_-}^{0_+}\mathcal{S}_{11}\D T$, which is the same as
$\int_{0_-}^T\mathcal{S}_{11}\D
T-1/2\int_{0_-}^{0_+}\mathcal{S}_{11}\D T$. We shall use these as
definitions of the function $\int^T\mathcal{S}_{11}\D T$ in the
upper and lower half planes respectively. Summarizing, we obtain
\begin{proposition}\label{pro:derasym}
The logarithm of $\det\mathbb{D}$ is given by the following
\begin{enumerate}
\item Uniformly for $|T|<N^{\frac{1}{5}}$, we have
\begin{equation}\label{eq:derasy1}
\begin{split}
&\det\mathbb{D}=C_0\exp\left(-\mathcal{S}_{10}-\int_{0_{\pm}}^T\mathcal{S}_{11}\D
T\mp1/2\int_{0_-}^{0_+}\mathcal{S}_{11}\D
T-\frac{M}{2}t\right)\\
&\times\int_{-\infty}^{\infty}H_0\D u\left(1+o(1)\right),\quad
\pm\mathrm{Im}(T)>0
\end{split}
\end{equation}
where $C_0$ is an integration constant and $\mathcal{S}_{11}$ is
given by (\ref{eq:S11}), while $\mathcal{S}_{10}$ is
\begin{equation*}
\begin{split}
\mathcal{S}_{10}=\int_{-\infty}^{\infty}H_1\D
u+\int_{T}^{\infty}\frac{\left(\int_{-\infty}^{\infty}H_1\D
u\right)^2}{\int_{-\infty}^{\infty}H_0\D u}\D
T-\frac{2}{3}T^{\frac{3}{2}}.
\end{split}
\end{equation*}
\item Uniformly for $|T|>N^{\frac{1}{5}}$, $\det\mathbb{D}$ is
given by
\begin{equation}\label{eq:derasy2}
\begin{split}
\det\mathbb{D}=C_0\exp\left(-\frac{N}{2}\int_{0}^t\left(1-\sqrt{\frac{t-4}{t}}\right)\D
t+O\left(N^{\frac{7}{30}}\right)\right).
\end{split}
\end{equation}
\end{enumerate}
\end{proposition}
\subsection{Saddle point analysis}
We can now carry out the steepest descent analysis for the contour
integral in $t$. As we have seen in the last section, the
regularized determinant $\mathrm{det}_2\left(I-\chi
GKG^{-1}\chi\right)$ remains bounded in $N$ for all values of $t$.
Then from (\ref{eq:derasy2}), we see that for $|T|>N^{\frac{1}{5}}$,
the saddle point is the solution of the following equation.
\begin{equation*}
\frac{2\tau}{(1+\tau)}-\left(1-\sqrt{\frac{t-4}{t}}\right)=0.
\end{equation*}
The saddle point is then given by
\begin{equation*}
t_{saddle}=\frac{(1+\tau)^2}{\tau}.
\end{equation*}
For $\tau\in(-1,0)$, we have $t_{saddle}\in(-\infty,0)$ and for
$0<\tau<1/4$, we have and $t_{saddle}\in[4,\infty)$. Hence the
saddle point $t_{saddle}$ will not intersect the bulk region $(0,4)$
for any value of $\tau\neq \pm1$. In this case, we can deform the
contour $\Gamma$ such that it does not intersect the interval
$[0,4]$ and by Proposition \ref{pro:largest}, we see that the
kernels $K_1$ and $K_2$ become the respective Airy kernels for all
$t\in\Gamma$. The largest eigenvalue distribution in this case
becomes
\begin{equation*}
\mathbb{P}\left(\left(\lambda_{max}-4\right)\left(N/4\right)^{\frac{2}{3}}\leq
\zeta\right)=C_0
\sqrt{\mathrm{det}_2\left(I-\chi_{[\zeta,\infty)}GK_{airy}G^{-1}\chi_{[\zeta,\infty)}\right)}\int_{\Gamma}\left(f_1(t)+f_2(t,z)\right)\D
t
\end{equation*}
for some function $f_1(t)$ independent on $z$ and $f_2(t,z)$ of
order $o(1)$. This gives us the Tracy Widom distribution for the
largest eigenvalue. This is a known result in \cite{FS}.

However, when $\tau=1$, saddle point coincides with the right edge
point. In this case, the main contribution of the integral in $t$
will come from a neighborhood of $t=4$ and the kernels will become
significantly different from the Airy kernel. This is the case when
the phase transition happens. Let us now find the steepest descent
contour in this case.
\begin{lemma}\label{le:steep}
There exists $\delta>0$ such that
\begin{equation}\label{eq:steep}
\mathrm{Re}\left(\int_4^{x}\sqrt{\frac{s-4}{s}}\D s\right)<0,\quad
\left|\mathrm{Im}(x)\right|<\delta,\quad -\delta<\mathrm{Re}(x)<4
\end{equation}
\end{lemma}
\begin{proof} The statement for $0<\mathrm{Re}(x)<4$ is a standard property that
follows from the Cauchy-Riemann equation. (See, e.g. \cite{V}).
Since
\begin{equation*}
\frac{\p}{\p
x}\mathrm{Im}\left(\int_4^{x}\left(\sqrt{\frac{s-4}{s}}\right)_{\pm}\D
s\right)=\pm\sqrt{\frac{4-x}{x}},
\end{equation*}
By the Cauchy-Riemann equations, we see that
$\mathrm{Re}\left(\int_4^{x}\left(\sqrt{\frac{s-4}{s}}\right)_{\pm}\D
s\right)$ is decreasing as we move away from the real axis. As this
real part is zero on $[0,4]$, we obtain the inequality
(\ref{eq:steep}) on for $0<\mathrm{Re}(x)<4$. As the function
$\int_4^{x}\sqrt{\frac{s-4}{s}}\D s$ behaves as
$c_0-2\sqrt{2}(-x)^{\frac{1}{2}}+O\left(x\right)$ when $x\rightarrow
0$ for some $c_0\in i\mathbb{R}$, we see that in a sufficiently
small disk $D$ center at $x=0$, we have
\begin{equation*}
\mathrm{Re}\left(\int_4^{x}\left(\sqrt{\frac{s-4}{s}}\right)_{\pm}\D
s\right)<0,\quad x\in D\setminus \mathbb{R}_+.
\end{equation*}
This completes the proof of the lemma.
\end{proof}
\begin{figure}
\centering \psfrag{4}[][][1.25][0.0]{\small$4$}
\psfrag{0}[][][1.25][0.0]{\small$0$}
\psfrag{xi}[][][1.25][0.0]{\small$\Xi$}
\psfrag{gamma}[][][1.25][0.0]{\small$\Gamma$}
\includegraphics[scale=0.5]{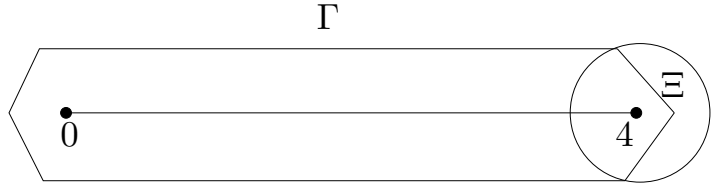}
\caption{The steepest descent contour $\Gamma$ is inside a
neighborhood of $[0,4]$ the conditions of Lemma \ref{le:steep}
holds. Inside a neighborhood of $4$, we choose the contour to be
symmetric and denote it by $\Xi$.}\label{fig:intcont}
\end{figure}
We can therefore choose our integration contour as in Figure
\ref{fig:intcont} to obtain
\begin{proposition}\label{pro:asymdis}
Let
$w=\left(\frac{N}{4}\right)^{\frac{1}{3}}(1-\tau)\in(-\infty,\infty)$
and let $\zeta=(z-4)\left(N/4\right)^{\frac{2}{3}}$, then as
$N\rightarrow\infty$, the largest eigenvalue distribution is given
by
\begin{equation*}
\begin{split}
\lim_{N\rightarrow\infty}\mathbb{P}\left(\left(\lambda_{max}-4\right)\left(\frac{N}{4}\right)^{\frac{2}{3}}\leq\zeta\right)
&=C\int_{\Xi}e^{-\frac{wT}{2}-\frac{1}{2}\mathcal{S}}
\left(\int_{-\infty}^{\infty}H_0\D
u\right)^{\frac{1}{2}}\\
&\times\left(\mathrm{det}_2\left(I-\chi_{\zeta}G(\xi_1)K_{\infty}G^{-1}(\xi_2)\chi_{\zeta}\right)\right)^{\frac{1}{2}}\D
T,
\end{split}
\end{equation*}
for some constant $C$, where $\Xi$ is a symmetric contour that does
not contain any zero of $\int_{-\infty}^{\infty}H_0\D u$ and
approaches $\infty$ in the sector $\pi/3<\arg T<4\pi/3$, $\arg T\neq
\pi$. It intersects the $\mathbb{R}$ at a point $T_0>\zeta$. The
function $\mathcal{S}$ is given by
\begin{equation*}
\mathcal{S}=\mathcal{S}_{10}+\int_{0_{\pm}}^T\mathcal{S}_{11}\D
T\pm1/2\int_{0_-}^{0_+}\mathcal{S}_{11}\D
T,\quad\pm\mathrm{Im}(T)>0.
\end{equation*}
\end{proposition}
\begin{remark} In the above formula, the integration should be
understood as a sum of integrations performed over the contours
$\Xi_+$ and $\Xi_-$, where $\Xi_{\pm}$ are the intersections of
$\Xi$ with the upper/lower half plane. Near the intersection point
$T_0$, the boundary values of the integrand in the upper/lower half
planes are to be taken when performing these integrals.
\end{remark}

\section{Fredholm determinant}\label{se:Fred}
In this section we will carry out the final part of the analysis and
express the determinant
$\mathrm{det}_2\left(I-\chi_{[z,\infty)}K\chi_{[z,\infty)}\right)$
in (\ref{eq:Pmaxdet}) in terms of the Hastings-McLeod solution of
the Painlev\'e II equation. We will follow the approach in
\cite{TW1}. While the following operations are formal and did not
take into account the fact that $I-G\chi K_{\infty}G^{-1}\chi$ is
not of trace class, the procedure and the resulting formula
(\ref{eq:det}) can be justified rigorously as in \cite{TW1}.

Let $S_{1,\infty}$ be the operator with the kernel
$S_{1,\infty}=K_{1,\infty}+K_{2,\infty}$, $D$ the differential
operator, $\epsilon$ the operator with kernel
$\epsilon(\xi_1-\xi_2)$ and $\chi$ the multiplication by
$\chi_{[\zeta,\infty)}$, where
$\zeta=(z-4)\left(N/4\right)^{\frac{2}{3}}$. Then by Proposition
\ref{pro:largest}, we would like to consider the determinant of the
following operator.
\begin{equation}\label{eq:Fredholm}
I-G\chi\begin{pmatrix}S_{1,\infty}&S_{1,\infty}D\\
\tilde{\epsilon}(S_{1,\infty})-\epsilon&S_{1,\infty}^T\end{pmatrix}\chi
G^{-1},
\end{equation}
where $S_{1,\infty}^T$ is the transpose of $S_{1,\infty}$ and
$\tilde{\epsilon}f=\int_{\xi_2}^{\xi}f(t)\D t$. Then by
(\ref{eq:asymk1}) and (\ref{eq:kinfty}), we see that
\begin{equation*}
\begin{split}
&-\frac{\p}{\p\xi_2}K_{1,\infty}(\xi_1,\xi_2)-\frac{\p}{\p\xi_1}K_{1,\infty}(\xi_2,\xi_1)
=\frac{\p}{\p\xi_2}K_{2,\infty}(\xi_1,\xi_2)+\frac{\p}{\p\xi_1}K_{2,\infty}(\xi_2,\xi_1)\\
&=-TH_0(\xi_1)H_0(\xi_2)- H_1(\xi_1)H_1(\xi_2)+H_0(\xi_1)H_2(\xi_2)+
H_2(\xi_1)H_0(\xi_1).
\end{split}
\end{equation*}
This implies $S_{1,\infty}D=DS_{1,\infty}^T$ and as
$D\tilde{\epsilon}=I$, we can follow \cite{TW1} to write the
operator as
\begin{equation*}
I-AB,\quad A=G\begin{pmatrix}\chi D&0\\
0&\chi\end{pmatrix}\begin{pmatrix}1&0\\
1&1\end{pmatrix}g,\quad B=g^{-1}\begin{pmatrix}1&0\\
-1&1\end{pmatrix}\begin{pmatrix}\tilde{\epsilon}(S_{1,\infty})&S_{1,\infty}^T\\
\tilde{\epsilon}(S_{1,\infty})-\epsilon&S_{1,\infty}^T\end{pmatrix}\chi
G^{-1}
\end{equation*}
then by using $\det\left(I-AB\right)=\det\left(I-BA\right)$, we
obtain
\begin{equation*}
I-g^{-1}\begin{pmatrix}1&0\\
-1&1\end{pmatrix}\begin{pmatrix}\tilde{\epsilon}(S_{1,\infty})\chi D&S_{1,\infty}^T\chi\\
\tilde{\epsilon}(S_{1,\infty})-\epsilon\chi D&S_{1,\infty}^T\chi\end{pmatrix}\begin{pmatrix}1&0\\
1&1\end{pmatrix}g,
\end{equation*}
from this, we see that the determinant can be written as the
determinant of the scalar operator
\begin{equation*}
\mathrm{det}_2\left(I-\chi
GK_{\infty}G^{-1}\chi\right)=\mathrm{det}_2\left(I-g^{-1}\left(\tilde{\epsilon}(S_{1,\infty})\chi
D+S_{1,\infty}^T\chi -S_{1,\infty}^T\chi\epsilon\chi
D\right)g\right).
\end{equation*}
Let us now show that $\tilde{\epsilon}
S_{1,\infty}=S_{1,\infty}^T\epsilon$. First by writing
$\tilde{\epsilon}S_{1,\infty}$ as
\begin{equation*}
\tilde{\epsilon}S_{1,\infty}=\epsilon\left(S_{1,\infty}\right)(\xi_1,\xi_2)-
\epsilon\left(S_{1,\infty}\right)(\xi_2,\xi_2),
\end{equation*}
we see that the first term is equal to $S_{1,\infty}\epsilon$
because $S_{1,\infty}D=DS_{1,\infty}^T$, and hence $\epsilon
S_{1,\infty}=S_{1,\infty}^T\epsilon$. Let us compute the second
term. Let $f$ be an $L^2$ function, then
\begin{equation}\label{eq:s}
\begin{split}
\epsilon\left(S_{1,\infty}\right)(\xi_2,\xi_2)f&=
-\int_{-\infty}^{\infty}\int_{\xi_2}^{\infty}S_{1,\infty}(t,\xi_2)f(\xi_2)\D
t\D\xi_2\\
&+\frac{1}{2}\int_{-\infty}^{\infty}\int_{-\infty}^{\infty}S_{1,\infty}(t,\xi_2)f(\xi_2)\D
t\D\xi_2.
\end{split}
\end{equation}
These terms can be computed using integration by parts. The first
term becomes
\begin{equation*}
\begin{split}
\int_{-\infty}^{\infty}\int_{\xi_2}^{\infty}S_{1,\infty}(t,\xi_2)f(\xi_2)\D
t\D\xi_2&=
\int_{-\infty}^{\infty}S_{1,\infty}(\xi_2,\xi_2)\int_{-\infty}^{\xi_2}f(u)\D
u\D
t\D\xi_2\\
&-\int_{-\infty}^{\infty}\int_{\xi_2}^{\infty}\frac{\p}{\p\xi_2}S_{1,\infty}(t,\xi_2)\int_{-\infty}^{\xi_2}f(u)\D
u\D t\D\xi_2.
\end{split}
\end{equation*}
As $-\frac{\p}{\p\xi_2}S_{1,\infty}(t,\xi_2)=\frac{\p}{\p
t}S_{1,\infty}(\xi_2,t)$, we obtain
\begin{equation*}
\begin{split}
\int_{-\infty}^{\infty}\int_{\xi_2}^{\infty}S_{1,\infty}(t,\xi_2)\int_{-\infty}^{\xi_2}f(u)\D
u\D t\D\xi_2&=
\int_{-\infty}^{\infty}S_{1,\infty}(\xi_2,\xi_2)\int_{-\infty}^{\xi_2}f(u)\D
u\D
t\D\xi_2\\
&
-\int_{-\infty}^{\infty}S_{1,\infty}(\xi_2,\xi_2)\int_{-\infty}^{\xi_2}f(u)\D
u\D t\D\xi_2=0.
\end{split}
\end{equation*}
The second term in (\ref{eq:s}) can be computed similarly and we
obtain $\tilde{\epsilon} S_{1,\infty}=S_{1,\infty}^T\epsilon$.
Therefore we have
\begin{equation}\label{eq:det}
\mathrm{det}_2\left(I-\chi
GK_{\infty}G^{-1}\chi\right)=\mathrm{det}_2\left(I-g^{-1}\left(S_1^T\epsilon\chi
D+S_1^T\chi -S_1^T\chi\epsilon\chi D\right)g\right).
\end{equation}
As mentioned before, the procedure in obtaining (\ref{eq:det}) can
be rigorously justified as in \cite{TW1}. We shall therefore treat
(\ref{eq:det}) as a rigorous formula and refer the readers to
\cite{TW1}.
\subsection{Differential equations for the Fredholm determinant}
Before we compute the asymptotics of this determinant, let us first
recall some facts about Fredholm determinants and their relations to
Painlev\'e equations.

Let us now recall some basic facts about operators of the form
(\ref{eq:airyker}). The Airy kernel (\ref{eq:airyker}) is an
integrable kernel, that is, it has the form
\begin{equation}\label{eq:intop}
K(x,y)=\frac{\sum_{j=1}^kf_{1j}(x)f_{2j}(y)}{x-y}.
\end{equation}
In this case, we have $k=2$ and $f_{11}=-f_{22}=Ai$, while
$f_{21}=f_{12}=Ai^{\prime}$. Operators of these form appear often in
random matrix theory and have been studied extensively in the
literature. These operators were first singled out as a
distinguished class in \cite{IIKS} in which their properties were
also studied. In random matrix theory, they were used to obtain the
celebrated Tracy-Widom distribution.

Let us remind ourselves the following well-known facts about
integrable operators. (See e.g. \cite{D}, \cite{IIKS}, \cite{TW3})
Let the kernel of an integrable operator on a contour $\Sigma$ be
given by (\ref{eq:intop}). Suppose $I-K$ is invertible, then the
resolvent $R$ given by
\begin{equation}\label{eq:resolvent}
R=(I-K)^{-1}K=(I-K)^{-1}-I,
\end{equation}
is also an integrable operator with kernel given by
\begin{equation}\label{eq:resokern}
R(x,y)=\frac{\sum_{j=1}^kF_{1j}(x)F_{2j}(y)}{x-y},
\end{equation}
where $F_{ij}$ are given by $F_{ij}=(I-K)^{-1}f_{ij}$. Let us now
specialize to the resolvent of the operator with the kernel
$K_{2,airy}\chi$ acting on $\mathbb{R}$. Then the resolvent of this
kernel is given by $(I-K_{2,airy}\chi)^{-1}=I+R_0\chi$, where $R_0$
has the kernel
\begin{equation}\label{eq:rels}
\begin{split}
R_0(\xi_1,\xi_2)&=\frac{\Phi_0(\xi_1)\Phi_1(\xi_2)-\Phi_0(\xi_2)\Phi_1(\xi_1)}{\xi_1-\xi_2}.
\end{split}
\end{equation}
where $\Phi_0$ and $\Phi_1$ are
\begin{equation}\label{eq:Phi}
\Phi_0=\left(I-K_{2,airy}\chi\right)^{-1}Ai,\quad
\Phi_1=\left(I-K_{2,airy}\chi\right)^{-1}Ai^{\prime}.
\end{equation}
For our analysis, we will also need $\Phi_2$ defined by
$\Phi_2=\left(I-K_{2,airy}\chi\right)^{-1}Ai^{\prime\prime}$. The
resolvent $R_0\chi$ is closely related to the Hastings-McLeod
solution of the Painlev\'e II equation. In fact, the determinant
$\det\left(I+R_0\chi\right)$ is well-known in the random matrix
literature and is given by the Tracy-Widom distribution for the GUE
\cite{TWairy}.
\begin{equation}\label{eq:TW2}
\det(I+R_0\chi)=TW_2(\zeta).
\end{equation}
The operator $S_{1,\infty}$ has kernel
$S_{1,\infty}=K_{1,\infty}+\psi^{-1}K_{2,airy}\psi$, where $\psi$ is
the multiplication of the function $\psi=(T- \xi)^{\frac{1}{2}}$. We
will also use the same notation to denote this square root function
itself.

By substituting the asymptotic kernels into (\ref{eq:det}), we
obtain
\begin{equation}\label{eq:detres}
\mathrm{det}_2\left(I-G\chi  K_{\infty}\chi
G^{-1}\right)=\det(I+g^{-1}R^Tg\chi)\det\left(I-g^{-1}\tilde{K}g\right),
\end{equation}
where $\tilde{K}$ and $R$ are the operators with the following
kernels
\begin{equation}\label{eq:prektilde}
\begin{split}
&R(\xi_1,\xi_2)=\psi^{-1}(\xi_1)R_0(\xi_1,\xi_2)\psi(\xi_2),\\
&\tilde{K}=(I+R^T\chi)\left(K_{1,\infty}^T+K_{2,airy}^T\right)(1-\chi)\epsilon\chi
D+(I+R^T\chi)K_{1,\infty}^T\chi.
\end{split}
\end{equation}
As the resolvent operator $R$ is conjugate to the resolvent $R_0$,
we have $\det(I+g^{-1}R^Tg\chi)=\det(I+R_0\chi)$, which is the
Tracy-Widom distribution for the GUE by (\ref{eq:TW2}). Let us now
consider the determinant of $I-g^{-1}\tilde{K}g$. We will follow the
steps in \cite{TW1} to show that the operator $g^{-1}\tilde{K}g$ is
of finite rank, that is, it is of the form
\begin{equation}\label{eq:finrank}
\tilde{K}=\sum_{j=1}^k\alpha_j\otimes\beta_j,
\end{equation}
where $f\otimes h$ is the operator with kernel $f(x)h(y)$. As in
\cite{TW1}, the operator $(1-\chi)\epsilon\chi D$ is given by (See
(16) in \cite{TW1})
\begin{equation}\label{eq:echi}
(1-\chi)\epsilon\chi
D=\left(1-\chi\right)\left(-\epsilon_{\zeta}\otimes\delta_{\zeta}+\epsilon_{\infty}\otimes\delta_{\infty}\right),
\end{equation}
where $\epsilon_{\zeta}$, $\epsilon_{\infty}$, $\delta_{\zeta}$ and
$\delta_{\infty}$ are given by
\begin{equation*}
\epsilon_{\zeta}(\xi)=\epsilon(\xi-\zeta),\quad
\epsilon_{\infty}(\xi)=-\frac{1}2,\quad
\delta_{\zeta}=\delta(\xi-\zeta),\quad
\delta_{\infty}=\delta(\xi-\infty).
\end{equation*}
From these definitions, it is easy to see that
\begin{equation}\label{eq:echi2}
\epsilon_{\zeta}(1-\chi)=\epsilon_{\infty}(1-\chi)=-\frac{1}{2}(1-\chi).
\end{equation}
Now from (\ref{eq:asymk1}), we see that the kernel $K_{1,\infty}$ is
of rank 2.
\begin{equation}\label{eq:K1rank}
\begin{split}
&K_{1,\infty}=H_0\otimes G_0+H_1\otimes G_1,\\
G_0&=\left(\frac{T}{2}I(H_0)-\mathcal{B}_1I(H_1)-
I(H_2)+\mathcal{B}_2\right),\quad
G_1=\left(\frac{1}{2}I(H_1)+\mathcal{B}_1I(H_0)\right)
\end{split}
\end{equation}
where $I$ is the integration $I(f)=\int_{-\infty}^{\xi}f(x)\D x$. By
using this, (\ref{eq:echi}) and (\ref{eq:echi2}) in
(\ref{eq:prektilde}), we see that $\tilde{K}$ can be written as
\begin{equation*}
\begin{split}
\tilde{K}&=\frac{1}{2}\left(I+R^T\chi\right) \Big(\psi
K_{2,airy}\psi^{-1}(1-\chi)+\left(H_0,(1-\chi)\right)G_0\\
&+\left(H_1,(1-\chi)\right)G_1\Big)
\otimes\left(\delta_{\zeta}-\delta_{\infty}\right)+\left(I+R^T\chi\right)\sum_{j=0}^1G_j\otimes
H_j\chi,
\end{split}
\end{equation*}
where $(f,h)=\int_{\mathbb{R}}f(x)h(x)\D x$ and we have used
$(a\otimes b)(c\otimes d)=(b,c)a\otimes d$ to obtain the above
equation. From this, we see that $\tilde{K}$ is indeed of the form
(\ref{eq:finrank}) with $k=3$ and $\alpha_j$, $\beta_j$ given by
\begin{equation}\label{eq:alpha}
\begin{split}
\alpha_1&=\frac{1}{2}\left(I+R^T\chi\right) \left(\psi
K_{2,airy}\psi^{-1}(1-\chi)+\left(H_0,(1-\chi)\right)G_0+\left(H_1,(1-\chi)\right)G_1\right),\\
\beta_1&=\delta_{\zeta}-\delta_{\infty},\quad
\alpha_j=(I+R^T\chi)G_{j-2},\quad \beta_j=H_{j-2}\chi,\quad j=2,3,
\end{split}
\end{equation}
Then the determinant $\det\left(I-g^{-1}\tilde{K}g\right)$ is given
by the determinant of the $3\times 3$ matrix with entries
$\delta_{ij}-(\alpha_i,\beta_j)$. We will now derive a close system
of ODEs in the variable $\zeta$ for these matrix entries. From the
expression of $G_j$ in (\ref{eq:K1rank}), we see that the matrix
entries involve the following quantities
\begin{equation}\label{eq:uij}
\begin{split}
Q_{-,j}&=\psi\left(I+R_0\chi\right)\psi^{-1}I(H_j),\quad
u_{-,jk}=\left(Q_{-,j},H_k\chi\right),\\
V_k&=\left(\psi\left(I+R_0\chi\right)K_{2,airy}\psi^{-1}(1-\chi),H_k\chi\right),\quad
k=0,1,\quad j=0,1,2.
\end{split}
\end{equation}
We will also introduce some auxiliary functions
\begin{equation}\label{eq:aux}
\begin{split}
Q_{+,j}&=\psi^{-1}\left(I+R_0\chi\right)\psi I(H_j),\quad u_{+,jk}=
\left(\psi^2Q_{+,j},H_k\chi\right),\\
\mathcal{R}_{\pm}&=\psi^{\mp1}\int_{-\infty}^{\zeta}R_0(\zeta,\xi)\psi^{\pm
1}(\xi)\D \xi,\quad
\tilde{\mathcal{R}}_{\pm}=-\psi^{\mp1}\int_{\zeta}^{\infty}R_0(\zeta,\xi)\psi^{\pm
1}(\xi)\D \xi\\
\mathcal{P}_{+,0}&=\int_{-\infty}^{\zeta}\Phi_0\psi\D \xi,\quad
\mathcal{P}_{-,j}=\int_{-\infty}^{\zeta}\Phi_j\psi^{-1}\D \xi,\\
\tilde{\mathcal{P}}_{+,0}&=-\int_{\zeta}^{\infty}\Phi_0\psi\D
\xi,\quad
\tilde{\mathcal{P}}_{-,j}=-\int_{\zeta}^{\infty}\Phi_j\psi^{-1}\D
\xi,\quad k=0,1,\quad j=0,1,2.
\end{split}
\end{equation}
which will be more convenient for the purpose of deriving the ODEs.
Note that, as in \cite{TW1} and \cite{TW3}, the $u_{\pm,jk}$ can be
written as
\begin{equation}\label{eq:upmjk}
\begin{split}
u_{\pm,jk}&=\left(\left(I+R_0\chi\right)\psi^{\pm 1}I(H_j),\psi
H_k\chi\right)=\left(\psi^{\pm 1}I(H_j),\left(I+R_0\chi\right)\psi
H_k\chi\right)\\
&=\left(\Phi_k,\psi^{\pm 1}I(H_j)\chi\right),
\end{split}
\end{equation}
where we have used (\ref{eq:Phi}) in the above. By using the
definition of the resolvent (\ref{eq:resolvent}) to write $R_0\chi$
as $R_0\chi=K_{2,airy}\chi\left(I+K_{2,airy}\chi\right)^{-1}$, we
see that
\begin{equation}\label{eq:xi1asym}
\begin{split}
R_0(\xi_1,\xi_2)=O\left(e^{-\frac{2}{3}\xi_1^{\frac{3}{2}}}\right),\quad
\xi_1\rightarrow\infty,\quad
Q_{\pm,j}(\infty)=\int_{-\infty}^{\infty}H_j\D \xi.
\end{split}
\end{equation}
We have the following relations between these variables.
\begin{lemma}\label{le:vV} The functions $V_k$ and
$Q_{+,k}$ in (\ref{eq:uij}) can be written as
\begin{equation}\label{eq:vV}
\begin{split}
V_k&=\mathcal{P}_{-,k}-\left(H_k,(1-\chi)\right),\quad
Q_{+,j}-Q_{-,j}=\frac{\Phi_0}{\psi}u_{-,j1}-\frac{\Phi_1}{\psi}u_{-,j0},\\
\mathcal{R}_+&-\mathcal{R}_-=\frac{\Phi_0(\zeta)}{\psi(\zeta)}\mathcal{P}_{-,1}
-\frac{\Phi_1(\zeta)}{\psi(\zeta)}\mathcal{P}_{-,0}.
\end{split}
\end{equation}
\end{lemma}
\begin{proof}
By using the definition of the resolvent (\ref{eq:resolvent}), we
see that $(I+R_0\chi)K_{2,airy}=R_0$ and hence $V_k$ can be written
as
\begin{equation*}
V_k=\left(R_0\psi^{-1}(1-\chi),\psi
H_k\chi\right)=\left(\psi^{-1}(1-\chi),R_0\psi H_k\chi\right)=
\left((1-\chi),\Phi_k\psi^{-1}\right)-\left(H_k,(1-\chi)\right),
\end{equation*}
where we have used the property of integrable operators
(\ref{eq:Phi}). The first equation then follows immediately from
this.

By using $\psi=(T- \xi)^{\frac{1}{2}}$, one can easily verify the
following.
\begin{equation*}
\psi^{-1}R_0\psi-\psi
R_0\psi^{-1}=\frac{\Phi_0}{\psi}\otimes\frac{\Phi_1}{\psi}
-\frac{\Phi_1}{\psi}\otimes\frac{\Phi_0}{\psi}
\end{equation*}
The equation relating $Q_{\pm,j}$ and $\mathcal{R}_{\pm}$ then
follows from this and (\ref{eq:upmjk}).
\end{proof}
By subtracting $1/2(H_0,(1-\chi))$ times the second column, and
$1/2(H_1,(1-\chi))$ times the third column from the first column of
the determinant $\det\left(I-(\alpha_j,\beta_k)\right)$, we see that
the determinant is the same as
\begin{equation}\label{eq:3det}
\det\left(I-(\alpha_j,\beta_k)\right)=\det\begin{pmatrix}1-\frac{\mathcal{R}_-}{2}&-(\alpha_2,\beta_1)&-(\alpha_3,\beta_1)\\
-\frac{1}{2}\mathcal{P}_{-,0}&1-(\alpha_2,\beta_2)&-(\alpha_3,\beta_2)\\
-\frac{1}{2}\mathcal{P}_{-,1}&-(\alpha_2,\beta_3)&1-(\alpha_3,\beta_3)\end{pmatrix}.
\end{equation}
By using (\ref{eq:xi1asym}), (\ref{eq:K1rank}), (\ref{eq:uij}) and
(\ref{eq:aux}), we see that the other entries $(\alpha_i,\beta_j)$
can be written as
\begin{equation}\label{eq:matrixent}
\begin{split}
\left(\alpha_2,\beta_1\right)&=\frac{T}{2}\left(q_{-,0}-q_{-,0}^{(\infty)}\right)
-\mathcal{B}_1\left(q_{-,1}-q_{-,1}^{(\infty)}\right)-
\left(q_{-,2}-q_{-,2}^{(\infty)}\right)
-\mathcal{B}_2\tilde{\mathcal{R}}_{-},\\
\left(\alpha_2,\beta_j\right)&=\frac{T}{2}u_{-,0j-2}
-\mathcal{B}_1u_{-,1j-1}- u_{-,2j-2}
-\mathcal{B}_2\tilde{\mathcal{P}}_{-,j-2},\\
\left(\alpha_3,\beta_1\right)&=\frac{1}{2}\left(q_{-,1}-q_{-,1}^{(\infty)}\right)
+\mathcal{B}_1\left(q_{-,0}-q_{-,0}^{(\infty)}\right),\\
\left(\alpha_3,\beta_j\right)&=\frac{1}{2}u_{-,1j-2}
+\mathcal{B}_1u_{-,0j-2},\quad j=2,3,
\end{split}
\end{equation}
where $q_{-,j}$ are the values of $Q_{-,j}$ at $\zeta$ and
$q_{-,j}^{(\infty)}$ are the values of $Q_{-,j}$ at $\infty$.

Let us now derive ODEs for these functions in the variable $\zeta$.
First recall the following formula that holds for arbitrary operator
$K$ that depends smoothly on a parameter $\zeta$.
\begin{equation}\label{eq:diffop}
\frac{\p}{\p \zeta}(I-K)^{-1}=(I-K)^{-1}\frac{\p K}{\p
\zeta}(I-K)^{-1}.
\end{equation}
Applying this to the operator $K_{2,airy}\chi$, we obtain, as in
\cite{TW3}, the followings.
\begin{equation}\label{eq:difker}
\begin{split}
\frac{\p}{\p
\zeta}K_{2,airy}\chi=-K_{2,airy}(\xi_1,\zeta)\delta(\zeta-\xi_2),\quad
\frac{\p}{\p
\zeta}\left(I-K_{2,airy}\chi\right)^{-1}=-R_0(\xi_1,\zeta)\rho(\zeta,\xi_2)
\end{split}
\end{equation}
where $\rho(\xi_1,\xi_2)$ is the kernel of $I+R_0\chi$, that is,
$\rho(\xi_1,\xi_2)=\delta(\xi_1-\xi_2)+R_0(\xi_1,\xi_2)\chi$. Now
from (\ref{eq:airyker}), we obtain the following.
\begin{equation*}
\left(\frac{\p}{\p
\xi_1}+\frac{\p}{\p\xi_2}\right)K_{2,airy}(\xi_1,\xi_2)\chi =
-Ai(\xi_1)Ai(\xi_2)\chi+K_{2,airy}(\xi_1,\xi_2)\delta(\xi_2-\zeta).
\end{equation*}
This then implies the following for the derivative of
$\rho(\xi_1,\xi_2)$.
\begin{equation*}
\begin{split}
\left(\frac{\p}{\p \xi_1}+\frac{\p}{\p\xi_2}\right)\rho(\xi_1,\xi_2)
&=\left(I+R_0\chi\right)\left(\left(\frac{\p}{\p
\xi_1}+\frac{\p}{\p\xi_2}\right)K_{2,airy}(\xi_1,\xi_2)\chi\right)\left(I+R_0\chi\right)\\
&=-\Phi_0(\xi_1)\Phi_0(\xi_2)\chi+R_0(\xi_1,\zeta)\rho(\xi,\xi_2),
\end{split}
\end{equation*}
where we have used the fact that the kernel of the operator
$(f\otimes g)A$ is the same as the kernel $(f\otimes A^Tg)$ and that
$\chi(I+R_0\chi)$ is the same as its transpose, together with
$(I+R_0\chi)K_{2,airy}\chi=R_0\chi$.

From this and (\ref{eq:difker}), we obtain the following for
$\rho(\xi_1,\xi_2)$ as in \cite{TW1}.
\begin{equation}\label{eq:drho}
\left(\frac{\p}{\p\zeta}+\frac{\p}{\p\xi_1}+\frac{\p}{\p\xi_2}\right)\rho(\xi_1,\xi_2)
=-\Phi_0(\xi_1)\Phi_0(\xi_2)\chi,
\end{equation}
 From (\ref{eq:drho}) and
$\rho(\xi_1,\xi_2)=\delta(\xi_1-\xi_2)+R_0(\xi_1,\xi_2)\chi$, we
obtain the following for $\xi_2\in(\zeta,\infty)$.
\begin{equation}\label{eq:dR}
\left(\frac{\p}{\p\zeta}+\frac{\p}{\p\xi_1}+\frac{\p}{\p\xi_2}\right)R_0(\xi_1,\xi_2)
=-\Phi_0(\xi_1)\Phi_0(\xi_2),\quad \xi_2\in[\zeta,\infty).
\end{equation}
We can now derive a system of ODEs satisfied by the functions in
(\ref{eq:uij}) and (\ref{eq:aux}).
\begin{lemma}\label{le:dif1}
Let $q_{\pm,j}$ and $\phi_j$ be the values of $Q_{\pm,j}$ and
$\Phi_j$ at $\xi=\zeta$ respectively and let the function
$\mathcal{Q}_j$ be $\mathcal{Q}_j=q_{+,j}-q_{-,j}$. Then functions
$u_{\pm,jk}$, $q_{\pm,j}$ and $\mathcal{Q}_j$ satisfy the following
differential equations.
\begin{equation}\label{eq:dif1}
\begin{split}
&\frac{\p
q_{+,j}}{\p\zeta}=-\frac{\phi_0}{\psi(\zeta)}u_{+,j0}+\frac{\phi_j}{\psi(\zeta)}+\frac{1}{2\psi^2(\zeta)}\mathcal{Q}_j,\\
&\frac{\p
\mathcal{Q}_{j}}{\p\zeta}=\frac{\p\log\left(\phi_0\psi^{-1}\right)}{\p\zeta}\mathcal{Q}_j
-\frac{\psi}{\phi_0}W\left(\frac{\phi_0}{\psi},\frac{\phi_1}{\psi}\right)u_{-,j0},\\
&u_{-,jk}=W\left(\mathcal{Q}_j,\phi_k\psi^{-1}\right)/W\left(\phi_0\psi^{-1},\phi_1\psi^{-1}\right)
,\quad\frac{\p u_{\pm,jk}}{\p\zeta}=-\psi^{\pm 1}q_{\pm,j}\phi_k,
\end{split}
\end{equation}
where $W(f,g)$ is the Wronskian $W(f,g)=fg^{\prime}-gf^{\prime}$.
\end{lemma}
\begin{proof}
First by using (\ref{eq:difker}) and (\ref{eq:upmjk}), one can show
as in \cite{TW3}, that, the derivatives of the $u_{\pm,jk}$ are
given by
\begin{equation}\label{eq:udif}
\begin{split}
\frac{\p
u_{\pm,jk}}{\p\zeta}&=-\left(R(\xi_1,\zeta),\psi^{\pm 1}I(H_j)\chi\right)\phi_k-\left(\psi^{\pm 1}I(H_j)\right)(\zeta)\phi_k,\\
&=-\left(\rho(\xi_1,\zeta),\psi^{\pm
1}I(H_j)\chi\right)\phi_k=-\psi^{\pm 1}q_{\pm,j}\phi_k.
\end{split}
\end{equation}
From this and (\ref{eq:vV}), we see that
\begin{equation}\label{eq:Qjeq}
\begin{split}
&\mathcal{Q}_j=\frac{\phi_0}{\psi(\zeta)}u_{-,j1}-
\frac{\phi_1}{\psi(\zeta)}u_{-,j0},\\
&\frac{\p\mathcal{Q}_j}{\p\zeta}=\left(\frac{\phi_0}{\psi(\zeta)}\right)^{\prime}u_{-,j1}-
\left(\frac{\phi_1}{\psi(\zeta)}\right)^{\prime}u_{-,j0}
\end{split}
\end{equation}
where the prime denotes derivative of $\zeta$. By eliminating
$u_{-,j1}$ from the second equation, we obtain the differential
equation for $\mathcal{Q}_j$ in (\ref{eq:dif1}). From
(\ref{eq:Qjeq}), we see that $u_{-,jk}$ are given by
\begin{equation*}
u_{-,jk}=W\left(\mathcal{Q}_j,\phi_k\psi^{-1}\right)/W\left(\phi_0\psi^{-1},\phi_1\psi^{-1}\right).
\end{equation*}

Now by using the definition of $Q_{+,j}$, we see that
\begin{equation*}
\begin{split}
\frac{\p q_{+,j}}{\p\zeta}&=\frac{1}{2\psi^2(\zeta)}q_{+,j}+
\psi^{-1}(\zeta)\int_{\mathbb{R}}\frac{\p
\rho(\zeta,\xi_2)}{\p\zeta}\psi(\xi_2)I(H_j)(\xi_2)\D \xi_2.
\end{split}
\end{equation*}
By applying (\ref{eq:drho}) to this, we obtain
\begin{equation*}
\begin{split}
&\int_{\mathbb{R}}\frac{\p
\rho(\zeta,\xi_2)}{\p\zeta}\psi(\xi_2)I(H_j)(\xi_2)\D \xi_2=-
\int_{\mathbb{R}}\frac{\p
\rho(\zeta,\xi_2)}{\p\xi_2}\psi(\xi_2)I(H_j)(\xi_2)\D \xi_2 -
\phi_0u_{+,j0}\\
&=\int_{\mathbb{R}}
\rho(\zeta,\xi_2)\frac{\p}{\p\xi_2}\left(\psi(\xi_2)I(H_j)(\xi_2)\right)\D
\xi_2- \phi_0 u_{+,j0}\\
&=-\frac{1}{2\psi(\zeta)}q_{-,j}+\phi_j-\phi_0u_{+,j0}.
\end{split}
\end{equation*}
This, together with (\ref{eq:vV}), gives us the differential
equation for $q_{+,j}$.
\end{proof}
This gives us the first set of differential equations. Let us now
derive the second set of differential equations for the functions
$\mathcal{R}_{\pm}$ and $\mathcal{P}_j$.
\begin{lemma}\label{le:seconddiff}
Let $\mathcal{R}_0$ be $\mathcal{R}_+-\mathcal{R}_-$. Then the
functions $\mathcal{R}_0$, $\mathcal{R}_+$, $\mathcal{P}_{-,j}$ and
$\mathcal{P}_{0,j}$ satisfy the following set of differential
equations.
\begin{equation}\label{eq:seconddiff}
\begin{split}
\frac{\p\mathcal{R}_0}{\p\zeta}&=\frac{\p\log\left(\phi_0\psi^{-1}\right)}{\p\zeta}\mathcal{R}_0-\frac{\psi(\zeta)}{\phi_0}W\left(\frac{\phi_0}{\psi},\frac{\phi_1}{\psi}\right)\mathcal{P}_{-,0},\\
\frac{\p\mathcal{R}_+}{\p\zeta}&=\frac{1}{2\psi^2}\mathcal{R}_0-\frac{\phi_0}{\psi(\zeta)}\mathcal{P}_{+,0},\quad
\frac{\p\mathcal{P}_{\pm,0}}{\p\zeta}=\phi_0\psi^{\pm1}(\zeta)\left(1-\mathcal{R}_{\pm}\right),\\
\mathcal{P}_{-,j}&=W\left(\mathcal{R}_0,\phi_j\psi^{-1}\right)/W\left(\phi_0\psi^{-1},\phi_1\psi^{-1}\right)
\end{split}
\end{equation}
The functions $\tilde{\mathcal{R}}_0$, $\tilde{\mathcal{R}}_+$,
$\tilde{\mathcal{P}}_{-,j}$ and $\tilde{\mathcal{P}}_{0,j}$ also
satisfy the same system of linear ODE.
\end{lemma}
\begin{proof} The proof is similar to (\ref{eq:dif1}). First, by
(\ref{eq:difker}), we can obtain the derivatives of $\Phi_k(\xi)$
with respect to $\zeta$.
\begin{equation*}
\frac{\p}{\p\zeta}\Phi_k(\xi)=-R_0\left(\xi,\zeta\right)\phi_k,
\end{equation*}
From this, we obtain the derivatives of $\mathcal{P}_{\pm,0}$.
\begin{equation*}
\begin{split}
\frac{\p\mathcal{P}_{\pm,j}}{\p\zeta}&=\phi_j\psi^{\pm1}+\int_{-\infty}^{\zeta}\frac{\p\Phi_j(\xi)}{\p\zeta}\psi^{\pm1}(\xi)\D
\xi=\phi_j\psi^{\pm1}-\phi_j\int_{-\infty}^{\zeta}R_0\left(\xi,\zeta\right)\psi^{\pm1}(\xi)\D
\xi.
\end{split}
\end{equation*}
The differential equation for $\mathcal{P}_{\pm,0}$ now follows from
the fact that $R_0(\xi_1,\xi_2)$ is symmetric with respect to the
interchange of $\xi_1$ and $\xi_2$.

From this and (\ref{eq:vV}), we again have the following equations
for $\mathcal{R}_0$.
\begin{equation*}
\begin{split}
\mathcal{R}_0&=\frac{\phi_0}{\psi(\zeta)}\mathcal{P}_{-,1}
-\frac{\phi_1}{\psi(\zeta)}\mathcal{P}_{-,0},\\
\frac{\p\mathcal{R}_0}{\p\zeta}&=\left(\frac{\phi_0}{\psi(\zeta)}\right)^{\prime}\mathcal{P}_{-,1}
-\left(\frac{\phi_1}{\psi(\zeta)}\right)^{\prime}\mathcal{P}_{-,0}.
\end{split}
\end{equation*}
This then gives us the Wroskian equations for $\mathcal{P}_{-,j}$.
Let us now derive the differential equation for $\mathcal{R}_+$.
First by using the fact that $R_0$ is symmetric with respect to the
interchange of its arguments, we can write $R_+$ as
$R_+=\psi^{-1}\int_{-\infty}^{\zeta}R_0(\xi,\zeta)\psi^{-1}(\xi)\D\xi$.
Then we have
\begin{equation*}
\begin{split}
\frac{\p\mathcal{R}_+}{\p\zeta}=\frac{1}{2\psi^2(\zeta)}\mathcal{R}_+
+\psi^{-1}(\zeta)R_0(\zeta,\zeta)\psi(\zeta)+\psi^{-1}(\zeta)\int_{-\infty}^{\zeta}\frac{\p
R_0(\xi,\zeta)}{\p\zeta}\psi(\xi)\D\xi.
\end{split}
\end{equation*}
By using (\ref{eq:dR}), the above becomes
\begin{equation*}
\begin{split}
\frac{\p\mathcal{R}_+}{\p\zeta}&=\frac{1}{2\psi^2(\zeta)}\mathcal{R}_+
-\psi^{-1}(\zeta)\int_{-\infty}^{\zeta}\frac{\p
R_0(\xi,\zeta)}{\p\xi}\psi(\xi)\D\xi-\frac{\phi_0}{\psi(\zeta)}\mathcal{P}_{+,0}\\
&=\frac{1}{2\psi^2(\zeta)}\mathcal{R}_0-\frac{\phi_0}{\psi(\zeta)}\mathcal{P}_{+,0}.
\end{split}
\end{equation*}
The same argument can be applied to the functions The functions
$\tilde{\mathcal{R}}_0$, $\tilde{\mathcal{R}}_+$,
$\tilde{\mathcal{P}}_{-,j}$ and $\tilde{\mathcal{P}}_{0,j}$ to
obtain the same system of linear ODEs.
\end{proof}
The ODEs satisfied by the various functions can be further
simplified. First, the function $\phi_0=(I-K_{2,airy}\chi)Ai(\zeta)$
is known to be the Hastings-McLeod solution of the Painlev\'e II
equation (\ref{eq:HM}) (See \cite{TWairy}).

The Wronskian of $\phi_0/\psi$ and $\phi_1/\psi$ can be written as
\begin{equation}\label{eq:wron}
W\left(\phi_0/\psi,\phi_1/\psi\right)=\psi^{-2}\left(\phi_0\phi_1^{\prime}-\phi_1\phi_0^{\prime}\right)
=\psi^{-2}R_0(\zeta,\zeta).
\end{equation}
The function $R_0(\zeta,\zeta)$, again is known to be the
logarithmic derivative of the Tracy-Widom distribution for the GUE
\cite{TWairy}. Therefore we have
\begin{equation}\label{eq:Rphi1}
R_0(\zeta,\zeta)=\int_{\zeta}^{\infty}\phi_0^2(\xi)\D
\xi=\sigma(\zeta).
\end{equation}
To obtain $\phi_1$, let us take the derivative of $\phi_0$ and use
(\ref{eq:drho}), then we have
\begin{equation*}
\begin{split}
\frac{\p\phi_0}{\p\zeta}&=\int_{\mathbb{R}}\frac{\p\rho(\zeta,\xi)}{\p\zeta}Ai(\xi)\D
\xi=-\int_{\mathbb{R}}\frac{\p\rho(\zeta,\xi)}{\p\xi}Ai(\xi)\D \xi
-\phi_0\left(\Phi_0,Ai\chi\right),\\
&=\phi_1-\phi_0\left(\Phi_0,Ai\chi\right).
\end{split}
\end{equation*}
As in the derivations of (\ref{eq:udif}), we see that
\begin{equation*}
\frac{\p}{\p\zeta}(\Phi_0,Ai\chi)=\phi_0^2,\quad
\frac{\p}{\p\zeta}(\Phi_1,Ai\chi)=\phi_0\phi_1.
\end{equation*}
Since $(\Phi_0,Ai\chi)=(\Phi_1,Ai\chi)=0$ at $\zeta=\infty$, we
obtain
\begin{equation}\label{eq:phi1}
\phi_1=\phi_0^{\prime}+\phi_0\int_{\zeta}^{\infty}\phi_0^2(\xi)\D\xi=
\phi_0^{\prime}+\sigma\phi_0
\end{equation}
To obtain $\phi_2$, we use $Ai^{\prime\prime}(\xi)=\xi Ai(\xi)$ to
obtain
\begin{equation*}
\Phi_2(\xi)=\xi\left(I+R_0\chi\right)Ai+\Phi_1(\xi)(\Phi_0,Ai\chi)-\Phi_0(\xi)(\Phi_1,Ai\chi).
\end{equation*}
Therefore we have
\begin{equation}\label{eq:phi2}
\phi_2=\left(\zeta+\int_{\zeta}^{\infty}\phi_0(\xi)\phi_1(\xi)\D\xi\right)\phi_0-\sigma\phi_1.
\end{equation}
Summarizing, we have the following.
\begin{proposition}\label{pro:difeq}
Let $U$ be the matrix
\begin{equation*}
U=\begin{pmatrix}0&0&0&-\psi\phi_0\\
0&0&\psi^{-1}\phi_0&-\psi^{-1}\phi_0\\
0&-\frac{ \sigma}{\phi_0\psi}&\frac{\p}{\p\zeta}\log\left(\phi_0/\psi\right)&0\\
-\phi_0\psi^{-1}&0&\frac{1}{2\psi^2}&0
\end{pmatrix}
\end{equation*}
and let $\vec{h}_j$ be the vector
$\vec{h}_j=\left(0,0,0,\frac{\phi_j}{\psi}\right)^T$ for $j=0,1,2$
and $\vec{h}_j=0$ for $j=3$ and $j=4$. Then the functions in
(\ref{eq:3det}) and (\ref{eq:matrixent}) are given by
\begin{equation*}
\begin{split}
&q_{-,j}-q_{-,j}^{(\infty)}=v_{j4}-v_{j3}-\int_{-\infty}^{\infty}H_j\D
u,\quad u_{-,0j}=v_{j2},\quad u_{-,jk}=\frac{\psi^2W\left(v_{j3},\phi_k\psi^{-1}\right)}{\sigma},\\
&\mathcal{R}_-=v_{34}-v_{33}+1,\quad \mathcal{P}_{-,0}=v_{32},\quad
\mathcal{P}_{-,1}=\frac{\psi^2W\left(v_{33},\phi_1\psi^{-1}\right)}{\sigma},\\
&\tilde{\mathcal{R}}_-=v_{44}-v_{43}+1,\quad
\tilde{\mathcal{P}}_{-,0}=v_{42},\quad
\tilde{\mathcal{P}}_{-,1}=\frac{\psi^2W\left(v_{43},\phi_1\psi^{-1}\right)}{\sigma},
\end{split}
\end{equation*}
where $\vec{v}_j$ is the vector that satisfies the linear system of
ODEs
\begin{equation*}
\begin{split}
&\frac{\p\vec{v}_j}{\p\zeta}=U(\zeta)\vec{v}_j+\vec{h}_j,\quad
\vec{v}_j\sim\left(0,0,0,\int_{-\infty}^{\infty}H_j\D
u\right)^T,\quad
\zeta\rightarrow+\infty,\quad j=0,1,2,\\
&\vec{v}_3\sim\left(0,0,-1,0\right)^T,\quad
\zeta\rightarrow-\infty,\quad
\vec{v}_4\sim\left(0,0,-1,0\right)^T,\quad \zeta\rightarrow+\infty
\end{split}
\end{equation*}
\end{proposition}
These functions can also be characterized using the connection
between Fredholm determinants and Riemann-Hilbert problems. We will
outline this connection in Appendix B.
\section*{Appendix: A proof of the j.p.d.f. formula using Zonal polynomials}
\renewcommand{\theequation}{A.\arabic{equation}}
\setcounter{equation}{0}

We present here a simpler algebraic proof of Theorem \ref{thm:main1}
using Zonal polynomials. Zonal polynomials are introduced by James
\cite{James} and Hua \cite{Hu} independently. They are polynomials
with matrix argument that depend on an index $p$ which is a
partition of an integer $k$. The real Zonal polynomials $Z_p(X)$
take arguments in symmetric matrices and are homogenous polynomials
in the eigenvalues of its matrix argument $X$. We shall not go into
the details of their definitions, but only state the important
properties of these polynomials that is relevant to our proof.
Readers who are interested can refer to the excellent references of
\cite{Muir}, \cite{Mac} and \cite{Ta}.

Let $p$ be a partition of an integer $k$ and let $l(p)$ be the
length of the partition. We will use $p\vdash k$ to indicate that
$p$ is a partition of $k$. Let $X$ and $Y$ be $N\times N$ symmetric
matrices and $x_i$, $y_i$ their eigenvalues. Given a partition
$p=(p_1,\ldots,p_{l(p)})$ of the integer $k$, we will order the
parts $p_i$ such that if $i<j$, then $p_i>p_j$. If we have 2
partitions $p$ and
$p^{\prime}=(p^{\prime}_1,\ldots,p_{l(p^{\prime})}^{\prime})$, then
we say that $p<p^{\prime}$ if there exists an index $j$ such that
$p_i=p_i^{\prime}$ for $i<j$ and $p_j<p_j^{\prime}$. Let the
monomial $x^p$ be $x_1^{p_1}\ldots x_{p_{l(p)}}^{p_{l(p)}}$, we say
that $x^{p^{\prime}}$ is of a higher weight than $x^p$ if
$p^{\prime}>p$. Then the Zonal polynomial $Z_{p}(X)$ is a homogenous
polynomial of degree $k$ in the eigenvalues $x_j$ with the highest
weight term being $x^p$. It has the following properties.
\begin{equation}\label{eq:zonal}
\begin{split}
&\left(\tr(X)\right)^k=\sum_{p\vdash k}Z_p(X),\\
&\int_{O(N)}e^{-M\tr\left(XgYg^T\right)}\D g
=\sum_{k=0}^{\infty}\frac{M^k}{k!}\sum_{p\vdash
k}\frac{Z_p(X)Z_p(Y)}{Z_p(I_N)}
\end{split}
\end{equation}
These properties can be found in the references \cite{Muir},
\cite{Mac} and \cite{Ta}. Another important property is the
following generating function formula for the Zonal polynomials,
which can be found in \cite{Mac} and \cite{Ta}.
\begin{equation}\label{eq:genfun}
\prod_{i,j=1}^N\left(1-2\theta
x_iy_j\right)^{-\frac{1}{2}}=\sum_{k=0}^{\infty}\frac{\theta^k}{k!}\sum_{p\vdash
k}\frac{Z_p(X)Z_p(Y)}{d_p}
\end{equation}
where $d_p$ are constants. In particular, if $(k)$ is the partition
of $k$ with length 1, that is, $(k)=(k,0,\ldots,0)$, then the
constant $d_{(k)}$ is given by
\begin{equation*}
d_{(k)}=\frac{1}{(2k-1)!!}.
\end{equation*}
For the rank 1 spiked model, let us consider the case where all but
one $y_j$ is zero and denote the non-zero eigenvalue by $y$. Then
from the fact that the highest weight term in $Z_p(Y)$ is
$y_1^{p_1}\ldots y_{p_{l(p)}}^{p_{l(p)}}$, we see that the only
non-zero $Z_p(Y)$ is $Z_{(k)}(Y)$, which by the first equation in
(\ref{eq:zonal}), is simply $y^k$. Therefore the formulae in
(\ref{eq:zonal}) and (\ref{eq:genfun}) are greatly simplified in
this case.
\begin{equation}\label{eq:gen2}
\begin{split}
&\int_{O(N)}e^{-M\tr\left(XgYg^T\right)}\D g
=\sum_{k=0}^{\infty}M^k\frac{Z_{(k)}(X)y^k}{k!Z_{(k)}(I_N)},\\
&\prod_{i=1}^N\left(1-2\theta
x_iy\right)^{-\frac{1}{2}}=\sum_{k=0}^{\infty}\theta^k\frac{(2k-1)!!Z_{(k)}(X)y^k}{k!}
\end{split}
\end{equation}
By using the generating function formula, we see that $Z_{(k)}(I_N)$
is given by
\begin{equation}\label{eq:ZI}
Z_{(k)}(I_N)=\frac{\left(N/2+k-1\right)!2^k}{(N/2-1)!(2k-1)!!}.
\end{equation}
By taking $\theta=\frac{1}{2t}$ in the second equation of
(\ref{eq:gen2}), we see that
\begin{equation*}
\prod_{i=1}^N\left(t-
x_iy\right)^{-\frac{1}{2}}=t^{-\frac{N}{2}}\sum_{k=0}^{\infty}(2t)^{-k}\frac{(2k-1)!!Z_{(k)}(X)y^k}{k!}.
\end{equation*}
We can now compute the integral
\begin{equation*}
S(t)=\int_{\Gamma}e^{Mt}\prod_{i=1}^N\left(t-
x_iy\right)^{-\frac{1}{2}}\D t
\end{equation*}
by taking residue at $\infty$, which is the $t^{-1}$ coefficient in
the following expansion
\begin{equation*}
e^{Mt}\prod_{i=1}^N\left(t-
x_iy\right)^{-\frac{1}{2}}=\sum_{k,j=0}^{\infty}\frac{M^jt^{-\frac{N}{2}+j-k}(2k-1)!!Z_{(k)}(X)y^k}{2^kj!k!}.
\end{equation*}
This coefficient is given by
\begin{equation*}
\begin{split}
S(t)&=M^{\frac{N}{2}-1}\sum_{k=0}^{\infty}\frac{Z_{(k)}(X)(2k-1)!!y^kM^k}{2^k(N/2+k-1)!k!}\\
&=\frac{M^{\frac{N}{2}-1}}{\left(N/2-1\right)!}\sum_{k=0}^{\infty}\frac{Z_{(k)}(X)y^kM^k}{Z_{(k)}(I_N)k!}=\frac{M^{\frac{N}{2}-1}}{\left(N/2-1\right)!}\int_{O(N)}e^{-M\tr\left(XgYg^T\right)}\D
g.
\end{split}
\end{equation*}
By taking $y=\frac{\tau}{2(1+\tau)}$, this proves Theorem
\ref{thm:main1}. There also exist complex and quarternionic Zonal
polynomials $C_{p}(X)$ and $Q_{p}(X)$ which satisfy the followings
instead.
\begin{equation*}
\begin{split}
&\int_{U(N)}e^{-M\tr\left(XgYg^{\dag}\right)}g^{\dag}\D g
=\sum_{k=0}^{\infty}\frac{M^k}{k!}\sum_{p\vdash
k}\frac{C_p(X)C_p(Y)}{C_p(I_N)},\\
&\int_{Sp(N)}e^{-M\mathrm{Re}\left(\tr\left(XgYg^{-1}\right)\right)}g^{-1}\D
g =\sum_{k=0}^{\infty}\frac{M^k}{k!}\sum_{p\vdash
k}\frac{Q_p(X)Q_p(Y)}{Q_p(I_N)}.
\end{split}
\end{equation*}
Their generating functions are given by
\begin{equation*}
\begin{split}
&\prod_{i,j=1}^N\left(1-2\theta
x_iy_j\right)^{-1}=\sum_{k=0}^{\infty}\frac{\theta^k}{k!}\sum_{p\vdash
k}\frac{C_p(X)C_p(Y)}{c_p},\\
&\prod_{i,j=1}^N\left(1-2\theta
x_iy_j\right)^{-2}=\sum_{k=0}^{\infty}\frac{\theta^k}{k!}\sum_{p\vdash
k}\frac{Q_p(X)Q_p(Y)}{q_p}
\end{split}
\end{equation*}
where $c_{(k)}$ and $q_{(k)}$ are
\begin{equation*}
\begin{split}
c_{(k)}=\frac{1}{2^kk!},\quad q_{(k)}=\frac{1}{(k+1)!2^k}
\end{split}
\end{equation*}
Then by following the same argument as in the real case, we can
write down the following integral formulae for rank one
perturbations of the complex and quarternionic cases.
\begin{equation}\label{eq:compsym}
\begin{split}
&\int_{U(N)}e^{-M\tr\left(XgYg^{\dag}\right)}g^{\dag}\D g
=\frac{(N-1)!}{M^{N-1}}\int_{\Gamma}e^{Mt}\prod_{i=1}^N\left(t-
x_iy\right)^{-1}\D t,\\
&\int_{Sp(N)}e^{-M\mathrm{Re}\left(\tr\left(XgYg^{-1}\right)\right)}g^{-1}\D
g =\frac{(2N-1)!}{M^{2N-1}}\int_{\Gamma}e^{Mt}\prod_{i=1}^N\left(t-
x_iy\right)^{-2}\D t.
\end{split}
\end{equation}
\section*{Appendix B: Connection to Riemann-Hilbert problem}
\renewcommand{\theequation}{A.\arabic{equation}}
\setcounter{equation}{0} The functions $\Phi_{j}$ involved in
(\ref{eq:aux}) can also be found by solving a Riemann-Hilbert
problem (See e.g. \cite{BD}, \cite{CIK}, \cite{D}, \cite{DIK},
\cite{DIZ}). Let $K$ be an integrable operator acting on
$\mathbb{R}$ of the form (\ref{eq:intop}) and $f_i$ be the column
vector with entries $f_{ij}$. Then operator $K\chi$ is also
integrable. Its resolvent can be found by solving the following
Riemann-Hilbert problem.
\begin{equation}\label{eq:RHPres}
\begin{aligned}
1. \quad  & \text{$Y(\xi)$ is analytic in $\mathbb{C}\setminus [\zeta,\infty)$},&&\\
2. \quad & Y_{+}(\xi)=Y_{-}(\xi)\left(I-2\pi if_{1}(\xi)f_{2}^T(\xi)\right),\quad \xi\in(\zeta,\infty),\\
3. \quad  & Y(\xi)=I+O(\xi^{-1}), \quad   \xi\rightarrow\infty,\\
4. \quad &Y(\xi)=O(\log|\xi-\zeta|),\quad \xi\rightarrow \zeta.
\end{aligned}
\end{equation}
Then the vectors $F_{i}$ with entries $F_{ij}$ in
(\ref{eq:resokern}) are given by $Y_+(\xi)f_{ij}$.

This Riemann-Hilbert problem can be connected to the Riemann-Hilbert
problem of the Painlev\'e II equation. We will now outline this
connection. For more details, please see \cite{CIK}. Let $A(\xi)$ be
the matrix in Lemma \ref{le:Vair}. Multiplying the solution $Y$ of
(\ref{eq:RHPres}) by $A$ on the right and then deform the regions
suitably will transform the Riemann-Hilbert problem
(\ref{eq:RHPres}) into the following Riemann-Hilbert problem.
\begin{equation}\label{eq:RHPresX}
\begin{aligned}
1. \quad  & \text{$X(\xi)$ is analytic in $\mathbb{C}\setminus \Sigma$},&&\\
2. \quad & X_{+}(\xi)=X_{-}(\xi)J_X,\quad \xi\in\Sigma,\\
3. \quad  & X(\xi)=\frac{1}{\sqrt{2}}\xi^{-\frac{\sigma_3}{4}}\begin{pmatrix}1&1\\
-1&1\end{pmatrix}e^{-\frac{\pi i}{4}\sigma_3}\left(I+O(\xi^{-\frac{3}{2}})\right)e^{-\frac{2}{3}\xi^{\frac{3}{2}}\sigma_3}, \quad   \xi\rightarrow\infty,\\
4. \quad &X(\xi)=O(\log|\xi-\zeta|),\quad \xi\rightarrow \zeta.
\end{aligned}
\end{equation}
where the contour $\Sigma$ is the union of
\begin{equation*}
\Sigma_1=(-\infty,\zeta],\quad \Sigma_2=\zeta+e^{\frac{2\pi
i}{3}}\mathbb{R}_+,\quad \Sigma_3=\zeta+e^{-\frac{2\pi
i}{3}}\mathbb{R}_+.
\end{equation*}
The contours in $\Sigma$ are all pointing towards $\zeta$. The jump
matrix $J_X$ is given by
\begin{equation*}
J_X=\begin{pmatrix}1&0\\
1&1\end{pmatrix},\quad \xi\in\Sigma_2\cup\Sigma_3,\quad J_X=\begin{pmatrix}0&1\\
-1&0\end{pmatrix},\quad \xi\in\Sigma_1.
\end{equation*}
Since $X=YA$, we have the following.
\begin{equation}\label{eq:X1X2}
X_{11}=\sqrt{2\pi}e^{-\frac{i\pi}{4}}(I-K_{2,airy}\chi)^{-1}Ai,\quad
X_{21}=\sqrt{2\pi}e^{-\frac{i\pi}{4}}(I-K_{2,airy}\chi)^{-1}Ai^{\prime}.
\end{equation}
Note the difference between our Riemann-Hilbert problem
(\ref{eq:RHPresX}) and the one given in (2.11-15) of \cite{CIK} . In
(\ref{eq:RHPresX}), as $\xi\rightarrow\infty$, the next to the
leading order term is of order $\xi^{-1}$ smaller than the leading
order term. This is due to the asymptotic behavior of the Airy
functions in the matrix $A$. As a result, our Riemann-Hilbert
problem (\ref{eq:RHPresX}) will be uniquely solvable. This is
important for us to keep track of the functions $\Phi_0$ and
$\Phi_1$ that appears in the resolvent. In \cite{CIK}, it was shown
that the Riemann-Hilbert problem (\ref{eq:RHPresX}) can be solved by
using the monodromy problem associating with the Painlev\'e II
equation. Let $\Sigma_{\Psi}$ be the union of the contours
\begin{equation*}
\Sigma_{\Psi,1}=e^{\frac{\pi i}{6}}\mathbb{R}_+,\quad
\Sigma_{\Psi,2}=e^{\frac{5\pi i}{6}}\mathbb{R}_+,\quad
\Sigma_{\Psi,3}=e^{\frac{7\pi i}{6}}\mathbb{R}_+,\quad
\Sigma_{\Psi,4}=e^{-\frac{\pi i}{6}}\mathbb{R}_+,
\end{equation*}
and let $S_{\pm}$, $S_1$ and $S_2$ be the regions
\begin{equation*}
S_+=\left\{\xi|\quad
-\frac{\pi}{6}<\arg\xi<\frac{\pi}{6}\right\},\quad S_-=-S_+,\quad
S_1=\left\{\xi|\quad
\frac{\pi}{6}<\arg\xi<\frac{5\pi}{6}\right\},\quad S_2=-S_1.
\end{equation*}
Then $\Psi(\xi,v)$ is the solution to the following monodromy
problem of the Painlev\'e II equation.
\begin{equation}\label{eq:RHPresPsi}
\begin{aligned}
1. \quad  & \text{$\Psi(\xi,v)$ is analytic in $\mathbb{C}\setminus \Sigma_{\Psi}$},&&\\
2. \quad & \Psi_{+}(\xi,v)=\Psi_{-}(\xi,v)J_{\Psi},\quad \xi\in\Sigma_{\Psi},\\
3. \quad  & \Psi(\xi,v)=\left(I+\frac{\Psi_{\infty}}{\xi}+O(\xi^{-2})\right)e^{-i\left(\frac{4}{3}\xi^3+v\xi\right)\sigma_3}, \quad   \xi\rightarrow\infty,\\
4. \quad &\Psi(\xi,v)=\left\{
                      \begin{array}{ll}
                        O\left(|\xi|^{\frac{\sigma_3}{2}}\right)\begin{pmatrix}1&0\\
\pm i&1\end{pmatrix}, & \hbox{$\xi\rightarrow 0$ in $S_{\pm}$, ;} \\
                        O\left(|\xi|^{\frac{\sigma_3}{2}}\right), & \hbox{$\xi\rightarrow 0$ in $S_1$;} \\
                        O\left(|\xi|^{-\frac{\sigma_3}{2}}\right), & \hbox{$\xi\rightarrow 0$ in $S_2$.}
                      \end{array}
                    \right.
\end{aligned}
\end{equation}
where $J_{\Psi}$ is given by
\begin{equation*}
J_{\Psi}=\begin{pmatrix}1&0\\
-i&1\end{pmatrix},\quad \xi\in
\Sigma_{\Psi,1}\cup\Sigma_{\Psi,2},\quad J_{\Psi}=\begin{pmatrix}1&-i\\
0&1\end{pmatrix},\quad \xi\in \Sigma_{\Psi,3}\cup\Sigma_{\Psi,4},
\end{equation*}
where the contours are all pointing towards infinity. The matrix
$\Psi$ also satisfies the Lax equation for the Painlev\'e II
equation.
\begin{equation}\label{eq:lax}
\begin{split}
\frac{\p}{\p v}\Psi&=\begin{pmatrix}-i\xi&q\\
q&i\xi\end{pmatrix}\Psi,\\
\xi\frac{\p}{\p \xi}\Psi&=\begin{pmatrix}-4i\xi^3-(2iq+v)\xi&4q\xi^2+2iq_v\xi-\frac{1}{2}\\
4q\xi^2-2iq_v\xi-\frac{1}{2}&4i\xi^3+(2iq+v)\xi\end{pmatrix}\Psi,
\end{split}
\end{equation}
where $q(v)$ is related to the Hastings-McLeod solution of
Painlev\'e II by (see (1.47) of \cite{CIK})
\begin{equation}\label{eq:qq0}
2^{-\frac{4}{3}}(v+2q^2(v)+2q_v(v))=\phi_0^2\left(2^{\frac{1}{3}}v\right).
\end{equation}
By using the Lax equations (\ref{eq:lax}) and the behavior of $\Psi$
at $\xi\rightarrow\infty$, one can check that the next to the
leading term $\Psi_{\infty}$ in the expansion in
(\ref{eq:RHPresPsi}) is given by
\begin{equation}\label{eq:Psiinf}
\Psi_{\infty}=\begin{pmatrix}-\frac{iq}{2}&\frac{q}{2i}\\
-\frac{q}{2i}&\frac{iq}{2}\end{pmatrix}
\end{equation}
The authors in \cite{CIK} then showed that the solution to the
Riemann-Hilbert problem (\ref{eq:RHPresX}) is given by
\begin{equation*}
X(\xi,\zeta)=\frac{1}{\sqrt{2}}\begin{pmatrix}1&0\\
h&1\end{pmatrix}\left(\xi-\zeta\right)^{-\frac{1}{4}\sigma_3}\begin{pmatrix}1&1\\
-1&1\end{pmatrix}\Psi\left(4^{-\frac{1}{6}}i\left(\xi-\zeta\right)^{\frac{1}{2}},-2^{\frac{1}{3}}\zeta\right)e^{-\frac{i\pi}{4}\sigma_3}
\end{equation*}
for some $h$ independent on $\xi$. The function $h(\zeta)$ is
determined by making sure that the next to the leading order term in
$X$ is of order indicated by (\ref{eq:RHPresX}). By using
(\ref{eq:Psiinf}), one can check that the appropriate choice of $h$
in this case is given by $h=-q$. The properties of the
Riemann-Hilbert problem in (\ref{eq:RHPresPsi}) is studied very
thoroughly in \cite{CIK}, and these properties will hopefully be
useful in the characterization of the functions in (\ref{eq:3det})
and (\ref{eq:matrixent}).

\vspace{.25cm}

\noindent\rule{16.2cm}{.5pt}

\vspace{.25cm}

{\small

\noindent {\sl School of Mathematics \\
                       University of Bristol\\
                       Bristol BS8 1TW, UK  \\
                       Email: {\tt m.mo@bristol.ac.uk}

                       \vspace{.25cm}

                       \noindent  26 January  2011}}

\end{document}